\documentclass[11pt,preprint]{imsart}

\RequirePackage[numbers]{natbib}
\RequirePackage[colorlinks,citecolor=blue,urlcolor=blue]{hyperref}
\usepackage{amsmath,amsmath,amssymb,amsfonts}
\usepackage{amsthm}
\usepackage[american]{babel}
\usepackage{graphicx}
\usepackage{flafter}
\usepackage[section]{placeins}
\usepackage{float}
\usepackage{makecell}
\usepackage{comment}
\usepackage[mathscr]{eucal}
\usepackage{bbm}
\usepackage{todonotes}
\usepackage{enumitem}
\presetkeys{todonotes}{color=red!30}{linecolor=black!50}
\startlocaldefs

\def\eps{{\varepsilon}}

\def\Bbb E{\mathbb{E}}
\def\Bbb R{\mathbb{R}}
\newtheorem{definition}{Definition}[section]
\newtheorem{example}{Example}[section]
\newtheorem{corollary}{Corollary}[section]
\usepackage{mathtools}
\makeatletter \@addtoreset{equation}{section}

\makeatother

\newtheorem{lemma}{Lemma}[section]
\newtheorem{theorem}{Theorem}[section]
\newtheorem{proposition}{Proposition}[section]

\newtheorem{remark}{Remark}[section]

\font\tencmmib=cmmib10 \skewchar\tencmmib '60
\newfam\cmmibfam
\textfont\cmmibfam=\tencmmib

\font\tenmsb=msbm10 
\def\Bbb#1{\hbox{\tenmsb#1}}

\def\lessim{\ \lower4pt\hbox{$
\buildrel{\displaystyle <}\over\sim$}\ }
\def\gessim{\ \lower4pt\hbox{$\buildrel{\displaystyle >}
\over\sim$}\ }

\def\eps{\varepsilon}

\def\go0{\to 0}

\def\leftitem#1{\item{\hbox to\parindent{\enspace#1\hfill}}}

\def\sg{\sigma}

\def\sg2{\sigma^2}

\def\__{_{\infty}}

\numberwithin{equation}{section} 

\newcommand{\1}{{\rm 1}\kern-0.24em{\rm I}}

\newtheorem{assumption}{Assumption}[section]

\endlocaldefs

\begin{document}

\begin{frontmatter}
\title{Functional estimation in high-dimensional and infinite-dimensional models}
\runtitle{Functional estimation}

\begin{aug}
\author{\fnms{Vladimir} \snm{Koltchinskii}\thanksref{t1}\ead[label=e1]{vlad@math.gatech.edu}} 
and 
\author{\fnms{Minghao} \snm{Li}\ead[label=e2]{minghaoli@gatech.edu}}
\thankstext{t1}{Supported in part by NSF grant DMS-2113121}
\runauthor{Koltchinskii and Li}


\address{School of Mathematics\\
Georgia Institute of Technology\\
Atlanta, GA 30332-0160, USA\\
\printead{e1}\\
\printead{e2}
}
\end{aug}
\vspace{0.2cm}
{\small \today}
\vspace{0.2cm}

\begin{abstract}
Let ${\mathcal P}$ be a family of probability measures on a measurable space $(S,{\mathcal A}).$ Given a Banach space $E,$ 
a functional $f:E\mapsto {\mathbb R}$ and a mapping $\theta: {\mathcal P}\mapsto E,$
our goal is to estimate $f(\theta(P))$ based on i.i.d. observations $X_1,\dots, X_n\sim P, P\in {\mathcal P}.$
In particular, if ${\mathcal P}=\{P_{\theta}: \theta\in \Theta\}$ is an identifiable statistical model with parameter set 
$\Theta\subset E,$ one can consider the mapping $\theta(P)=\theta$ for $P\in {\mathcal P}, P=P_{\theta},$
resulting in a problem of estimation of $f(\theta)$ based on i.i.d. observations $X_1,\dots, X_n\sim P_{\theta}, \theta\in \Theta.$
Given a smooth functional $f$ and estimators $\hat \theta_n(X_1,\dots, X_n), n\geq 1$ of $\theta(P),$ we use these estimators, the sample split and the Taylor expansion of $f(\theta(P))$ of a proper order to construct estimators $T_f(X_1,\dots, X_n)$ of $f(\theta(P)).$ 
For these estimators and for a functional $f$ of smoothness $s\geq 1$ we prove the bound 
\begin{align*}
    &\left\|T_f(X_1,\cdots,X_n)-f(\theta(P))\right\|_{L_p(\mathbb{P}_P)}
    \lesssim_s \|f'\|_{C^{s-1}}\Bigl[\sqrt{\frac{a_p(P)}{n}} +
   \Bigl(\sqrt{\frac{d_{ps}(P)}{n}}\Bigr)^{s}\Bigr],
\end{align*}
provided that $d_{ps}(P)\lesssim n,$ where, for $p\geq 1$ and $P\in {\mathcal P},$ 
 \begin{align*}
 a_p(P):=\sup_{n\geq 1}\sup_{\|u\|\leq 1} n \Bigl\|\langle \hat \theta_n-\theta(P), u\rangle\Bigr\|_{L_p({\mathbb P}_{P})}^2\ \ 
 {\rm and}
 \ \ 
 d_p(P):= \sup_{n\geq 1} n \Bigl\|\|\hat \theta_n-\theta(P)\|\Bigr\|_{L_{p}({\mathbb P}_{P})}^2.
 \end{align*}
In addition, under the assumption that the linear forms $\langle \hat \theta_n-\theta(P), u\rangle, u\in E^{\ast}$ could be approximated in Wasserstein $W_p$-distances by normal r.v., we establish normal approximation in the same 
distances of $\sqrt{n}(T_f(X_1,\dots, X_n)-f(\theta(P))).$ We study the performance of estimators $T_f(X_1,\dots, X_n)$ in several concrete problems,
showing their minimax optimality and asymptotic efficiency.
In particular, this includes functional estimation in high-dimensional models with many low dimensional components, functional estimation in high-dimensional exponential families and estimation of functionals of covariance operators in infinite-dimensional subgaussian models. 
\end{abstract}



\end{frontmatter}

\section{Introduction}

Let $(S, {\mathcal A})$ be a measurable space and let ${\mathcal P}$ be a family of probability distributions on $(S,{\mathcal A}).$
Let $E$ be a Banach space with the dual space $E^{\ast}$ and let $\theta: {\mathcal P}\mapsto E.$ Consider 
a smooth functional $f:E\mapsto {\mathbb R}.$ In what follows, we assume that $f\in C^s(E)$ for some $s=m+\rho,$ 
$m\geq 0, \rho\in (0,1],$ where $C^s(E)$ is the space of functionals of H\"older smoothness $s.$ Our goal is to estimate $f(\theta(P))$ based on i.i.d. observations $X_1,\dots, X_n\sim P,$
where $P\in {\mathcal P}.$ 
In particular, it is of interest to consider the case of statistical model ${\mathcal P}:= \{P_{\theta}: \theta\in \Theta\}$
with parameter space $\Theta \subset E.$ If $\Theta\ni \theta\mapsto P_{\theta}\in {\mathcal P}$ is an identifiable model (that is, $\theta_1\neq \theta_2$ implies that $P_{\theta_1}\neq P_{\theta_2}$), then, for any $P\in {\mathcal P},$ there exists a unique $\theta\in \Theta$ such that $P=P_{\theta},$ and one can set $\theta(P)=\theta(P_{\theta}):=\theta,$ defining the mapping $\theta: {\mathcal P}\mapsto \Theta\subset E.$  In this case, the problem 
can be rephrased as estimation of the value $f(\theta)$ of functional $f$ based on i.i.d. observations $X_1,\dots, X_n\sim P_{\theta}, \theta \in \Theta.$   

More precisely, the goal is, for a given class $\mathcal F$ of functionals $f: E\mapsto {\mathbb R}$ (such as, for instance, the H\"older ball
$\{f: \|f\|_{C^s(E)}\leq 1\}$ of certain degree of smoothness $s>0$) to determine the size of the maximal minimax risk 
\begin{align*}
\delta_n({\mathcal P}, {\mathcal F}):=\sup_{f\in {\mathcal F}}\inf_{T_n} \sup_{P\in {\mathcal P}} {\mathbb E}_P (T_n(X_1,\dots, X_n)-f(\theta(P)))^2
\end{align*}
of estimation of functionals from the class ${\mathcal F}$ and to develop estimators $T_{n,f}(X_1,\dots, X_n)$ for which the risk of the order 
$\delta_n({\mathcal P}, {\mathcal F})$ is attained. It is also of interest to study the conditions on the complexity of statistical model ${\mathcal P}$
and functional class ${\mathcal F}$ (for instance, the degree of H\"older smoothness $s$ of the functionals) under which the parametric rate $\delta_n({\mathcal P}, {\mathcal F})=O(n^{-1})$ of functional estimation becomes possible.

In this estimation problem, the functional $f: E\mapsto {\mathbb R}$ is given to a statistician in the sense that its value 
$f(\theta)$ can be computed at any given point $\theta \in E.$ Sometimes, it could be also assumed that the values of the derivatives $f^{(j)}, j=1,\dots, m$ 
at given points in $E$ are also available to a statistician. Thus, the problem becomes, given i.i.d. observations $X_1,\dots, X_n\sim P,$ 
to choose a finite number of data dependent points 
in $E$ and to use the values of the functional (and, possibly, its derivatives) at these points to construct an estimator of $f(\theta(P)).$

A naive approach to the problem is to choose a single data dependent point in $E$ and to use the value of the functional at this point as an estimator.  
In other words, this approach is based on using a plug-in estimator $f(\hat \theta_n)$ for a suitable estimator $\hat \theta_n$ of $\theta(P)$
based on $X_1,\dots, X_n$ i.i.d. $\sim P.$ In the classical case of regular statistical models $\{P_{\theta}:\theta\in \Theta\}$ with open parameter set 
$\Theta\subset {\mathbb R}^d$ and with non-singular Fisher information $I(\theta),$ plugging in the maximum likelihood estimator $\hat \theta_n$ in a continuously differentiable functional $f:\Theta\mapsto {\mathbb R}$ does yield an asymptotically optimal solution of the problem. Namely, under proper regularity conditions, the plug-in estimator $f(\hat \theta_n)$ is asymptotically normal with $\sqrt{n}$-rate 
\begin{align*}
\sqrt{n}(f(\hat \theta_n)-f(\theta)) \overset{d}{\to} N(0, \sigma_f^2(\theta))\ {\rm as}\ n\to\infty
\end{align*}
and $\sigma_f^2(\theta):= \langle I(\theta_0)^{-1}f'(\theta), f'(\theta)\rangle$ is the limit of properly normalized mean squared error:
\begin{align*}
n{\mathbb E}_{\theta}(f(\hat \theta_n)-f(\theta))^2 \to \sigma_f^2(\theta)\ {\rm as}\ n\to\infty.
\end{align*}
Moreover, the above limit relationships usually hold uniformly (over the whole parameter space $\Theta,$ or over its proper subsets). 
The optimality of plug-in estimator $f(\hat \theta_n)$ follows from the local asymptotic minimax theorem of H\'ajek and LeCam: for all $\theta_0\in \Theta,$
\begin{align*}
\lim_{c\to +\infty}\liminf_{n\to\infty} \inf_{T_n} \sup_{\|\theta-\theta_0\|\leq cn^{-1/2}}n{\mathbb E}_{\theta}(T_n(X_1,\dots, X_n)-f(\theta))^2 \geq \sigma_f^2(\theta_0),
\end{align*}
where the infimum is taken over all estimators $T_n(X_1,\dots, X_n)$ based on i.i.d. observations $X_1,\dots, X_n\sim P_{\theta}.$

In the case of high-dimensional and infinite-dimensional models, the problem is considerably harder: the optimal error rates are not necessarily $n^{-1/2}$
and they are not necessarily attained for plug-in estimators.  
The plug-in method could still yield optimal error rates for functionals of smoothness $s\leq 2.$ 
In the case of functionals of higher order of smoothness $s>2,$ better convergence rates become possible, but they are usually not attained for plug-in 
estimators due to their large bias, and the development of non-trivial bias reduction methods is an important part of the problem.


This problem of functional estimation in high-dimensional and infinite-dimensional models
has a long history that goes back  to the 1970s (see \cite{Levit_1, Levit_2, Ibragimov, Ibragimov_Khasm_Nemirov, Bickel_Ritov, Nemirovski_1990, Birge, Laurent, Lepski, Nemirovski, Robins, C_C_Tsybakov} and references therein). Many of these authors studied the problem for special functionals (including 
quadratic functionals, special classes of integral functionals, norms in functional spaces, etc) and for special statistical models (such as Gaussian white noise model, density estimation model, etc) and obtained highly non-trivial results in these special cases. Ibragimov, Nemirovski and Khasminskii  \cite{Ibragimov_Khasm_Nemirov} followed by Nemirovski \cite{Nemirovski_1990, Nemirovski} studied the problem of estimation of  general smooth functionals of a 
signal, represented by a function from a compact subset of $L_2([0,1])$ and observed in a Gaussian white noise, and showed that efficient estimation of such functionals with classical error rates is possible if the degree of smoothness of the functional is above certain threshold and impossible for some smooth functionals if the degree of smoothness is below the threshold. The value of the threshold depends on the complexity of the underlying compact space of signals
characterized by the rate of decay of its Kolmogorov's widths. We follow this approach in the current paper trying to study how the optimal error rates of estimation 
of smooth functionals depend on their degree of smoothness as well as on complexity of the parameter and the sample size, and also trying to determine 
the smoothness thresholds for which the ``phase transition" from slow error rates to classical parametric $n^{-1/2}$ rates occurs.

General methods of bias reduction in functional estimation problems (namely, iterated bootstrap, jackknife and Taylor expansions methods) have been recently discussed in \cite{Jiao} in the case of estimation of a smooth function of the parameter of binomial model. It was shown that,
even in the case of this simple classical  one-dimensional model, the analysis of the bias reduction methods is often related to non-trivial problems in approximation theory. 
We will briefly discuss these approaches to bias reduction in the case of parametric model $X_1,\dots, X_n$ i.i.d. $\sim P_{\theta}, \theta \in \Theta.$

Iterated bootstrap approach goes back to the paper \cite{Hall_1} by Hall and Martin. Given a reasonable estimator $\hat \theta\in \Theta$ of the parameter $\theta$
itself, the first iteration of this method is the plug-in estimator $f(\hat \theta).$  Its bias is $({\mathcal B} f)(\theta):= {\mathbb E}_{\theta} f(\hat \theta)-f(\theta)$ 
and, to reduce the bias, one can use the plug-in estimator $({\mathcal B}f)(\hat \theta)$ of the bias and to subtract it from $f(\hat \theta).$ This yields the estimator 
$f_1(\hat \theta):= f(\hat \theta)-({\mathcal B} f)(\hat \theta).$ The bias of estimator $({\mathcal B} f)(\hat \theta)$ of $({\mathcal B} f)(\theta)$ is equal to 
$({\mathcal B}^{2} f)(\theta)= ({\mathcal B} f)(\hat \theta)-({\mathcal B} f)(\theta).$ One can use the plug-in estimator $({\mathcal B}^2 f)(\hat \theta)$ of this 
bias and add it to the estimator $f_1(\hat \theta),$ yielding the next estimator 
$f_2(\hat \theta):= f(\hat \theta)-({\mathcal B} f)(\hat \theta)+({\mathcal B}^2 f)(\hat \theta),$ and so on. 

Another way to look at this method is to consider the following operator $({\mathcal T} g)(\hat \theta) := {\mathbb E}_{\theta} g(\hat \theta), \theta\in \Theta$
acting in the space of bounded functions on $\Theta.$ Then, to find a function $g$ for which the bias of estimator $g(\hat \theta)$ of $f(\theta)$ is small means 
solving approximately the operator equation $(Tg)(\theta)=f(\theta), \theta \in \Theta$ (or, in other words approximating ${\mathcal T}^{-1} f$). 
If ${\mathcal B}:= {\mathcal T}-{\mathcal I}$ is a ``small" operator, one can use the partial sums of Neumann series $({\mathcal I}+{\mathcal B})^{-1} = {\mathcal I}-{\mathcal B}+{\mathcal B}^2-\dots$ to approximate ${\mathcal T}^{-1}.$ This yields the following approximate solution of equation 
$({\mathcal T} g)(\theta) = f(\theta), \theta \in \Theta:$
\begin{align*}
f_k(\theta) := \sum_{j=0}^k (-1)^k ({\mathcal B}^j f)(\theta), \theta \in \Theta
\end{align*}
and the corresponding estimator $f_k(\hat \theta).$
The bias of this estimator is 
\begin{align*}
{\mathbb E}_{\theta} f_k(\hat \theta)-f(\theta) = (-1)^k ({\mathcal B}^{k+1}f)(\theta),
\end{align*}
so, to justify this bias reduction method one has to prove that the function $({\mathcal B}^{k+1}f)(\theta), \theta\in \Theta$ is 
sufficiently small. 

Note that operator ${\mathcal T}$ could be viewed as an integral operator 
\begin{align*}
({\mathcal T} g)(\theta) = \int_{\Theta} g(t) P(\theta, dt), \theta \in \Theta
\end{align*}
with respect to Markov kernel $P(\theta, A):= {\mathbb P}_{\theta}\{\hat \theta\in A\}, \theta\in \Theta, A\subset \Theta.$
If $\hat \theta^{(k)}, k\geq 0$ is a Markov chain with $\hat \theta^{(0)}=\theta$ and with transition kernel $P(\theta, A),$ then 
it is easy to check that 
\begin{align}
\label{finite_differ_B^k}
({\mathcal B}^k f)(\theta)= {\mathbb E}_{\theta} \sum_{j=0}^k (-1)^{k-j} {k\choose j} f(\hat \theta^{(j)}).
\end{align}
Note that Markov chain $\hat \theta^{(k)}, k\geq 0$ can be interpreted as an iterative application of bootstrap resampling 
to the parametric model $\{P_{\theta}: \theta\in \Theta\}:$ for $k=0,$ $\hat \theta^{(0)}$ is the value of the unknown parameter $\theta;$
for $k=1,$ $\hat \theta^{(1)}=\hat \theta$ is the estimator $\hat \theta$ based on $n$ i.i.d. observation sampled from distribution $P_{\theta};$
for $k=2,$ $\hat \theta^{(2)}$ is the bootstrap estimator based on $n$ i.i.d. observations sampled from distribution $P_{\hat \theta}$ (conditionally 
on $\hat \theta$); for $k=3,$ $\hat \theta^{(3)}$ is the next iteration of bootstrap estimator based on $n$ i.i.d. observations sampled from distribution $P_{\hat \theta^{(2)}}$ (conditionally on $\hat \theta^{(2)}$), etc. Although for a given statistical model $P_{\theta}, \theta\in \Theta,$ functionals 
${\mathcal B}^j f$ are known in principle, it could be hard to compute their values $({\mathcal B}^j f)(\hat \theta)$ and, thus, also the value of estimator 
$f_k(\hat \theta).$ However, formula \eqref{finite_differ_B^k} and interpretation of $\{\hat \theta^{(k)}: k\geq 0\}$ as a {\it bootstrap chain} allows one to compute 
$({\mathcal B}^j f)(\hat \theta)$ using Monte Carlo simulation (which was the initial idea of iterated bootstrap bias reduction in \cite{Hall_1}).
In addition, formula \eqref{finite_differ_B^k} provides a representation of $({\mathcal B}^k f)(\theta)$ as an expected value of the $k$-th order 
difference of functional $f$ along the bootstrap chain. Such representations could be used to obtain bounds on functionals ${\mathcal B}^k f$ 
and their derivatives, which, in turn, are used to bound the bias of estimator $f_k(\hat \theta)$ and also to study its concentration properties.  

In high-dimensional and infinite-dimensional setting, the study of this bias reduction method was initiated in \cite{Koltchinskii_2017, Koltchinskii_2018} for functionals of unknown covariance of high-dimensional normal models and continued in \cite{Koltchinskii_Zhilova} for Gaussian shift models, in \cite{Koltchinskii_Zhilova_21} for more general random shift models, in \cite{Koltchinskii_Zhilova_19} for high-dimensional normal models with unknown mean and covariance, in \cite{Koltchinskii_2021} for general parametric models admitting base estimators with strong Gaussian approximation properties 
and in \cite{Koltchinskii_Wahl} for high-dimensional log-concave location families with MLE as a base estimator. In all these examples, 
the mean squared error of estimator $f_k(\hat \theta)$ (with properly chosen $k$ depending on the smoothness of the functional) was of the order 
$
\frac{1}{n}\vee \Bigl(\frac{d}{n}\Bigr)^s,
$
with $d$ being the dimension of the model and $s$ the degree of smoothness of functional $f,$ and this rate happened to be optimal.
The proof of these results required the development of rather involved analytic and probabilistic methods. This includes representation of bootstrap 
chains as superpositions of i.i.d. random maps (random homotopies) and using such representations along with Fa\`a di Bruno calculus 
to obtain more explicit analytic formulas for functionals $({\mathcal B}^k f)(\theta)$ as well as the development of concentration inequalities 
for $f(\hat \theta)$ around its expectation for smooth functionals $\theta.$ Although this approach has been already developed for a number 
of interesting models, there are serious difficulties with its extension to other models, including, for instance Gaussian models with ``non-spherical" 
covariances with their complexity characterized by effective rank.  

Another interesting approach to bias reduction in functional estimation is based on linear aggregation of several plug-in estimators with different 
sample sizes with a goal to almost cancel out the biases of these plug-in estimators, see \cite{Jiao, Koltchinskii_2022}. Namely, under certain assumptions on 
base estimators $\hat \theta_n$ (that hold, in particular, for sample means in linear normed spaces), it is possible, for a functional $f$ of smoothness $s=m+\rho,$
$m=2,3,\dots$ and $\rho\in (0,1],$ to choose coefficients $C_1,\dots, C_m$ and sample sizes $n/c \leq n_1,\dots, n_m\leq n$ in such a way that the bias 
of combined estimator 
\begin{align*}
T_f(X_1,\dots, X_n):=\sum_{j=1}^m C_j f(\hat \theta_{n_j})
\end{align*}
would be of the order $\Bigl(\sqrt{\frac{d}{n}}\Bigr)^s.$ Moreover, the coefficients $C_j, j=1,\dots, m$ also satisfy the conditions $\sum_{j=1}^m C_j=1$
and $\sum_{j=1}^m |C_j|\lesssim 1.$ If, in addition, reasonable concentration inequalities hold for plug-in estimator $f(\hat \theta_n)$ around its 
expectation such that $f(\hat \theta_n)-{\mathbb E}_{\theta} f(\hat \theta_n)$ is of the order $\frac{1}{\sqrt{n}},$ it becomes also possible to show 
that similar concentration holds for estimator $T_f(X_1,\dots, X_n)$ and its mean squared error is of the order 
$
\frac{1}{n}\vee \Bigl(\frac{d}{n}\Bigr)^s,
$
where $d$ is the dimension or other suitable complexity parameter. Moreover, it also makes sense to use a symmetrized 
version of estimator $T_f(X_1,\dots, X_n),$ which is a linear combination of $U$-statistics $\sum_{j=1}^m C_j U_n f(\hat \theta_{n_j})$
(a jackknife bias reduction method discussed in \cite{Jiao}). Such estimators enjoy the same properties as estimator $T_f(X_1,\dots, X_n),$
but, in addition, it is possible to prove their asymptotic normality and asymptotic efficiency. 
In \cite{Koltchinskii_2022}, this approach to bias reduction and functional estimation was studied in the case of estimation of functionals of unknown covariance operator of centered
Gaussian r.v. in a separable Banach space and it was shown that, for such estimators, minimax optimal error rates are attained in this problem. So far, it is not known how to obtain 
similar results for estimators based on bootstrap chain bias reduction. 

Bias reduction in functional estimation could be also based on a Taylor expansion of functional $f$ of smoothness $s=m+\rho,$ $m\geq 0, \rho\in (0,1]$ in a neighborhood of the value 
of a base estimator $\hat \theta$ 
\begin{align*}
f(\theta) = \sum_{k=0}^m \frac{f^{(k)}(\hat \theta)[\theta-\hat \theta, \dots, \theta-\hat \theta]}{k!} + R_m
\end{align*}
with the following bound on the remainder $R_m:$
\begin{align*}
|R_m|\leq \|f^{(m)}\|_{{\rm Lip}_{\rho}} \|\hat \theta-\theta\|^s.
\end{align*}
The problem then is reduced to estimation of the polynomial
\begin{align*}
p(\theta) :=\sum_{k=0}^m \frac{f^{(k)}(\hat \theta)[\theta-\hat \theta, \dots, \theta-\hat \theta]}{k!},
\end{align*}
for which it is often possible to construct an unbiased estimator. 
Usually, this approach requires 
splitting the sample into two disjoint parts, one of them being used for the base estimator 
$\hat \theta$ and another (independent) part for estimating the polynomial $p(\theta).$
Unbiased estimation of polynomial could be often achieved utilizing $U$-statistics. 
If it is possible to construct an unbiased estimator $T_f(X_1,\dots, X_n)$ of $p(\theta)$ 
and prove that the mean squared error of this estimator (conditionally on $\hat \theta$)
is of the order $\frac{1}{n},$ then, to achieve the error $\frac{1}{n}\vee \Bigl(\frac{d}{n}\Bigr)^s$ 
of estimation of the value of functional $f(\theta),$ it is enough to show that ${\mathbb E}_{\theta}\|\hat \theta-\theta\|^s$
is of the order $\Bigl(\sqrt{\frac{d}{n}}\Bigr)^s.$ Note that, for this approach, one has to compute not only the values of the functional 
$f$ at data dependent points, but also the values of its derivatives. 

Taylor expansion bias reduction method was used in  \cite{Ibragimov_Khasm_Nemirov, Nemirovski} in the case of Gaussian white noise model,
where Hermite polynomials were used to construct unbiased estimator of $p(\theta)$ (in this case, it was possible to avoid the sample split). 
Similar approaches were used in \cite{Robins}, where the authors developed the technique of higher order influence functions and $U$-statistics 
in functional estimation, and, more recently, in \cite{Fan}. 

Note that one of the difficulties in the analysis of functional estimation methods based on iterated bootstrap and on linear aggregation of plug-in estimators (including a jackknife version)
as well as some versions of Taylor expansion methods is the need to develop concentration inequalities for $f(\hat \theta)-{\mathbb E}_{\theta}f(\hat \theta)$ for rather general functionals $f.$
Such inequalities are not always available which limits the theory to classes of models for which concentration methods are well developed. 

Our goal in this paper is to suggest a simple version of functional estimators based on Taylor expansions for which the use of concentration inequalities 
could be avoided and replaced by rather simple moment bounds. We will also provide several applications of this estimation method, including estimation of functionals in high-dimensional models with independent low-dimensional components, estimation of functionals of covariance operators in infinite-dimensional subgaussian 
models, estimation of functionals in high-dimensional exponential families and estimation of functionals of probability density.

Before proceeding to the statements of the main results, let us introduce some notations used throughout the paper.

In what follows, we use generic notation $\|\cdot\|$ for norms of all Banach spaces. Its meaning depends on the context: for instance, the norms of $x\in E$ and $u\in E^{\ast}$
will be denoted by $\|x\|$ and $\|u\|$ without providing subscripts. The subscripts will be used only in the case of some ambiguity, for instance, for the norms of function 
spaces. The ball with center $x\in E$ and radius $\delta>0$ will be denoted by $B(x,\delta).$ 
For a set $A\subset E,$ ${\rm l.s.}(A)$ denotes the linear span of set $A.$
We will use the inner product notation for the values of linear functionals: for $x\in E, u\in E^{\ast},$ $\langle x,u\rangle$ denotes the value of linear functional $u$  
 on vector $x.$

We often have to deal with finite-dimensional linear normed spaces $E.$ In this case, ${\rm dim}(E^{\ast})={\rm dim}(E).$  
For instance, $E$ could be the linear space of $d\times d$ matrices equipped with the operator norm. Then $E^{\ast}$ is 
the same linear space equipped with the nuclear norm.  Let $e_1,\dots, e_N$ be a basis of $E$ (as a linear space) and let $f_1,\dots, f_N$ be the bi-orthogonal basis of $E^{\ast}:$
\begin{align*} 
\langle e_i, f_j\rangle = \delta_{ij}, i,j=1,\dots, N.  
\end{align*} 
Then, the coordinates of $y\in E$ in the basis $e_1,\dots, e_N$ are $y_j= \langle y,f_j\rangle$ and the coordinates of $u\in E^{\ast}$ in the basis 
$f_1,\dots, f_N$ are $u_i:= \langle e_i, u\rangle.$ This allows one to identify spaces $E$ and $E^{\ast}$ with the copies of ${\mathbb R}^N$ equipped 
with the corresponding norms and we also have  $\langle y,u\rangle= \sum_{j=1}^N y_j u_j.$

Given Banach spaces $E,F,$ a subset $U\subset E$ and a function $g:U\mapsto F,$ denote 
\begin{align*}
\|g\|_{L_{\infty}(U)} := \sup_{x\in U}\|g(x)\|\ \ {\rm and}\ \|g\|_{{\rm Lip}(U)}:= \sup_{x,x'\in U, x\neq x'}\frac{\|g(x)-g(x')\|}{\|x-x'\|}.
\end{align*} 
Also, given $\rho\in (0,1],$ denote 
\begin{align*}
\|g\|_{{\rm Lip}_{\rho}(U)}:= \sup_{x,x'\in U, x\neq x'}\frac{\|g(x)-g(x')\|}{\|x-x'\|^{\rho}}.
\end{align*}
Clearly, for $\rho=1,$ $\|g\|_{{\rm Lip}_{\rho}(U)}=\|g\|_{{\rm Lip}(U)}.$

We will now use the $L_{\infty}$- and the ${\rm Lip}_{\rho}$-norms to define H\"older $C^{s}$-norms for $s=m+\rho,$ $m\geq 0, \rho\in (0,1].$ 
Let $U\subset E$ be an open subset and let $g:U\mapsto F$ be an $m$ times Fr\'echet continuously differentiable function.
Note that its $j$-th order Fr\'echet derivative is a symmetric $j$-linear form on $E$ with values in $F.$ For such a form $M[x_1,\dots, x_j],$ 
define its operator norm as 
\begin{align*}
\|M\|:= \sup_{\|x_1\|\leq 1, \dots, \|x_j\|\leq 1} \|M[x_1,\dots, x_j]\|.
\end{align*}
Such norms will be used for the derivatives $g^{(j)}(x), x\in U$ and, based on this, we can define $\|g^{(j}\|_{L_{\infty}(U)},$ $\|g^{(j)}\|_{{\rm Lip}_{\rho}(U)}.$ Note that other choices of the norms of the derivatives $g^{(j)}(x)$ would lead to different smoothness classes (even in the case of a finite-dimensional space $E,$ different norms on multi-linear forms are equivalent with dimension dependent constants). 
With this in mind, denote
\begin{align*}
\|g\|_{C^s(U)}:= \max\Bigl(\|g\|_{L_{\infty}(U)}, \dots, \|g^{(m)}\|_{L_{\infty}(U)},\|g^{(m)}\|_{{\rm Lip}_{\rho}(U)}\Bigr).
\end{align*}
If $U=E,$ we will write $\|\cdot \|_{L_{\infty}}, \|\cdot\|_{{\rm Lip}}, \|\cdot\|_{{\rm Lip}_{\rho}}$ and $\|\cdot\|_{C^s}$ instead of 
 $\|\cdot \|_{L_{\infty}(E)}, \|\cdot\|_{{\rm Lip}(E)}, \|\cdot\|_{{\rm Lip}_{\rho}(E)}$ and $\|\cdot\|_{C^s(E)}.$

For two real variables $A, B>0,$ the notation $A\lesssim B$ means that $A\leq CB$ for a numerical constant $C>0.$ 
We write $A\gtrsim B$ if $B\lesssim A$ and $A\asymp B$ if $A\lesssim B$ and $B\lesssim A.$
The signs $\lesssim , \gtrsim, \asymp$ could be provided with subscripts when the constants in the corresponding 
relationships depend on certain parameters. Say, $A\lesssim_{\gamma} B$ means that there exists a constant $C_{\gamma}>0$
depending only on $\gamma$ such that $A\leq C_{\gamma} B.$

Occasionally, we will use Orlicz norms of random variables. Let $\psi: {\mathbb R}_+\mapsto {\mathbb R}_+$ be a convex increasing 
function with $\psi(0)=0.$ For a r.v. $\xi$ on a probability space $(\Omega, \Sigma, {\mathbb P}),$ define the Orlicz $\psi$-norm of $\xi$
as 
\begin{align*}
\|\xi\|_{\psi} := \inf\Bigl\{c>0: {\mathbb E}\psi\Bigl(\frac{|\xi|}{c}\Bigr)\leq 1\Bigr\}.
\end{align*}
Denote $L_{\psi}({\mathbb P}):=\{\xi: \|\xi\|_{\psi}<\infty\}.$ It is well known that $(L_{\psi}({\mathbb P}),\|\cdot\|_{\psi})$ is a Banach space.
Sometimes, we will write $\|\xi\|_{L_{\psi}({\mathbb P})}$ instead of $\|\xi\|_{\psi}$ to emphasize the dependence on the norm on probability measure ${\mathbb P}.$
For $p\geq 1,$ the $L_p$-norm is a special choice of Orlicz $\psi$-norm with $\psi(u)=u^p, u\geq 0.$ Another well known choice is 
$\psi_{\alpha}(u)=e^{u^{\alpha}}-1, u\geq 0$ for $\alpha \geq 1,$ in particular the $\psi_2$-norm for subgaussian r.v. and the $\psi_1$-norm 
for subexponential r.v. Orlicz $\psi_{\alpha}$-norms are equivalent to certain norms defined in terms of the $L_p$-norms for all $p\geq 1:$
\begin{align}
\label{psi_alpha_L_p}
\|\xi\|_{\psi_{\alpha}} \asymp \sup_{p\geq 1} p^{-1/\alpha} \|\xi\|_{L_p}.
\end{align}
Note that, for $\alpha<1,$ $\psi_{\alpha}$ is not a convex function and $\|\cdot\|_{\psi_{\alpha}}$ is not a norm. 
However, the right hand side of \eqref{psi_alpha_L_p} is a norm and it could be used as the definition of $\psi_{\alpha}$-norm 
for $\alpha<1.$

The paper is organized as follows. In sections \ref{Main_results}-\ref{exponential_section}, we state the general results of the paper and discuss their applications in three concrete examples. 
In the following sections \ref{technical}-\ref{sec:lower_bounds}, we give more details on the results and provide their proofs.  
More specifically, in Section \ref{Main_results}, we describe the construction of estimators of functionals $f(\theta(P))$ 
based on Taylor expansion and sample split and study their properties under moment assumptions on base estimators of parameter $\theta(P).$
In particular, we provide upper bounds on the estimation errors as well as the results on the normal approximation of the estimators. 
We also discuss minimax lower bounds implying the optimality of the error rates as well as the asymptotic efficiency of the estimators
in specific examples.  
The proofs of the results on the upper bounds are provided in Section \ref{technical}. Some additional results for base estimators satisfying Bernstein type inequalities are discussed in Section \ref{functionals_Bernstein}. In Section \ref{HD_comps}, we study applications of the general results to statistical models with a large number of independent low-dimensional components; the proofs of the corresponding upper bounds are provided in Section \ref{sec:HDLD}.  
In Section \ref{func_cov_op}, we discuss the problem of estimation of functionals of covariance operators in infinite-dimensional subgaussian models 
(the complexity of the model being characterized by the effective rank) with the proofs given in Section \ref{func_cov_proofs}. In Section \ref{exponential_section}, we consider estimation of functionals of high-dimensional 
parameters of exponential families. More details on exponential families and the proofs of the results are provided in Section \ref{exp_fam_det}. Minimax lower bounds in functional estimation 
(stated in sections \ref{Main_results} and \ref{HD_comps}) are proved in Section \ref{sec:lower_bounds}.

\section{Main results}
\label{Main_results}


In what follows, it will be assumed that there exist estimators $\hat \theta_n=\hat \theta_n(X_1,\dots, X_n), n\geq 1$ 
of parameter $\theta(P)$ such that certain bounds on the linear functionals 
$\langle \hat \theta_n-\theta(P), u\rangle, u\in E^{\ast}$ and 
on the norms $\|\hat \theta_n-\theta(P)\|$ hold. Estimator $\hat \theta_n$ will be called 
{\it the base estimator} and such estimators along with the sample split will be used to 
construct estimators of the value $f(\theta(P))$ of functional $f.$

\begin{definition}
\label{assume_AAA}
\normalfont
Let $p\geq 1.$ Define
\normalfont
\item 
 \begin{align*}
 a_p(P):=\sup_{n\geq 1}\sup_{\|u\|\leq 1} n \Bigl\|\langle \hat \theta_n-\theta(P), u\rangle\Bigr\|_{L_p({\mathbb P}_{P})}^2,\ P \in {\mathcal P}
 \end{align*}
 and
 \begin{align*}
 d_p(P):= \sup_{n\geq 1} n \Bigl\|\|\hat \theta_n-\theta(P)\|\Bigr\|_{L_{p}({\mathbb P}_{P})}^2,\ P\in {\mathcal P}.
 \end{align*}
 \end{definition}

Note that, if $a_p(P)<\infty,$ then, for all $n\geq 1,$ 
\begin{align*}
\sup_{\|u\|\leq 1} \Bigl\|\langle \hat \theta_n-\theta(P), u\rangle\Bigr\|_{L_p({\mathbb P}_{P})}\leq \sqrt{\frac{a_p(P)}{n}}
\end{align*}
and, if $d_p(P)<\infty,$ then, for all $n\geq 1,$
\begin{align*}
\Bigl\|\|\hat \theta_n-\theta(P)\|\Bigr\|_{L_{p}({\mathbb P}_{P})}\leq \sqrt{\frac{d_p(P)}{n}}.
\end{align*}
Clearly, $a_p(P)\leq d_p(P), p\geq 1, P\in {\mathcal P}.$ Also, $a_p(P)$ and $d_p(P)$ are non-decreasing functions with respect to $p\geq 1.$

The quantity $a_p(P)$ characterizes the accuracy of estimation 
of linear functionals of $\theta(P),$ whereas the quantity $d_p(P)$ is usually related to some ``dimension" or ``complexity" 
of parameter $\theta=\theta(P)$ and it characterizes the error of estimation of $\theta$ in the norm of Banach space $E.$
If ${\mathcal P}:=\{P_{\theta}:\theta\in \Theta\}$ is an identifiable model and $\theta(P_{\theta}):=\theta,$ then we will use notations 
$a_p(\theta):=a_p(P_{\theta})$ and $d_p(\theta):=d_p(P_{\theta}).$ 

\begin{example}
\label{ex_Euc}
\normalfont 
Let $E={\mathbb R}^d$ be equipped with the standard Euclidean norm.
Then, by an elementary argument,
\begin{align*}
\Bigl\|\|\hat \theta_n-\theta(P)\|\Bigr\|_{L_{p}({\mathbb P}_{P})} \leq \sqrt{\frac{a_p(P)d}{n}},
\end{align*}
which implies that $d_p(P)\leq a_p(P)d.$
\end{example}

\begin{example}
\normalfont
Let $T:S\mapsto E$ satisfy, for some $p\geq 2$ and for all $P\in {\mathcal P},$ the condition 
\begin{align*}
\sup_{\|u\|\leq 1}{\mathbb E}_P |\langle T(X),u\rangle|^p<\infty. 
\end{align*}
Define $\theta(P):= {\mathbb E}_P T(X)$ and 
\begin{align*}
\hat \theta_n  := \frac{T(X_1)+\dots+T(X_n)}{n}.
\end{align*}
Using well known bounds on the $L_p$-norms of sums of independent r.v. (see, e.g., Th. 1.5.11 in \cite{Pena_Gine}),
it is easy to check that 
\begin{align*}
&
\|\langle \hat \theta_n -\theta(P), u\rangle\|_{L_p} 
\lesssim \frac{p}{\log p} \frac{\|\langle T(X)-{\mathbb E}_P T(X), u\rangle\|_{L_p({\mathbb P}_P)}}{\sqrt{n}}. 
\end{align*}
Therefore, 
\begin{align*}
a_p(P)\lesssim \Bigl(\frac{p}{\log p}\Bigr)^2 \sup_{\|u\|\leq 1}\|\langle T(X)-{\mathbb E}_P T(X), u\rangle\|_{L_p({\mathbb P}_P)}^2.
\end{align*}
\end{example}

The next results for plug-in estimator $f(\hat \theta_n)$ of $f(\theta(P))$ for $f\in C^s(E), s\leq 2$ immediately follow from H\"older condition on function $f,$ bound on the remainder of the first order Taylor expansion and Assumption \ref{assume_AAA}. 

\begin{proposition}
\label{prop_s<2}
(i) Let $f\in C^s(E)$ for some $s\leq 1.$ If, for some $P\in {\mathcal P}$ and $p\geq 1,$ $d_{ps\vee 1}(P)<\infty,$  then 
\begin{align*}
\|f(\hat \theta_n)-f(\theta(P))\|_{L_p({\mathbb P}_P)}\leq \|f\|_{C^s} \Bigl(\sqrt{\frac{d_{ps\vee 1}(P)}{n}}\Bigr)^s.
\end{align*} 
(ii) Let $f\in C^s(E)$ for some $s=1+\rho$ for some $\rho \in (0,1].$ If, for some $P\in {\mathcal P}$ and $p\geq 1,$ $d_{ps}(P)<\infty,$ then 
\begin{align*}
\|f(\hat \theta_n)-f(\theta(P))\|_{L_p({\mathbb P}_P)}\leq \|f'(\theta(P))\| \sqrt{\frac{a_p(P)}{n}}+\|f'\|_{{\rm Lip}_{\rho}} \Bigl(\sqrt{\frac{d_{ps}(P)}{n}}\Bigr)^s,
\end{align*} 
and, moreover,
\begin{align*}
\|f(\hat \theta_n)-f(\theta(P))- \langle \hat \theta_n-\theta(P), f'(\theta(P))\rangle\|_{L_p({\mathbb P}_P)} \leq \|f'\|_{{\rm Lip}_{\rho}} \Bigl(\sqrt{\frac{d_{ps}(P)}{n}}\Bigr)^s.\end{align*}  
\end{proposition}

Note that, for $s<1,$ the error rate of plug-in estimator $f(\hat \theta_n)$ is of the order $O(n^{-s/2})$ when the dimension $d_p(P)$ is bounded 
by a constant and it is even slower when $d_{ps\vee 1}(P)$ is allowed to grow with $n.$ Thus, in this case, the error rate is slower than $n^{-1/2},$
and it becomes of the order $n^{-1/2}$ for bounded $d_{ps\vee 1}(P)$ only when $s=1.$ If $s\in (1,2]$ and $d_{ps}(P)\leq n^{\alpha}$ for some 
$\alpha\in (0,1/2],$
then the error rate is of the order $O(n^{-1/2})$ for $s\geq \frac{1}{1-\alpha},$ and, at least according to the upper bounds of Proposition \ref{prop_s<2}, 
the error rate of plug-in estimator $f(\hat \theta_n)$ would be slower than $n^{-1/2}$ if $d_{ps}(P)\geq n^{\alpha}$ for some $\alpha>1/2.$

It turns out that the error rate $\Bigl(\frac{1}{\sqrt{n}} + \Bigl(\sqrt{\frac{d}{n}}\Bigr)^s\Bigr)\wedge 1$ is minimax optimal {\it for all $s>0$} at least in some 
important instances of $d$-dimensional problems. Namely, the following proposition holds. It is a corollary of slightly more general Proposition \ref{max_min_max_many} in the next section.

\begin{proposition}
\label{sup_min_max_bd}
Let $\{P_{\theta}: \theta\in \Theta\}, \Theta\subset {\mathbb R}^d$ be a statistical model, let $\theta_0\in \Theta$ and suppose that for some $\rho>0,$
$B_{\ell_{\infty}}(\theta_0,\frac{\rho}{\sqrt{n}}):=\{\theta: \|\theta-\theta_0\|_{\ell_{\infty}}\leq \frac{\rho}{\sqrt{n}}\}\subset \Theta.$ Suppose also that, for all $\theta\in B_{\ell_{\infty}}(\theta_0,\frac{\rho}{\sqrt{n}}),$
\begin{align*}
K(P_{\theta}\|P_{\theta_0}) \leq C^2\|\theta-\theta_0\|^2
\end{align*}
with some numerical constant $C>0.$ Finally, suppose that $\rho\leq \frac{\gamma}{C}$ for a sufficiently small numerical constant $\gamma>0.$
Then, for all $s>0,$ 
\begin{align*}
\sup_{\|f\|_{C^s}\leq 1}\inf_{T_n} \sup_{\theta\in B_{\ell_{\infty}}(\theta_0,\frac{\rho}{\sqrt{n}})} {\mathbb E}_{\theta} (T_n(X_1,\dots, X_n)-f(\theta))^2 
\gtrsim \Bigl(\frac{\rho^2}{n} + \Bigl(\rho^2\frac{d}{n}\Bigr)^s\Bigr)\wedge 1,
\end{align*}
where the infimum is taken over all estimators $T_n(X_1,\dots, X_n)$ of $f(\theta)$ based on i.i.d. $X_1,\dots, X_n\sim P_{\theta}.$
\end{proposition}

By Proposition \ref{prop_s<2} for $s\leq 2,$ the lower bound of Proposition \ref{sup_min_max_bd} is attained for plug-in estimators based on any estimator of the parameter $\theta$ satisfying 
the assumption $d_{ps\vee 1}(\theta)\lesssim d$. In what follows, our goal is to develop (under the assumption that $d_{ps}(P)\lesssim d$) estimators of $f(\theta(P))$ for 
which the same error rate is attained for $s>2.$

The estimators discussed in this paper are based on the Taylor expansion of $f(\theta)$ in a neighborhood of a given estimator $\hat \theta^{(0)}$ of 
$\theta=\theta(P)$ based on $X_1,\dots, X_n$ i.i.d. $\sim P\in {\mathcal P}.$
For the functional $f$ of smoothness $s=m+\rho,$ $m\geq 2, \rho\in (0,1],$ we use the Taylor expansion of order $m$: 
\begin{align*}
f(\theta)= f(\hat \theta^{(0)}) +\sum_{k=1}^{m} \frac{f^{(k)}(\hat \theta^{(0)})[\theta-\hat \theta^{(0)},\dots, \theta-\hat \theta^{(0)}]}{k!} + R_m
\end{align*}
with the remainder $R_m$ satisfying the bound 
\begin{align*}
|R_m| \leq \frac{\|f^{(m)}\|_{{\rm Lip}_{\rho}}}{m!} \|\hat \theta^{(0)}-\theta\|^s.
\end{align*}
Then, the goal is to develop estimators of polynomials 
\begin{align}
\label{Taylor_mult_form}
f^{(k)}(\hat \theta^{(0)})[\theta-\hat \theta^{(0)},\dots, \theta-\hat \theta^{(0)}], k=1,\dots, m
\end{align}
involved in the Taylor expansion of $f(\theta).$ If we are in a possession of estimators $\hat \theta_j^{(k)}, j=1,\dots, k$ 
of parameter $\theta$ that are {\it independent r.v.} conditionally on $\hat \theta^{(0)},$ it becomes possible to use 
\begin{align}
\label{est_mult_lin}
f^{(k)}(\hat \theta^{(0)})[\hat \theta_1^{(k)}-\hat \theta^{(0)},\dots, \hat \theta_k^{(k)}-\hat \theta^{(0)}], k=1,\dots, m
\end{align}
as estimators of \eqref{Taylor_mult_form}.
Note that, conditionally on $\hat \theta^{(0)},$ such estimators are unbiased provided that $\hat \theta_j^{(k)}, j=1,\dots, k$ are unbiased estimators of $\theta.$ This approach yields the following estimator of $f(\theta(P)):$
\begin{align}
\label{basic_T_f}
T_f(X_1,\dots, X_n) := \sum_{k=0}^m \frac{f^{(k)}(\hat \theta^{(0)})[\hat \theta_1^{(k)}-\hat \theta^{(0)},\dots, \hat \theta_k^{(k)}-\hat \theta^{(0)}]}{k!}.
\end{align}
We will also use a truncated version of estimator $T_f(X_1,\dots, X_n):$
for a constant $M>0,$ define 
\begin{align}
\label{basic_T_fM}
\tilde{T}_f(X_1,\cdots,X_n):= \tilde{T}_{f, M}(X_1,\cdots,X_n):=
\begin{cases}
M&T_f(X_1,\cdots,X_n)>M,\\
T_f(X_1,\cdots,X_n)&T_f(X_1,\cdots,X_n)\in[-M,M],\\
-M&T_f(X_1,\cdots,X_n)<-M.
\end{cases}
\end{align}
The constant $M$ will depend on $f$ and, usually, it will satisfy the condition $M\geq \|f\|_{L_{\infty}}.$


In what follows, it will be assumed that, for all $k=1,\dots, m,$ the estimators $\hat \theta^{(0)}, \hat \theta_j^{(k)}, j=1,\dots, k$ are independent r.v.  
A standard way to construct such estimators is to start with base estimators $\hat \theta_n(X_1,\dots, X_n), n\geq 1$ of parameter $\theta=\theta(P)$ 
and to define estimators $\hat \theta^{(0)}, \hat \theta_j^{(k)}, j=1,\dots, k, k=1,\dots, m$ using a sample split. Namely, for $J\subset \{1,\dots, n\},$ denote $X_J:= (X_j: j\in J).$
Let $J_0 \subset \{1,\dots, n\}$ with ${\rm card}(J_0) =: n^{(0)} \leq n-m$ and let $\hat \theta^{(0)}:= \hat \theta_{n^{(0)}}(X_{J_0}).$  
For $k=1,\dots, m,$ let $J_1^{(k)}, \dots, J^{(k)}_k$ be a partition of $\{1,\dots, n\}\setminus J_0$ into $k$ disjoint non-empty 
subsets. Let $n_j^{(k)}:= {\rm card}(J_j^{(k)})$ and $\hat \theta^{(k)}_j:= \hat \theta_{n_j^{(k)}}(X_{J^{(k)}_j}),$ where 
$j=1,\dots, k, k=1,\dots, m.$ 
Let ${\mathcal J}:= \{J_0, J_j^{(k)}: j=1,\dots, k, k=1,\dots, m\}$ and 
define $T_f(X_1,\dots, X_n) := T_{f, {\mathcal J}}(X_1,\dots, X_n)$ by \eqref{basic_T_f} and 
$\tilde{T}_f(X_1,\cdots,X_n):= \tilde{T}_{f, M}(X_1,\cdots,X_n):= \tilde{T}_{f, M, {\mathcal J}}(X_1,\cdots,X_n)$
by \eqref{basic_T_fM}.

For estimator $T_f(X_1,\dots, X_n),$ the following result will be proved.

\begin{theorem}
\label{th_1_asssume_AAA}
Let $s:=m+\rho$ for some $m\geq 2$ and $\rho\in (0,1]$ and let $f'\in C^{s-1}(E).$
Assume that, for some $P\in {\mathcal P}$ and for some $p\geq 1,$ $d_{ps}(P)<\infty.$

(i) Suppose that $n^{(0)}\asymp_m n$ and $n_j^{(k)}\asymp_m n$, $j=1,\cdots,k$, $k=1,\cdots,m.$
Then 
\begin{align*}
    &\left\|T_f(X_1,\cdots,X_n)-f(\theta(P))\right\|_{L_p(\mathbb{P}_P)}\\
    &\lesssim_m \max_{1\leq k\leq m} \|f^{(k)}\|_{L_\infty}\sqrt{\frac{a_p(P)}{n}} \Bigl(\sqrt{\frac{d_{ps}(P)}{n}}\Bigr)^{k-1}+\|f^{(m)}\|_{{\rm Lip}_\rho}
   \Bigl(\sqrt{\frac{d_{ps}(P)}{n}}\Bigr)^{s} .
\end{align*}

(ii) Suppose that, for all $j=1,\dots, k, k=1,\dots, m,$ $n_j^{(k)}\asymp_m n.$ Then 
\begin{align*}
&
\Bigl\|T_f(X_1,\dots, X_n) - f(\theta(P))- \langle \hat \theta^{(1)}_1-\theta(P), f'(\theta(P))\rangle\Bigr\|_{L_p({\mathbb P}_{P})} 
\\
&
\lesssim_m
\max_{2\leq k\leq m}\|f^{(k)}\|_{L_{\infty}} \sqrt{\frac{a_p(P)}{n}} 
\Bigl(\sqrt{\frac{d_{ps}(P)}{n^{(0)}}}\Bigr)^{k-1} 
+ \|f^{(m)}\|_{{\rm Lip}_{\rho}}\Bigl(\sqrt{\frac{d_{ps}(P)}{n^{(0)}}}\Bigr)^{s}.
\end{align*}
\end{theorem}

\vskip 3mm

The proof of this result immediately follows from more general and more technical statements given and proved in Section \ref{technical}. 
It relies on the following bound
\begin{align*}
&\left\|\left(f^{(k)}(t)[\hat{\theta}_1^{(k)}-t,\cdots,\hat{\theta}_k^{(k)}-t]-f^{(k)}(t)[\theta-t,\cdots,\theta-t]\right)
\right\|_{L_p}
    \leq k \|f^{(k)}(t)\| A_k (B_k+\|t-\theta\|)^{k-1},
\end{align*}
where
\begin{align*}
A_k:=
\max_{1\leq j\leq k}
\sup_{\|u\|\le 1}\|\langle \hat \theta_j^{(k)}-\theta,u\rangle\|_{L_p}
\ {\rm and}\ 
B_k:=\max_{1\leq j\leq k}
\|\|\hat \theta_j^{(k)}-\theta\|\|_{L_p},
\end{align*}
which itself follows from simple bounds on the multilinear forms of independent r.v. in a Banach space
(see lemmas \ref{Prop_1A} and \ref{Prop_2A}). Such bounds could be then used (conditionally on $\hat \theta^{(0)}$)
to control the $L_p$-norms of multilinear forms 
$f^{(k)}(\hat \theta^{(0)})[\hat \theta_1^{(k)}-\hat \theta^{(0)},\dots, \hat \theta_k^{(k)}-\hat \theta^{(0)}]$ 
involved in the definition of estimator $T_f(X_1,\dots, X_n).$

For truncated estimators $\tilde{T}_{f, M}(X_1,\cdots,X_n),$ we will prove the following local version of Theorem \ref{th_1_asssume_AAA}.

\begin{theorem}
\label{th_1_AA_asssume_AAA}
Let $f:E\mapsto {\mathbb R}$ be a uniformly bounded functional and let $\theta=\theta(P)$ for some $P\in {\mathcal P}.$
Suppose $f$ is $m$ times Fr\`echet continuously differentiable in $U:=B(\theta;\delta)$ for some $\delta\in (0,1]$ and, moreover,
$f^{(m)}$ satisfies the Hölder condition with exponent $\rho\in(0,1]$ in $U.$ 
Suppose also that, for some $P\in {\mathcal P}$ and for some $p\geq 1,$ $d_{ps}(P)<\infty.$

(i) Suppose that $n^{(0)}\asymp_m n$ and $n_j^{(k)}\asymp_m n$, $j=1,\cdots,k$, $k=1,\cdots,m.$ For $M\geq \|f\|_{L_{\infty}},$ 
\begin{align*}
    &\left\|\tilde{T}_f(X_1,\cdots,X_n)-f(\theta(P))\right\|_{L_p(\mathbb{P}_P)}\\
    &\lesssim_m \max_{1\leq k\leq m} \|f^{(k)}\|_{L_\infty(U)} \sqrt{\frac{a_p(P)}{n}}  
    +(M+\|f\|_{L_\infty}+\|f^{(m)}\|_{{\rm Lip}_\rho(U)} )\Bigl(\frac{1}{\delta}\sqrt{\frac{d_{ps}(P)}{n}}\Bigr)^{s}.
\end{align*}

(ii) Suppose that, for all $j=1,\dots, k, k=1,\dots, m,$ $n_j^{(k)}\asymp_m n.$
If $d_{ps}(P)\lesssim n^{(0)}$ and $M\geq \|f\|_{\infty}+ \|f\|_{{\rm Lip}}\delta,$ then 
\begin{align*}
    &\left\|\tilde T_f(X_1,\cdots,X_n)-f(\theta)- \langle \hat \theta_1^{(1)}-\theta(P), f'(\theta(P))\rangle\right\|_{L_p(\mathbb{P}_P)}
    \\
    &
    \lesssim_m \max_{2\leq k\leq m} \|f^{(k)}\|_{L_\infty(U)} \sqrt{\frac{a_p(P)}{n}}  
    +(M+\|f\|_{L_\infty}+\|f^{(m)}\|_{{\rm Lip}_\rho(U)})\Bigl(\frac{1}{\delta}\sqrt{\frac{d_{ps}(P)}{n^{(0)}}}\Bigr)^{s}.    
    \end{align*}
\end{theorem}

\vskip 3mm

Some other results in the same direction as theorems \ref{th_1_asssume_AAA} and \ref{th_1_AA_asssume_AAA} will be given 
in Section \ref{functionals_Bernstein} under the assumptions that $\langle \hat \theta_n-\theta(P), u\rangle$ and $\|\hat \theta_n-\theta(P)\|$
satisfy Bernstein type inequalities. They are used to study estimation of functionals of covariance in subgaussian models as well as functional 
estimation in high-dimensional exponential families. 

\vskip 2mm

In what follows, we discuss some corollaries of Theorem \ref{th_1_asssume_AAA}. Similar corollaries also hold for Theorem \ref{th_1_AA_asssume_AAA}.

For numbers $a>0$ and $d>0,$ denote 
$${\mathcal P}_{p,s} (a,d):= \{P\in {\mathcal P}: a_p(P)\leq a, d_{ps}(P)\leq a d\}.$$
The next corollary is immediate. 

\begin{corollary}
\label{cor_lin_appr}
(i) Suppose assumptions of Theorem \ref{th_1_asssume_AAA} (i) hold. 
If $d\lesssim n^{\alpha}$ for some $\alpha\in (0,1),$ then 
\begin{align*}
&
\sup_{P \in {\mathcal P}_{p,s}(a,d)}\Bigl\|T_f(X_1,\dots, X_n) - f(\theta(P))\Bigr\|_{L_p({\mathbb P}_{P})} 
\\
&
\lesssim_s \max_{1\leq k\leq m}\frac{\|f^{(k)}\|_{L_{\infty}}}{k!} a^{k/2} n^{-1/2} + \|f^{(m)}\|_{{\rm Lip}_{\rho}} a^{s/2} n^{-s(1-\alpha)/2}.
\end{align*}

(ii) Suppose assumptions of Theorem \ref{th_1_asssume_AAA} (ii) hold and, in addition, $n^{(0)}\asymp \frac{n}{\log n}.$
If $d\lesssim n^{\alpha}$ for some $\alpha\in (0,1),$ then 
\begin{align*}
&
\sup_{P \in {\mathcal P}_{p,s}(a,d)}\Bigl\|T_f(X_1,\dots, X_n) - f(\theta(P))-\langle \hat \theta^{(1)}_1-\theta(P), f'(\theta(P))\rangle\Bigr\|_{L_p({\mathbb P}_{P})} 
\\
&
\lesssim_s 
\max_{2\leq k\leq m} \|f^{(k)}\|_{L_{\infty}} a^{k/2}\sqrt{\frac{1}{n}} n^{-(1-\alpha)/2}\sqrt{\log n} 
+ \|f^{(m)}\|_{{\rm Lip}_{\rho}} a^{s/2}\Bigl(n^{-(1-\alpha)/2}\sqrt{\log n}\Bigr)^s.
\end{align*}
\end{corollary}

Thus, if the degree $s$ of smoothness of functional $f$ satisfies the condition $s\geq \frac{1}{1-\alpha},$ the $L_2$-error rate of estimator $T_f(X_1,\dots, X_n)$ is of the the order $n^{-1/2}$
(with a slower rate for $s<\frac{1}{1-\alpha}$).
Moreover, if $s>\frac{1}{1-\alpha},$ then, under the assumption that $n^{(0)}\asymp \frac{n}{\log n},$
\begin{align*}
&
\sup_{\|f'\|_{C^{s-1}}\leq 1}\sup_{P\in {\mathcal P}_{p,s}(a,d)}\Bigl\|T_f(X_1,\dots, X_n) - f(\theta(P))-\langle \hat \theta^{(1)}_1-\theta(P), f'(\theta(P))\rangle\Bigr\|_{L_p({\mathbb P}_{P})} 
\\
&
= o(n^{-1/2})
\end{align*}
as $n\to \infty,$ which could lead to asymptotically normal and asymptotically efficient
estimation of $f(\theta(P)).$
To state the results in this direction more precisely, we will need the following assumption.

\begin{assumption}
\label{assume_var_lim}
\normalfont
Suppose that
\begin{align}
\label{assume_sigma_u}
\sup_{\|u\|\leq 1}\sup_{P\in {\mathcal P}}\Bigl|\sqrt{n}\|\langle \hat \theta_n-\theta(P), u\rangle\|_{L_p({\mathbb P}_{P})}-\sigma_p (P, u)\Bigr|\to 0\ {\rm as}\ n\to \infty,
\end{align}
where $\sigma_p(P, u)\geq 0, P \in {\mathcal P}, u\in E^{\ast}$ and 
\begin{align}
\label{bd_sigma_theta}
\sup_{\|u\|\leq 1}\sup_{P\in {\mathcal P}}\sigma_p(P, u) \lesssim 1.
\end{align}
\end{assumption}

Note that $\sup_{\|u\|\leq 1}\sup_{P\in {\mathcal P}}\sigma_p(P, u)\leq \sup_{P\in {\mathcal P}}a_p^{1/2}(P),$ so, the condition $\sup_{P\in {\mathcal P}}a_p(P)\lesssim 1$
implies \eqref{bd_sigma_theta}.
Under the assumption $n^{(0)}\asymp \frac{n}{\log n}$ of Corollary \ref{cor_lin_appr}, we have 
$$
n^{(1)}_1 = n-n^{(0)} = n\Bigl(1-\frac{n^{(0)}}{n}\Bigr)= (1+o(1))n.
$$
Therefore, Assumption \ref{assume_var_lim} implies that 
\begin{align*}
\sup_{\|u\|\leq 1}\sup_{P\in {\mathcal P}}\Bigl|\sqrt{n}\|\langle \hat \theta_1^{(1)}-\theta(P), u\rangle\|_{L_p({\mathbb P}_{P})}-\sigma_p (P, u)\Bigr|\to 0\ {\rm as}\ n\to \infty,
\end{align*} 
Moreover, denoting 
\begin{align*}
\sigma_{p,f}(P):=\sigma_p(P, f'(\theta(P))),
\end{align*}
we have 
\begin{align*}
\sup_{\|f'\|_{L_{\infty}}\leq 1}\sup_{P\in {\mathcal P}}\Bigl|\sqrt{n}\|\langle \hat \theta_1^{(1)}-\theta, f'(\theta(P))\rangle\|_{L_p({\mathbb P}_{P})}-\sigma_{p,f}(P)\Bigr|\to 0\ {\rm as}\ n\to \infty.
\end{align*}
Therefore, Corollary \ref{cor_lin_appr} implies the following statement.

\begin{corollary}
\label{cor_asymp_MS}
Under the assumptions of Corollary \ref{cor_lin_appr} (ii), Assumption \ref{assume_var_lim} and for $s>\frac{1}{1-\alpha},$
\begin{align*}
\sup_{\|f'\|_{C^{s-1}}\leq 1}\sup_{P\in {\mathcal P}_{p,s}(a,d)}\Bigl|n^{1/2}\|T_f(X_1,\dots, X_n) - f(\theta(P))\|_{L_p({\mathbb P}_{P})}-\sigma_{p,f}(P)\Bigr| \to 0
\end{align*}
as $n\to \infty.$ 
\end{corollary}

In what follows, we will use the notion of Wasserstein $W_p$-distance between random variables (or, more precisely, their distributions). 
If $\xi, \eta$ are random variables in a linear normed space $F$ and $p\geq 1,$ denote 
\begin{align*}
W_p(\xi, \eta):= \inf \Bigl\{\Bigl\|\|\xi'-\eta'\|\Bigr\|_{L_p}: \xi'\overset{d}{=}\xi, \eta'\overset{d}{=}\eta\Bigr\},
\end{align*}
where the infimum is taken over all copies $\xi'$ of $\xi$ and $\eta'$ of $\eta.$
In statistical context, to emphasize the dependence of the $W_p$-distance on the parameters (such as $\theta$ or $P$)
that determine the distributions of r.v., we will use the notations $W_{p, {\mathbb P}_{\theta}}$ or $W_{p, {\mathbb P}_P}$
for the corresponding distances.

We will need the following assumption on normal approximation of linear forms $\langle \hat \theta_n-\theta, u\rangle$ (stronger than Assumption \ref{assume_var_lim}):

\begin{assumption}
\label{assume_normal_approx_A}
\normalfont
Suppose that 
\begin{align}
\label{assume_normal_approx}
\sup_{\|u\|\leq 1}\sup_{P\in {\mathcal P}}
W_{p,{\mathbb P}_{P}}\Bigl(\sqrt{n}\langle \hat \theta_n-\theta(P), u\rangle, \sigma_2(P, u)Z\Bigr)\to 0\ {\rm as}\ n\to \infty,
\end{align}
where $Z\sim N(0,1),$ $\sigma_2(P,u)\geq 0, \theta \in \Theta, u\in E^{\ast}$ and $\sigma_2(P,u)$ satisfies condition \eqref{bd_sigma_theta} with $p=2.$
\end{assumption}

Note that Assumption \ref{assume_normal_approx_A} implies Assumption \ref{assume_var_lim} with $\sigma_{p}(P,u)=\sigma_2(P,u)\|Z\|_{L_p}.$
Under the assumption $n^{(0)}\asymp \frac{n}{\log n}$ of Corollary \ref{cor_lin_appr}, Assumption \ref{assume_normal_approx_A} implies that 
\begin{align*}
\sup_{\|u\|\leq 1}\sup_{P\in {\mathcal P}}
W_{p,{\mathbb P}_{P}}\Bigl(\sqrt{n}\langle \hat \theta_1^{(1)}-\theta(P), u\rangle, \sigma_2(P, u)Z\Bigr)\to 0\ {\rm as}\ n\to \infty
\end{align*}
and the bound of Corollary \ref{cor_lin_appr} (ii) implies the following statement. 

\begin{corollary}
\label{Cor_cor_cor}
Under assumptions of Corollary \ref{cor_lin_appr} (ii), Assumption \ref{assume_normal_approx_A} and for $s>\frac{1}{1-\alpha},$ we have 
\begin{align*}
\sup_{\|f'\|_{C^{s-1}}\leq 1}\sup_{P\in {\mathcal P}_{p,s}(a,d)}
W_{p,{\mathbb P}_{P}}\Bigl(\sqrt{n}(T_f(X_1,\dots, X_n)-f(\theta(P))), \sigma_{2,f}(P)Z\Bigr)\to 0\ {\rm as}\ n\to \infty.
\end{align*}
\end{corollary}

Recall that, in the case of identifiable statistical model $X_1,\dots, X_n$ i.i.d. $\sim P_{\theta}, \theta\in \Theta, \Theta\subset E,$
we set
\begin{align*}
a(\theta)= a_p(P_{\theta})=\sup_{n\geq 1}\sup_{\|u\|\leq 1} n \Bigl\|\langle \hat \theta_n-\theta, u\rangle\Bigr\|_{L_p({\mathbb P}_{\theta})}^2,\ \theta \in \Theta
 \end{align*}
 and
 \begin{align*}
 d_p(\theta)=  d_p(P_{\theta})=\sup_{n\geq 1} n \Bigl\|\|\hat \theta_n-\theta\|\Bigr\|_{L_{p}({\mathbb P}_{\theta})}^2,\ \theta\in \Theta.
 \end{align*}
Assumption \ref{assume_var_lim} takes, in this case, the following form 
\begin{align}
\label{assume_sigma_u_theta}
\sup_{\|u\|\leq 1}\sup_{\theta\in \Theta}\Bigl|\sqrt{n}\|\langle \hat \theta_n-\theta, u\rangle\|_{L_p({\mathbb P}_{P})}-
\sigma_p (\theta, u)\Bigr|\to 0\ {\rm as}\ n\to \infty,
\end{align}
where $\sigma_p(\theta, u)\geq 0, \theta \in \Theta, u\in E^{\ast}$ and 
\begin{align}
\label{bd_sigma_theta_theta}
\sup_{\|u\|\leq 1}\sup_{\theta\in \Theta}\sigma_p(\theta, u) \lesssim 1.
\end{align}
We also set 
\begin{align*}
\sigma_{p,f}(\theta):=\sigma_p(\theta, f'(\theta))
\end{align*}
and 
\begin{align*}
\Theta_{p,s} (a,d):= \{\theta \in \Theta: a_p(\theta)\leq a, d_{ps}(\theta)\leq ad\}.
\end{align*}
Thus, Corollary \ref{cor_asymp_MS} yields 
\begin{align*}
\sup_{\|f'\|_{C^{s-1}}\leq 1}\sup_{\theta\in \Theta_{p,s}(a,d)}\Bigl|n^{1/2}\|T_f(X_1,\dots, X_n) - f(\theta)\|_{L_p({\mathbb P}_{\theta})}-\sigma_{p,f}(\theta)\Bigr| \to 0
\end{align*}
as $n\to \infty,$ under the assumptions that $d=d_n\leq n^{\alpha}$ for some $\alpha\in [1/2,1)$ and $s>\frac{1}{1-\alpha}.$ 
Note that, for $\alpha\in (0,1/2)$ and $s> \frac{1}{1-\alpha},$ the same result holds for the plug-in estimator $f(\hat \theta_n).$

Assume now that $E$ is a finite-dimensional Banach space with the dual space $E^{\ast}$ (in the finite-dimensional case, $E$ is reflexive, so, the dual space of $E^{\ast}$
is $E$). Let $\Theta\subset E$ be an open subset and let $\{P_{\theta}:\theta \in \Theta\}$ be a statistical model. 
We will also assume that standard regularity assumptions (such as quadratic 
mean differentiability) hold for model $\{P_{\theta}:\theta \in \Theta\},$ allowing us to define the score function $\frac{\partial}{\partial \theta}\log p_{\theta}$ 
and the Fisher information $I(\theta).$  In this case, the score function $\frac{\partial}{\partial \theta}\log p_{\theta}(X)$ takes values in $E^{\ast}$
and the Fisher information
\begin{align*}
I(\theta)= {\mathbb E}_{\theta}\frac{\partial}{\partial \theta}\log p_{\theta}(X)\otimes \frac{\partial}{\partial \theta}\log p_{\theta}(X)
\end{align*}
could be viewed as a symmetric bounded linear operator from $E$ into $E^{\ast}.$ If the Fisher information operator $I(\theta)$ is invertible for all $\theta\in \Theta$
with the inverse $I(\theta)^{-1}: E^{\ast}\mapsto E$
and functional $f:\Theta\mapsto {\mathbb R}$ is differentiable, we can define 
\begin{align*}
\sigma_f^2 (\theta):= \langle I(\theta)^{-1}f'(\theta), f'(\theta)\rangle, \theta\in \Theta. 
\end{align*}
Suppose there exists an estimator $\hat \theta_n$ based on i.i.d. observations $X_1,\dots, X_n\sim P_{\theta}$ (for instance, the maximum likelihood estimator) such that $\sqrt{n}(\hat \theta_n-\theta)$ is close in distribution to $N(0, I(\theta)^{-1}).$ To be more specific, suppose that 
\begin{align*}
\sup_{\|u\|\leq 1}\sup_{\theta\in \Theta}
W_{p,{\mathbb P}_{\theta}}\Bigl(\sqrt{n}\langle \hat \theta_n-\theta, u\rangle, \langle I(\theta)^{-1} u,u\rangle^{1/2} Z\Bigr)\to 0\ {\rm as}\ n\to \infty,
\end{align*}
where $Z\sim N(0,1).$ Then, assumption \eqref{assume_sigma_u_theta} holds with $\sigma_p(\theta,u)=\langle I(\theta)^{-1} u,u\rangle^{1/2}\|Z\|_{L_p}$
and condition \eqref{bd_sigma_theta_theta} is equivalent to $\sup_{\theta\in \Theta}\|I(\theta)^{-1}\|\lesssim 1.$ Thus, we can conclude that, 
under the assumptions $d=d_n\leq n^{\alpha}$ for some $\alpha\in [1/2,1)$ and $s>\frac{1}{1-\alpha},$  
\begin{align*}
\sup_{\|f'\|_{C^{s-1}}\leq 1}\sup_{\theta\in \Theta_{p,s}(a,d)}\Bigl|n^{1/2}\|T_f(X_1,\dots, X_n) - f(\theta)\|_{L_p({\mathbb P}_{\theta})}-\sigma_{f}(\theta)\|Z\|_{L_p}\Bigr| \to 0
\end{align*}
and, in view of Corollary \ref{Cor_cor_cor}, 
\begin{align*}
\sup_{\|f'\|_{C^{s-1}}\leq 1}\sup_{\theta\in \Theta_{p,s}(a,d)}
W_{p,{\mathbb P}_{\theta}}\Bigl(\sqrt{n}(T_f(X_1,\dots, X_n)-f(\theta)), \sigma_f(\theta)Z\Bigr)\to 0\ {\rm as}\ n\to \infty
\end{align*}
(with similar results holding for the plug-in estimator $f(\hat \theta_n)$ if $\alpha \in (0,1/2)$ and $s>\frac{1}{1-\alpha}$).

Recall also that, in the case of Euclidean space $E={\mathbb R}^d,$ we have $d_p(\theta)\leq a_p(\theta)d$ (see Example \ref{ex_Euc}).
Thus, if we choose $a:=\sup_{\theta\in \Theta}a_p(\theta),$ then $\Theta_{p,s}(a,d)=\Theta$ and, under the assumption that 
$\sup_{\theta\in \Theta}a_p(\theta)\lesssim 1,$ we have 
\begin{align}
\label{conv_to_zero}
\sup_{\|f'\|_{C^{s-1}}\leq 1}\sup_{\theta\in \Theta}\Bigl|n^{1/2}\|T_f(X_1,\dots, X_n) - f(\theta)\|_{L_p({\mathbb P}_{\theta})}-\sigma_{f}(\theta)\|Z\|_{L_p}\Bigr| \to 0
\end{align}
and
\begin{align*}
\sup_{\|f'\|_{C^{s-1}}\leq 1}\sup_{\theta\in \Theta}
W_{p,{\mathbb P}_{\theta}}\Bigl(\sqrt{n}(T_f(X_1,\dots, X_n)-f(\theta)), \sigma_f(\theta)Z\Bigr)\to 0\ {\rm as}\ n\to \infty.
\end{align*}

\begin{remark}
\normalfont
Note that relevant parameter $d$ is not always the linear dimension of space $E.$ For instance, if $E$ is the space of $d\times d$
matrices equipped with the operator norm, the linear dimension is $d^2,$ but $d(\theta)$ would be of the order $d,$ provided that $\langle \hat \theta_n-\theta, u\rangle$
satisfies Bernstein type inequalities, see Section \ref{functionals_Bernstein}. 
\end{remark}

It turns out that, in the framework described above, estimator $T_f(X_1,\dots, X_n)$ is asymptotically efficient under 
the assumptions that $d=d_n\leq n^{\alpha}$ for some $\alpha\in [1/2,1)$ and $s>\frac{1}{1-\alpha}$ and so is the plug-in 
estimator $f(\hat \theta_n)$  for $\alpha\in [0,1/2)$ and $s>\frac{1}{1-\alpha}.$ 
These facts follow from local minimax lower bounds stated bellow and proved in Section \ref{sec:lower_bounds}.



For $\theta_0\in \Theta$ and $\delta>0,$ denote 
\begin{align*}
\omega_I (\theta_0, \delta):= \sup_{\theta\in \Theta, \|\theta-\theta_0\|\leq \delta} \|I(\theta)-I(\theta_0)\|
\end{align*}
and, for $f\in C^1(\Theta),$
\begin{align*}
\omega_{f'} (\theta_0, \delta):= \sup_{\theta\in \Theta, \|\theta-\theta_0\|\leq \delta} \|f'(\theta)-f'(\theta_0)\|.
\end{align*}

\begin{theorem}
\label{vanTrees_bd}
Let $f:\Theta\mapsto {\mathbb R}$ be a continuously differentiable functional. Suppose that, for $\theta_0\in \Theta$ and $\delta>0,$
$B(\theta_0,\delta)\subset \Theta$ and the Fisher information operator $I(\theta):E\mapsto E^{\ast}$ is invertible for all $\theta\in B(\theta_0,\delta).$
Then, with some numerical 
constant $D\geq 2,$
\begin{align*}
&
\inf_{T_n} \sup_{\|\theta-\theta_0\|\leq \delta} \frac{n {\mathbb E}_{\theta}(T_n(X_1,\dots, X_n)-f(\theta))^2}{\sigma_f^2(\theta)}
\\
&
\geq 1- D\|I(\theta_0)\|\|I(\theta_0)^{-1}\|\Bigl( \frac{\omega_{f'}(\theta_0,\delta)}{\|f'(\theta_0)\|}  + \|I(\theta_0)^{-1}\|\ \omega_I(\theta_0,\delta)+ 
\frac{\|I(\theta_0)^{-1}\|}{\delta^2 n}\Bigr),
\end{align*}
where the infimum is taken over all estimators $T_n(X_1,\dots, X_n)$ based on i.i.d. observations $X_1,\dots, X_n\sim P_{\theta}.$
\end{theorem}

Setting in the statement of the theorem $\delta=\frac{c}{\sqrt{n}}$ and assuming that 
\begin{align}
\label{assume_on_I}
\|I(\theta_0)\|\lesssim 1, \|I(\theta_0)^{-1}\|\lesssim 1\ {\rm and}\ \omega_I(\theta_0, \delta)\to 0\ {\rm as}\ \delta\to 0
\end{align}
yield the following local asymptotic minimax bound for all $\rho\in (0,1], \lambda<\rho:$
\begin{align*}
\lim_{c\to \infty}\liminf_{n\to\infty}\inf_{\|f'\|_{{\rm Lip}_{\rho}(U_{n,c}(\theta_0))}\leq 1, \|f'(\theta_0)\|\geq (\frac{c}{\sqrt{n}})^{\lambda}}
\inf_{T_n} \sup_{\|\theta-\theta_0\|\leq cn^{-1/2}} \frac{\sqrt{n} \Bigl\|T_n(X_1,\dots, X_n)-f(\theta)\Bigr\|_{L_2({\mathbb P}_{\theta})}}{\sigma_f(\theta)}
\geq 1,
\end{align*}  
where $U_{n,c}(\theta_0):= B(\theta_0, cn^{-1/2}).$ To compare, it follows from \eqref{conv_to_zero} that, under the assumptions 
$\sup_{\theta\in \Theta}\|I(\theta)\|\lesssim 1,$
$\sup_{\theta\in \Theta}\|I(\theta)^{-1}\|\lesssim 1$
and 
$\sup_{\theta\in \Theta} a_2(\theta)\lesssim 1,$
for all $\tau>0$ we have 
\begin{align*}
\sup_{\|f'\|_{C^{s-1}}\leq 1, \|f'(\theta)\|\geq \tau}\sup_{\theta\in \Theta}\Bigl|\frac{n^{1/2}\|T_f(X_1,\dots, X_n) - f(\theta)\|_{L_2({\mathbb P}_{\theta})}}{\sigma_f(\theta)}-1\Bigr| \to 0\
{\rm as}\ n\to\infty
\end{align*}
provided that $d=d_n\leq n^{\alpha}$ for some $\alpha\in [1/2,1)$ and $s>\frac{1}{1-\alpha}$ (with a similar result holding for $f(\hat \theta_n)$
if $\alpha\in (0,1/2)$ and $s>\frac{1}{1-\alpha}$). This shows the asymptotic efficiency of estimators $T_f(X_1,\dots, X_n)$ or $f(\hat \theta_n)$
in these two cases.

\section{High-dimensional models with independent low-dimensional components}
\label{HD_comps}

Let $(S_j, {\mathcal A}_j), j=1,\dots, d$ be measurable spaces and, for $j=1,\dots, d,$ let ${\mathcal P}_j$ be a family 
of probability measures in $(S_j, {\mathcal A}_j).$ We will be interested in the product space $S:=S_1\times \dots \times S_d$
equipped with the product of $\sigma$-algebras ${\mathcal A}:= {\mathcal A}_1\times \dots \times {\mathcal A}_d$ and 
a family of probability measures 
\begin{align*}
{\mathcal P}:= \{P_1\times \dots \times P_d: P_j\in {\mathcal P}_j, j=1,\dots, d\}.
\end{align*}
Let $X\sim P, P\in {\mathcal P}$ and let $X_1,\dots, X_n$ be i.i.d. observations of r.v. $X.$
In other words, $X=(X^{(1)}, \dots, X^{(d)})$ is a vector with independent components $X^{(1)}, \dots, X^{(d)}$
sampled from distributions $P_1\in {\mathcal P}_1,\dots, P_d\in {\mathcal P}_d,$ respectively, and $P=P_1\times \dots \times P_d.$

Consider also Banach spaces $E_1,\dots, E_d$ and mappings $\theta^{(j)} : {\mathcal P}_j \mapsto E_j, j=1,\dots, d.$
Denote $\theta=(\theta^{(1)}, \dots, \theta^{(d)}): {\mathcal P}\mapsto E:=E_1\times \dots \times E_d,$
\begin{align*}
\theta (P) = (\theta^{(1)}(P_1),\dots, \theta^{(d)}(P_d)), P=P_1\times \dots \times P_d, P_j\in {\mathcal P}_j, j=1,\dots, d.
\end{align*}
Given a functional $f:E\mapsto {\mathbb R},$ the goal is to estimate its value $f(\theta(P))$ based on i.i.d. observations 
$X_1,\dots, X_n\sim P, P\in {\mathcal P}.$ We will be primarily interested in this problem in the case when the components 
of the model are low-dimensional, but the number $d$ of independent components is large. 
In what follows, we will equip the space $E=E_1\times \dots \times E_d$ with the norm 
\begin{align*}
\|x\|:= \Bigl(\sum_{j=1}^d \|x^{(j)}\|^2\Bigr)^{1/2}, x=(x_1^{(1)},\dots, x^{(d)})\in E_1\times \dots \times E_d.
\end{align*}
Its dual space $E^{\ast}=(E_1\times \dots \times E_d)^{\ast}$ could be then identified with the space $E_1^{\ast}\times \dots \times E_d^{\ast}$
equipped with the norm 
\begin{align*}
\|u\|:= \Bigl(\sum_{j=1}^d \|u^{(j)}\|^2\Bigr)^{1/2}, u=(u_1^{(1)},\dots, u^{(d)})\in E_1^{\ast}\times \dots \times E_d^{\ast}.
\end{align*}

Let $\hat \theta_n^{(j)}= \hat \theta_n^{(j)}(X_1^{(j)},\dots, X_n^{(j)})$ be an estimator of parameter $\theta^{(j)}(P_j)$
based on i.i.d. observations $X_1^{(j)},\dots, X_n^{(j)}$ and let $\hat \theta_n:= \hat \theta_n (X_1,\dots, X_n)= (\hat \theta_n^{(1)},\dots, \hat \theta_n^{(d)})$ be an estimator of parameter $\theta(P),$ so that the components $\hat \theta_n^{(1)},\dots, \hat \theta_n^{(d)}$ of estimator $\hat \theta_n$ are independent r.v..  
Given a smooth functional $f:E\mapsto {\mathbb R},$ we will consider estimator $T_f(X_1,\dots, X_n)$ based 
on i.i.d. observations $X_1,\dots, X_n\sim P, P\in {\mathcal P}$ and defined by \eqref{basic_T_f}. The next result is a corollary of Theorem 
\ref{th_1_asssume_AAA}. 


For $p\geq 1,$ $1\leq j\leq d$ and $P_j\in {\mathcal P}_j,$ denote 
\begin{align*}
b_{j}(P_j):= \sup_{n\geq 1}n\Bigl\|{\mathbb E}_{P_j} \hat \theta_{n}^{(j)} - \theta^{(j)}(P_j)\Bigr\|,
\end{align*}
\begin{align*}
a_{p,j}(P_j):=\sup_{n\geq 1}\sup_{\|u^{(j)}\|\leq 1, u^{(j)}\in E_j^{\ast}} n\Bigl\|\langle \hat \theta_n^{(j)}-\theta^{(j)}(P_j), u^{(j)}\rangle\Bigr\|_{L_p({\mathbb P}_{P})}^2
 \end{align*}
and 
\begin{align*}
d_{p,j}(P_j):=\sup_{n\geq 1} n\Bigl\|\|\hat \theta_n^{(j)}-\theta^{(j)}(P_j)\|\Bigr\|_{L_{p}({\mathbb P}_{P})}^2.
\end{align*}


\begin{theorem}
\label{cor_cor_th_1_asssume_AAA}
Let $s:=m+\rho$ for some $m\geq 2$ and $\rho\in (0,1]$ and let $f'\in C^{s-1}(E).$ 
Suppose that $d\lesssim n$ and, for some $p\geq 2,$ 
\begin{align}
\label{cond_a_b_d}
\max_{1\leq j\leq d}b_j(P_j)\lesssim 1,\ \max_{1\leq j\leq d} a_{p,j}(P_j)\lesssim 1,\  \max_{1\leq j\leq d}d_{ps,j}(P_j)\lesssim 1.
\end{align}

(i) Suppose that $n^{(0)}\asymp_m n$ and $n_j^{(k)}\asymp_m n$, $j=1,\cdots,k$, $k=1,\cdots,m.$
Then, for all $P\in {\mathcal P},$ 
\begin{align*}
    &\left\|T_f(X_1,\cdots,X_n)-f(\theta(P))\right\|_{L_p(\mathbb{P}_P)}
    \lesssim_{p,s} \max_{1\leq k\leq m} \|f^{(k)}\|_{L_\infty} \frac{1}{\sqrt{n}} \Bigl(\sqrt{\frac{d}{n}}\Bigr)^{k-1}+\|f^{(m)}\|_{{\rm Lip}_\rho}
   \Bigl(\sqrt{\frac{d}{n}}\Bigr)^{s} .
\end{align*}

(ii) Suppose that, for all $j=1,\dots, k, k=1,\dots, m,$ $n_j^{(k)}\asymp_m n.$ Then, for all $P\in {\mathcal P},$ 
\begin{align*}
&
\Bigl\|T_f(X_1,\dots, X_n) - f(\theta(P))- \langle \hat \theta^{(1)}_1-\theta(P), f'(\theta(P))\rangle\Bigr\|_{L_p({\mathbb P}_{P})} 
\\
&
\lesssim_{p,s}
\max_{2\leq k\leq m}\|f^{(k)}\|_{L_{\infty}} \frac{1}{\sqrt{n}} 
\Bigl(\sqrt{\frac{d}{n^{(0)}}}\Bigr)^{k-1} 
+ \|f^{(m)}\|_{{\rm Lip}_{\rho}}\Bigl(\sqrt{\frac{d}{n^{(0)}}}\Bigr)^{s}.
\end{align*}
\end{theorem}

The proof of this theorem is given in Section \ref{sec:HDLD}.

In particular, this framework includes the case when, for $j=1,\dots, d,$ $E_j={\mathbb R}^{l_j}$ are finite-dimensional Euclidean spaces, $E=E_1\times \dots\times E_d={\mathbb R}^l,$ where $l:=l_1+\dots+l_d,$  
and 
${\mathcal P}_{j}:=\{P_{\theta^{(j)}}: \theta^{(j)}\in \Theta_j\}, \Theta_j\subset E_j$ are statistical models with identifiable parameter 
$\theta^{(j)}$ (recall that, in this case, we set $\theta^{(j)}(P_{\theta^{(j)}})=\theta^{(j)}$). 
Then, ${\mathcal P}:= \{P_{\theta}: \theta\in \Theta\}, \Theta:= \Theta_1\times \dots \times \Theta_d,$
where 
\begin{align*}
P_{\theta}:= P_{\theta^{(1)}}\times \dots \times P_{\theta^{(d)}}, \theta:= (\theta^{(1)},\dots, \theta^{(d)})\in \Theta.
\end{align*}
We will also assume that models $\{P_{\theta^{(j)}}: \theta^{(j)}\in \Theta_j\}$ are regular with non-singular Fisher information matrices 
$I_j(\theta^{(j)}), j=1,\dots, d.$ Let $\hat \theta_n^{(j)}$ be a maximum likelihood estimator of parameter $\theta^{(j)}, j=1,\dots, d.$
Then, by the asymptotic normality of MLE, the sequence of r.v. $\sqrt{n}(\hat \theta_n^{(j)}-\theta^{(j)})$ converges in distribution  
to $\xi^{(j)}(\theta^{(j)})= I_j(\theta^{(j)})^{-1/2} Z_j,$ $Z_j$ being a standard normal r.v. in ${\mathbb R}^{l_j}.$ 
We focus on the case when the components 
of the model are low-dimensional (so, $\max_{1\leq j\leq d} l_j\lesssim 1$), but the number $d$ of the components 
could be large. In this case, the accuracy of normal approximation of maximum likelihood estimators of the low-dimensional components 
$\theta^{(j)}$ is typically of the order $O(n^{-1/2}),$ which could be quantified, for instance, in terms of Wasserstein $W_p$-distances. 
Namely, for some $p\geq 2,$ for $j=1,\dots, d$ and for $\theta\in \Theta,$ denote  
\begin{align}
\label{cond_wasser}
C_{p,j}(\theta^{(j)}):= 
\sup_{n\geq 1} \sqrt{n} W_{p,{\mathbb P}_{\theta}} 
\Bigl(\sqrt{n}(\hat \theta_n^{(j)}-\theta^{(j)}), \xi^{(j)}(\theta^{(j)})\Bigr).
\end{align}
In this framework, the following result holds (see Section \ref{sec:HDLD} for the proof).

\begin{corollary}
\label{cor_cor_th_1_asssume_regular}
Let $s:=m+\rho$ for some $m\geq 2$ and $\rho\in (0,1]$ and let $f'\in C^{s-1}(E).$  Suppose that $d\lesssim n$ and 
\begin{align*}
\max_{1\leq j\leq d}\sup_{\theta^{(j)}\in \Theta_j}\|I(\theta_j)^{-1/2}\|\lesssim 1,\ \max_{1\leq j\leq d} \sup_{\theta^{(j)}\in \Theta_j}C_{ps,j}(\theta^{(j)})\lesssim 1. 
\end{align*}

(i) Suppose that $n^{(0)}\asymp_m n$ and $n_j^{(k)}\asymp_m n$, $j=1,\cdots,k$, $k=1,\cdots,m.$
Then
\begin{align*}
    &\sup_{\theta\in \Theta}\left\|T_f(X_1,\cdots,X_n)-f(\theta)\right\|_{L_p(\mathbb{P}_\theta)}
    \lesssim_{m,p} \max_{1\leq k\leq m} \|f^{(k)}\|_{L_\infty} \frac{1}{\sqrt{n}} \Bigl(\sqrt{\frac{d}{n}}\Bigr)^{k-1}+\|f^{(m)}\|_{{\rm Lip}_\rho}
   \Bigl(\sqrt{\frac{d}{n}}\Bigr)^{s} .
\end{align*}

(ii) Suppose that, for all $j=1,\dots, k, k=1,\dots, m,$ $n_j^{(k)}\asymp_m n.$ Then, for all $\theta\in \Theta,$ 
\begin{align*}
&
\sup_{\theta\in \Theta}\Bigl\|T_f(X_1,\dots, X_n) - f(\theta)- \langle \hat \theta^{(1)}_1-\theta, f'(\theta)\rangle\Bigr\|_{L_p({\mathbb P}_{\theta})} 
\\
&
\lesssim_{m,p}
\max_{2\leq k\leq m}\|f^{(k)}\|_{L_{\infty}} \frac{1}{\sqrt{n}} 
\Bigl(\sqrt{\frac{d}{n^{(0)}}}\Bigr)^{k-1} 
+ \|f^{(m)}\|_{{\rm Lip}_{\rho}}\Bigl(\sqrt{\frac{d}{n^{(0)}}}\Bigr)^{s}.
\end{align*}
\end{corollary}

\begin{remark}
\normalfont
Earlier result in the same direction as Corollary \ref{cor_cor_th_1_asssume_regular} was obtained in \cite{Koltchinskii_2022} for functional estimators based 
on iterated bootstrap bias reduction and for the $L_2$-errors (see Corollary 2.5 in that paper and the discussion afterwards). However, these results 
required certain smoothness assumptions on functions $\theta_j \mapsto I(\theta_j)^{-1/2}$ that are, most likely,
related to the methods of proofs used in \cite{Koltchinskii_2022}. In the current paper, such smoothness 
assumptions are not required for estimators based on Taylor expansion and the sample split. 

Berry-Esseen type bounds of the order $n^{-1/2}$ on the accuracy of normal approximations of MLE in regular low-dimensional models 
could be found in \cite{Pfanzagl, Bentkus, Pinelis} whereas the bounds in Wasserstein $W_1$-distance were proved in \cite{Anastas}.
It would not be hard to adapt the methods used in these papers along with known bounds on normal approximations of sums of i.i.d.
random variables in Wasserstein $W_p$-distances (see, for instance, \cite{Rio}) to obtain similar results for normal approximation of MLE 
in $W_p$-distances. 
\end{remark}

It easily follows from the first bound of Corollary \ref{cor_cor_th_1_asssume_regular} that, for $p=2$ and for truncated estimator $\tilde T_f(X_1,\dots, X_n)$
with $M=1,$  
\begin{align*}
\sup_{\|f\|_{C^{s}({\mathbb R}^m)}\leq 1}\sup_{\theta\in \Theta } {\mathbb E}_{\theta} (\tilde T_f(X_1,\dots, X_n)-f(\theta))^2 \lesssim 
\Bigl(\frac{1}{n} \vee \Bigl(\frac{d}{n}\Bigr)^s\Bigr)\wedge 1.
\end{align*}
The next proposition provides the corresponding minimax lower bound. Its proof is given in Section \ref{sec:lower_bounds}.

\begin{proposition}
\label{max_min_max_many}
Let $\Theta\subset {\mathbb R}^l={\mathbb R}^{l_1}\times \dots \times {\mathbb R}^{l_d},$ $l=l_1+\dots+l_d,$ and let $\{P_{\theta}: \theta \in \Theta\}$ be a statistical model 
on an arbitrary measurable space $(S,{\mathcal A}).$
Let $\theta_0=(\theta_0^{(1)}, \dots, \theta_0^{(d)})\in \Theta$ and suppose that, for some $C>0$ and $\rho\leq \frac{\gamma}{C}$ with a sufficiently 
small numerical constant $\gamma>0,$ and for all $j=1,\dots, d,$
\begin{align*}
\Theta_n(\theta_0;\rho):=B(\theta_0^{(1)},\rho n^{-1/2}) \times \dots \times B(\theta_0^{(d)},\rho n^{-1/2})\subset \Theta.
\end{align*}
Also suppose that 
\begin{align*}
K(P_{\theta}\|P_{\theta_0}) \leq C^2 \|\theta-\theta_0\|^2, \theta\in \Theta_n(\theta_0;\rho).
\end{align*}
Then 
\begin{align*}
\sup_{\|f\|_{C^{s}({\mathbb R}^m)}\leq 1}\inf_{T_n}\sup_{\theta\in \Theta_n(\theta_0;\rho)} {\mathbb E}_{\theta} (T_n(X_1,\dots, X_n)-f(\theta))^2 \gtrsim 
\Bigl(\frac{\rho^2}{n} \vee \Bigl(\rho^2\frac{d}{n}\Bigr)^s\Bigr)\wedge 1,
\end{align*}
where the infimum is taken over all the estimators $T_n(X_1,\dots, X_n)$ based on i.i.d. observations $X_1,\dots, X_n\sim P_{\theta}.$
In particular, the bound above holds if $S=S_1\times \dots \times S_d,$ $P_{\theta}=P_{\theta^{(1)}}\times \dots \times P_{\theta^{(d)}},$
$\theta\in \Theta=\Theta_1\times \dots \times \Theta_d,$ provided that, for all $j=1,\dots, d,$ $B(\theta_0^{(j)},\rho n^{-1/2})\subset \Theta_j$ and 
\begin{align*}
K(P_{\theta^{(j)}}\| P_{\theta_0^{(j)}}) \leq C^2\|\theta^{(j)}-\theta_0^{(j)}\|^2, \theta^{(j)}\in B(\theta_0^{(j)},\rho n^{-1/2}).
\end{align*}
\end{proposition}

The following fact is an immediate consequence of Corollary \ref{cor_cor_th_1_asssume_regular}(ii)
(see Section \ref{sec:HDLD} for the proof). 
Let 
\begin{align*}
\sigma_f^2(\theta) = \sum_{j=1}^d \langle I_j(\theta^{(j)})^{-1} (\partial_j f)(\theta), (\partial_j f)(\theta)\rangle.
\end{align*}
Here $(\partial_j f)(\theta)= \frac{\partial}{\partial \theta^{(j)}}f (\theta^{(1)}, \dots, \theta^{(d)})$ denotes the partial derivative of $f(\theta)$
with respect to its variable $\theta^{(j)}.$

\begin{proposition}
\label{asym_eff_many_comp}
Suppose assumptions of Corollary \ref{cor_cor_th_1_asssume_regular}(ii) hold and $n^{(0)}\asymp \frac{n}{\log n}.$ Suppose also that $d\leq n^{\alpha}$ for some $\alpha\in [1/2,1)$ and $s>\frac{1}{1-\alpha}.$ Then 
\begin{align*}
\sup_{\|f'\|_{C^{s-1}(\Theta)}\leq 1}\sup_{\theta\in \Theta}\Bigl|n^{1/2}\|T_f(X_1,\dots, X_n) - f(\theta)\|_{L_p({\mathbb P}_{\theta})}-\sigma_{f}(\theta)\|Z\|_{L_p}\Bigr| \to 0
\end{align*}
and
\begin{align*}
\sup_{\|f'\|_{C^{s-1}}\leq 1}\sup_{\theta\in \Theta}
W_{p,{\mathbb P}_{\theta}}\Bigl(\sqrt{n}(T_f(X_1,\dots, X_n)-f(\theta)), \sigma_f(\theta)Z\Bigr)\to 0\ {\rm as}\ n\to \infty,
\end{align*} 
where $Z\sim N(0,1).$
\end{proposition}

\begin{remark}
\normalfont
Note that the Fisher information $I(\theta)$ of the model with independent components
$P_{\theta}=P_{\theta^{(1)}}\times \dots \times P_{\theta^{(d)}}, \theta=(\theta^{(1)}, \dots, \theta^{(d)})\in \Theta=\Theta_1\times \dots \times \Theta_d$
satisfies the following property:
\begin{align*} 
\langle I(\theta)^{-1}u,u\rangle= \sum_{j=1}^d \langle I_j(\theta_j)^{-1} u^{(j)}, u^{(j)}\rangle.
\end{align*}
Therefore, $\sigma_f^2(\theta)= \langle I(\theta)^{-1} f'(\theta), f'(\theta)\rangle.$ It is easy to check that under minor smoothness assumptions 
on the functions $\theta^{(j)}\mapsto I_j(\theta^{(j)}),$ the local minimax lower bound of Theorem \ref{vanTrees_bd} holds. It means that, under conditions of 
Proposition \ref{asym_eff_many_comp}, estimator $T_f(X_1,\dots, X_n)$ is asymptotically efficient. The same property holds for the plug-in 
estimator $f(\hat \theta_n)$ provided that $d\leq n^{\alpha}$ for some $\alpha\in [0,1/2)$ 
and $s>\frac{1}{1-\alpha}.$   
\end{remark}

\section{Functionals of covariance operators}
\label{func_cov_op}

Let $E$ be a separable Banach space with the dual space $E^{\ast}.$ 
For a centered random variable $X$ in $E$ with the finite weak second moment
$
{\mathbb E}\langle X,u\rangle^2 <\infty, u\in E^{\ast},
$
the covariance operator $\Sigma:E^{\ast}\mapsto E$ is defined as 
$
\Sigma u := {\mathbb E}\langle X,u\rangle X, u\in E^{\ast}.
$
It is a symmetric bounded operator from $E^{\ast}$ into $E.$

A centered random variable $X$ with covariance operator $\Sigma$ is called {\it pre-gaussian} if there exists a centered Gaussian r.v. $Y$ in $E$ with the same covariance operator $\Sigma.$ 
The following fact is well known (see, e.g.,  \cite{Kwapien}).

\begin{proposition}
\label{Gauss_repres}
Let $Y$ be a centered Gaussian random variable in a separable Banach space $E.$ 
Then, there exists a sequence $\{x_n\}$ of linearly independent vectors 
in $E$ (that is, for all $n\geq 1,$ $x_n \not\in {\rm c.l.s.}(\{x_k: k\neq n\})$) and a sequence $\{g_n\}$ of i.i.d. standard normal r.v. such that $Y=\sum_{n\geq 1} g_n x_n$ with the random series in the right hand side converging in $E$ a.s. and $\sum_{n\geq 1}\|x_n\|^2<\infty.$
\end{proposition}

Note that the following representation holds for the covariance operator $\Sigma$ of r.v. $Y:$
\begin{align*}
\Sigma=\sum_{n\geq 1} x_n\otimes x_n,
\end{align*}
where the tensor product $x\otimes y, x,y\in E$ is defined as the operator from $E^{\ast}$ into $E,$ 
$(x\otimes y)u=x \langle y,u\rangle, u\in E^{\ast}.$

For such vector $Y\sim N(0,\Sigma),$ define 
\begin{align*}
{\bf r}(\Sigma):= \frac{{\mathbb E}\|Y\|^2}{\|\Sigma\|}.
\end{align*}
The quantity ${\bf r}(\Sigma)$ is called {\it the effective rank} of $\Sigma$ and it has been used as 
a complexity characteristic of covariance estimation problem (see \cite{Koltchinskii_Lounici, Koltchinskii_Lounici_bilinear, Koltchinskii_2022} with further references and discussion therein). Note that ${\bf r}(\Sigma)\geq 1,$ ${\bf r}(\lambda \Sigma)={\bf r}(\Sigma), \lambda>0$ and 
\begin{align*}
{\bf r}(\Sigma) \leq {\rm rank}(\Sigma)\leq {\rm dim}(E).
\end{align*} 
If $E$ is a Hilbert space, then ${\bf r}(\Sigma)=\frac{{\rm tr}(\Sigma)}{\|\Sigma\|}.$ 
If $E={\mathbb R}^d$ and the spectrum of $\Sigma$ is in the interval $[1/a, a]$ for some $a\geq 1,$
 then ${\bf r}(\Sigma)\asymp_{a} d.$ However, the effective rank is finite even for covariance operators 
 of infinite rank and it becomes a natural complexity parameter in covariance estimation problems 
 when the target covariance has some of its eigenvalues close to zero. 

A centered random variable $X$ in $E$ is called {\it subgaussian} iff there exists a constant $C>0$ such that, for all $u\in E^{\ast},$
\begin{align}
\label{sub-gauss}
\|\langle X,u\rangle\|_{\psi_2} \leq C\|\langle X,u\rangle\|_{L_2}.
\end{align}
In what follows in this section, denote by ${\mathcal P}$ the class of distributions of all centered subgaussian and pre-gaussian random variables $X.$

Our goal is to construct estimators of $f(\Sigma)$ for a smooth functional $f:L(E^{\ast}, E)\mapsto {\mathbb R}$ based 
on i.i.d. copies $X_1,\dots, X_n$ of r.v. $X.$ Here $L(E^{\ast}, E)$ denotes the space of bounded symmetric operators from 
$E^{\ast}$ into $E$ equipped with the operator norm. 
If $P\in {\mathcal P}$ is the distribution of $X,$ then $\Sigma = {\mathbb E}_P(X\otimes X)=: \theta(P),$ so, we can use the results of Section \ref{Main_results}
and their modified versions from Section \ref{functionals_Bernstein} (see Theorem \ref{th_1_Bern_AA}) 
on estimation of $f(\theta(P))$ based on i.i.d. observations $X_1,\dots, X_n\sim P.$ 
We will use the sample covariance operator $\hat \Sigma_n: E^{\ast}\mapsto E$
\begin{align*}
\hat \Sigma_n u := n^{-1}\sum_{j=1}^n \langle X_j,u\rangle X_j,\ u\in E^{\ast}
\end{align*}
as a base estimator of $\Sigma$ and construct estimator $T_f(X_1,\dots, X_n)$ (defined by \eqref{basic_T_f} with $\hat \theta_n=\hat \Sigma_n$) and its truncated version $\tilde T_f(X_1,\dots, X_n)$
(defined by \eqref{basic_T_fM}), using the sample split with the sub-sample sizes $n^{(0)}, n_j^{(k)}, j=1,\dots, k, k=1,\dots, m,$ as described in Section \ref{Main_results}.

The next result is a corollary of Theorem \ref{th_1_Bern_AA}.

\begin{theorem}
\label{th_1_Bern_AA_covariance}
Let $f:L(E^{\ast}, E)\mapsto {\mathbb R}$ be a uniformly bounded functional and let $\Sigma=\Sigma_P$ for some $P\in {\mathcal P}.$
Suppose $f$ is $m$ times Fr\`echet continuously differentiable in $U:=B(\Sigma;\delta)$ for some $\delta\in (0,1]$ and for some $m\geq 2,$ and, moreover,
$f^{(m)}$ satisfies the Hölder condition with exponent $\rho\in(0,1]$ in $U.$ Denote $s:=m+\rho.$
Moreover, assume that for a sufficiently large constant $C'>0,$
\begin{align*}
C'\|\Sigma\|\left(\sqrt{\frac{{\bf r}(\Sigma)}{n^{(0)}}}\vee \frac{{\bf r}(\Sigma)}{n^{(0)}}\right)\leq \delta.
\end{align*}

(i) Suppose that $n^{(0)}\asymp_m n$ and $n_j^{(k)}\asymp_m n$, $j=1,\cdots,k$, $k=1,\cdots,m.$
Then, for $M\geq \|f\|_{L_{\infty}},$ $p\geq 1$ and for a constant $c_1>0$ depending on $m,C,$     
\begin{align*}
    &\left\|\tilde{T}_f(X_1,\cdots,X_n)-f(\Sigma)\right\|_{L_p(\mathbb{P}_P)}
    \\
    &
    \lesssim_{s, C} \max_{1\leq k\leq m} \|f^{(k)}\|_{L_\infty(U)} \|\Sigma\|\left(\sqrt{\frac{p}{n}}\vee \frac{p}{n}\right)    
   +\|f^{(m)}\|_{{\rm Lip}_\rho(U)} \|\Sigma\|^s\left(\Bigl(\frac{{\bf r}(\Sigma)}{n}\Bigr)^{s/2}\vee \Bigl(\frac{{\bf r}(\Sigma)}{n}\Bigr)^s\vee \Bigl(\frac{p}{n}\Bigr)^{s/2}\vee \Bigl(\frac{p}{n}\Bigr)^s\right)
   \\
   &
 \ \ \ \ \ \ \  +(\|f\|_{L_\infty}+M)\exp\Bigl\{-c_1 \frac{n}{p}\Bigl(\frac{\delta^2}{\|\Sigma\|^2}\wedge \frac{\delta}{\|\Sigma\|}\Bigr)\Bigr\}.
    \end{align*}
    
 (ii) Suppose that, for all $j=1,\dots, k, k=1,\dots, m,$ $n_j^{(k)}\asymp_m n.$
If $M\geq \|f\|_{\infty}+ \|f\|_{{\rm Lip}}\delta,$ then 
\begin{align*}
    &\left\|\tilde T_f(X_1,\cdots,X_n)-f(\Sigma)- \langle \hat \Sigma_1^{(1)}-\Sigma, f'(\Sigma)\rangle\right\|_{L_p(\mathbb{P}_P)}
    \\
    &
    \lesssim_{s, C} \max_{2\leq k\leq m} \|f^{(k)}\|_{L_\infty(U)} \|\Sigma\|^2\left(\sqrt{\frac{p}{n}}\vee \frac{p}{n}\right) 
    \left(\sqrt{\frac{{\bf r}(\Sigma)}{n^{(0)}}}\vee \frac{{\bf r}(\Sigma)}{n^{(0)}}\vee \sqrt{\frac{p}{n^{(0)}}}\vee \frac{p}{n^{(0)}}\right) 
 \\
 &
 +\|f^{(m)}\|_{{\rm Lip}_\rho(U)} \|\Sigma\|^s\left(\Bigl(\frac{{\bf r}(\Sigma)}{n^{(0)}}\Bigr)^{s/2}\vee \Bigl(\frac{{\bf r}(\Sigma)}{n^{(0)}}\Bigr)^s\vee \Bigl(\frac{p}{n^{(0)}}\Bigr)^{s/2}\vee 
 \Bigl(\frac{p}{n^{(0)}}\Bigr)^s   
\right)
 \\
 &
 + 
\left(\|f\|_{L_\infty}+M + \|f'(\Sigma)\| \|\Sigma\| \left( \sqrt{\frac{{\bf r}(\Sigma)}{n}}\vee \frac{{\bf r}(\Sigma)}{n}\vee \sqrt{\frac{p}{n}}\vee \frac{p}{n}\right) \right)
  \exp\Bigl\{-c_1 \frac{n^{(0)}}{p}\Bigl(\frac{\delta^2}{\|\Sigma\|^2}\wedge \frac{\delta}{\|\Sigma\|}\Bigr)\Bigr\}.
\end{align*}    
\end{theorem}

\begin{remark}
\normalfont
The problem of estimation of smooth functionals of unknown covariance matrix in Gaussian models was studied 
for estimators based on iterated bootstrap in \cite{Koltchinskii_2017, Koltchinskii_2018} in the case of functionals of the form ${\rm tr}(g(\Sigma)B),$
where $g$ is a smooth function in the real line and $B$ is a given matrix. This approach was further developed in \cite{Koltchinskii_Zhilova_19}
for general H\"older smooth functionals. In these papers, the problem was studied only the case of r.v. $X$ taking values in the Euclidean 
space ${\mathbb R}^d$ and under the assumption that the spectrum of unknown covariance is bounded from above and bounded away 
from zero by numerical constants. In this case, the effective rank of covariance is of the same order as the dimension of the space 
and there is no need to use it as a complexity parameter.
The justification of iterated bootstrap method in the case of covariance operators in infinite-dimensional spaces 
and for ``more singular" covariances in high-dimensional spaces, when the effective rank is relevant, is a challenging open problem that seems to be beyond the reach of analytic and 
probabilistic methods developed in \cite{Koltchinskii_2017, Koltchinskii_2018, Koltchinskii_Zhilova_19}. More recently, the problem was studied 
for functional estimators based on linear aggregation of plug-in estimators with different sample sizes \cite{Koltchinskii_2022} in the case 
of general Gaussian models in infinite-dimensional Banach space and it was shown that error rates with optimal dependence on the effective rank 
of the unknown covariance operator hold in this case. It is not clear, however, how to extend these results beyond the Gaussian models 
(even to the sugaussian case as in Theorem \ref{th_1_Bern_AA_covariance}) since 
the analysis in these papers heavily relied on Gaussian concentration inequalities. 
For estimators developed in the current paper, it was possible to circumvent these difficulties.
\end{remark}

\begin{remark}
\normalfont
For $a>0$ and $r\geq 1,$ let ${\mathcal P}_{a,r}$ be the set of distributions $P\in {\mathcal P}$ with covariance operator $\Sigma$ satisfying the conditions 
$\|\Sigma\|\leq a$
and ${\bf r}(\Sigma)\leq r.$ It follows from Theorem \ref{th_1_Bern_AA_covariance}(i) that for $r\lesssim n$ and $\delta\in (C'a\sqrt{\frac{r}{n}},1]$
for a sufficiently large constant $C'>0,$ 
\begin{align*}
 &\sup_{\|f\|_{C^s(U)}\leq 1}\sup_{P\in {\mathcal P}_{a,r}}\left\|\tilde{T}_f(X_1,\cdots,X_n)-f(\Sigma)\right\|_{L_2(\mathbb{P}_P)}
    \lesssim_{s, C} \frac{a}{\sqrt{n}}
   +\Bigl(a\sqrt{\frac{r}{n}}\Bigr)^{s}
   +\exp\Bigl\{-c_1n\Bigl(\frac{\delta^2}{a^2}\wedge \frac{\delta}{a}\Bigr)\Bigr\}.
\end{align*}
This bound essentially coincides with the bound on the $L_2$-error obtained in \cite{Koltchinskii_2022} in the Gaussian case for linear combinations of plug-in 
estimators with different samples sizes. Moreover, in \cite{Koltchinskii_2022}, such $L_2$-error rates were proved to be minimax optimal in the Gaussian case. 
However, the dependence on $p$ of the general bounds on the $L_p$-error of Theorem \ref{th_1_Bern_AA_covariance}(i) is a bit worse than what was proved 
in \cite{Koltchinskii_2022} in the Gaussian case. In particular, the result in \cite{Koltchinskii_2022} implied the bound on Orlicz $\psi_1$-norm error whereas 
Theorem \ref{th_1_Bern_AA_covariance}(i) implies only the bound on $\psi_{1/s}$-norm. 
\end{remark}

\section{Functional estimation in exponential families}
\label{exponential_section}
In this section, we consider an exponential family $\{P_{\theta}: \theta \in \Theta\}$ on a measurable space $(S,{\mathcal A})$ equipped 
with a measure $\mu.$ The densities $p_{\theta}:= \frac{dP_{\theta}}{d\mu}$ are of the form 
\begin{align*}
p_{\theta}(x):= \frac{1}{Z(\theta)} \exp\{\langle T(x), \theta\rangle\}, x\in S,
\end{align*}
where statistic $T:S\mapsto E$ takes values in a finite-dimensional Banach space $E$ and parameter $\theta$ takes values in the dual 
space $E^{\ast}.$ Since $E$ is finite-dimensional, both $E$ and $E^{\ast}$ could be identified with ${\mathbb R}^N$ for some $N\geq 1.$
However, in some examples (such as, for instance, exponential families with matrix parameter), it could be convenient to use different norms for the values 
of $T$ and for the parameter $\theta$ (for instance, the operator norm and the nuclear norm). The normalizing constant
\begin{align*} 
Z(\theta) := \int_S \exp\{\langle T(x), \theta\rangle\} \mu(dx), \theta \in E^{\ast}
\end{align*}
is a convex function and the set $\Theta:=\{\theta\in E^{\ast}:Z(\theta)<+\infty\}$ is convex. 
This set is a natural parameter space for exponential family. 

In what follows, we will assume that $\mu$ is a probability measure. We also assume that ${\rm Int}\Theta\neq \emptyset,$ 
that ${\rm l.s.}(\Theta)=E^{\ast}$ and that, for all $u\in E^{\ast}, c\in {\mathbb R},$ $\mu\{x: \langle T(x),u\rangle=c\}=0.$   
Under these assumptions, the exponential family is an identifiable statistical model. It is well known that more general 
exponential families could be always reduced to this form by a proper change of measure $\mu$ and by affine transformations of statistic $T(X)$ and parameter $\theta$ (see Section \ref{background} for more details on this and other properties of exponential families).
Also, under the above assumptions the function $\psi(\theta):= \log Z(\theta)$ is strictly convex in $\Theta$ and its gradient 
\begin{align*}
\Psi(\theta):= \psi'(\theta)=(\nabla \psi)(\theta)= {\mathbb E}_{\theta} T(X), \theta \in {\rm Int}\Theta
\end{align*}
is a strictly monotone mapping, so, it is one-to-one. Moreover, $\Psi$ is a $C^{\infty}$-diffeomorphism between ${\rm Int}\Theta$
and its image. The last property makes it natural to re-parametrize the model and to use $t=\Psi(\theta)={\mathbb E}_{\theta} T(X)\in E$  
as a new parameter. It is usually called {\it the mean parameter}, but we will often use the terminology of Chencov \cite{Chencov} and call 
it {\it the natural parameter} of exponential family.

It is also well known that 
\begin{align*}
\Sigma_{\theta}= \Psi'(\theta)= {\rm Cov}_{\theta}(T(X))= {\mathbb E}_{\theta}(T(X)-\Psi(\theta))\otimes (T(X)-\Psi(\theta))
\end{align*}
is the covariance operator of $T(X),$ acting from $E^{\ast}$ to $E.$ It is straightforward to check that $\Sigma_{\theta}$ coincides 
with the Fisher information $I(\theta)$ of exponential family $\{P_{\theta}:\theta \in {\rm Int}\Theta\}.$ On the other hand, the operator 
${\mathcal I}(t)= \Sigma_{\Psi^{-1}(t)}^{-1},$ acting from $E$ into $E^{\ast},$ coincides with the Fisher information of the model 
$\{P_{\Psi^{-1}(t)}: t\in \Psi({\rm Int}\Theta)\}$ with natural parametrization.

\begin{example}
\normalfont
Suppose $\mu$ is a Borel probability measure on ${\mathbb R}^N$ invariant with respect to the group of orthogonal transformations. 
In this case, $\psi(\theta)=\varphi(\|\theta\|),$ where $\varphi(\rho) :=\int_{{\mathbb R}^N} e^{\rho x_1} \mu(dx), \rho \geq 0.$
Assume that $\varphi (\rho)<\infty$ for some $\rho>0.$ It is easy to see that $\Theta$ is either the whole space ${\mathbb R}^N,$
or a ball of radius $r>0$ centered at $0$ (open or closed). Also, function $\varphi$ is strictly convex and $\Phi(\rho):= \varphi'(\rho)$ is strictly 
increasing. It is also easy to see that 
\begin{align*}
\Psi(\theta) = \Phi(\|\theta\|)\frac{\theta}{\|\theta\|}, \theta\in B(0,r)\ 
{\rm and}\ 
\Psi^{-1}(t) = \Phi^{-1}(\|t\|)\frac{t}{\|t\|}, t\in B(0, \Phi(r)),
\end{align*}
so, both $\Psi$ and $\Psi^{-1}$ are spherically symmetric and strictly monotone vector fields. 
A special case of this example is so called von Mises-Fisher distribution, an important model for directional data.
In this case, $\mu$ is a uniform distribution on $S^{N-1}.$
For more details, see Example \ref{example_spherical}.
\end{example}

We are interested in the problem of estimation of the value of a smooth functional $f(\theta)$ of parameter $\theta$ of exponential family 
based on i.i.d. observations $X_1,\dots, X_n\sim P_{\theta}.$ Our approach is based on reducing this problem to estimation of the value 
$f(\Psi^{-1}(t))$ of smooth functional of the corresponding natural parameter $t=\Psi(\theta).$ Since $\Psi$ is a $C^{\infty}$-diffeomorphism,  
the smoothness properties of functionals $f$ and $f\circ \Psi^{-1}$ are closely related (at least, at the points of the parameter space that are away 
from its boundary). To estimate $(f\circ \Psi^{-1})(t),$ we will use $\bar T_n := \frac{T(X_1)+\dots T(X_n)}{n}$ as a base estimator of 
$t={\mathbb E}_{\theta} T(X),$ and construct estimator $T_{f\circ \Psi^{-1}}(X_1,\dots, X_n)$ and its truncated version 
$\tilde T_{f\circ \Psi^{-1}}(X_1,\dots, X_n)$ using the sample split, as described in Section \ref{Main_results}.

\begin{example}
\normalfont
Consider the problem of estimation of the entropy 
\begin{align*}
H(\theta):= -\int_S p_{\theta} \log p_{\theta} d\mu
\end{align*}
of unknown distribution from exponential family $P_{\theta}, \theta \in \Theta$
based on i.i.d. observations $X_1,\dots, X_n\sim P_{\theta}.$
It is easy to check that 
$
H(\theta)= -\psi^{\ast}({\mathbb E}_{\theta} T(X))=-\psi^{\ast}(\Psi(\theta)),
$
where
\begin{align*}
\psi^{\ast}(t):= \sup_{\theta\in E^{\ast}}[\langle t, \theta\rangle- \psi(\theta)], t\in E,
\end{align*}
is the Legendre transform (or Young--Fenchel conjugate) of convex function $\psi(\theta)=\log Z(\theta).$ Note also 
that 
\begin{align*}
\psi^{\ast}(t) = \langle t, \Psi^{-1}(t)\rangle - (\psi\circ \Psi^{-1})(t).
\end{align*}
Thus, the problem of estimation of $H(\theta)$ could be reduced to the problem of estimation of function $-\psi^{\ast}(t)$ of natural 
parameter $t.$ For regular exponential models, $\psi^{\ast}$ is an infinitely differentiable functional (see Section \ref{background} for more detail), so, in principle, one can use estimator $\tilde T_{-\psi^{\ast}}(X_1,\dots, X_n),$ based on the Taylor expansion of an arbitrary 
order $m.$ The quality of such estimators would depend on the H\"older $C^s$-norms of function $\psi^{\ast}$ (as in the bounds of Theorem
\ref{estim_exp_fam_simple} below). 
\end{example}

Suppose there exists a finite subset ${\mathcal M}$ of $E^{\ast}$ such that $\|u\|\leq 1, u\in {\mathcal M},$ $\log {\rm card}({\mathcal M})\leq d$ for some parameter $d\geq 1$
and, for some constant $C>0,$ 
\begin{align}
\label{norm_discrete}
\|x\|\leq C\max_{u\in {\mathcal M}}|\langle x,u\rangle|.
\end{align} 

\begin{remark}
\normalfont
Note that, if $E$ is a Banach space of dimension $d,$ then the dimension of $E^{\ast}$ is also $d.$ In this case, one can choose 
set ${\mathcal M}$ as a $1/2$-net for the unit ball of $E^{\ast}$ of cardinality $\leq 5^d. $ It is easy to see that condition \eqref{norm_discrete}
holds with $C=2$ and $\log({\rm card}({\mathcal M}))\leq (\log 5) d.$ 

On the other hand, if $E$ is the space of $d\times d$ symmetric matrices equipped 
with the operator norm, one can choose ${\mathcal M}:= \{u\otimes u: u\in {\mathcal N}\},$ where ${\mathcal N}$ is a $1/4$-net of the unit 
ball in $\ell_2^d$ of cardinality $\leq 9^d.$ In this case, condition \eqref{norm_discrete} still holds with $C=2$ and 
$\log({\rm card}({\mathcal M}))\leq (2\log 9) d$ (note that ${\rm dim}(E)=\frac{d(d+1}{2}$). This approach is useful in the following 
example of exponential family with matrix parameter. 
\end{remark}

\begin{example}
\normalfont
Consider the following exponential model 
\begin{align*}
P_{\theta}(dx) = \frac{1}{Z(\theta)} \exp\{\langle x\otimes x, \theta\rangle\}\mu(dx), x\in {\mathbb R}^d, 
\end{align*}
where the parameter $\theta$ is a symmetric $d\times d$ matrix and $\mu$ is a Borel probability measure in 
${\mathbb R}^d.$ The statistic generating this exponential family is $T(x)= x\otimes x.$ It also takes values in the 
space of symmetric matrices. This space will be equipped 
with the operator norm and the resulting Banach space will be denotes by $E.$ Its dual space $E^{\ast}$ is again the linear space of 
symmetric $d\times d$ matrices equipped with the nuclear norm.  With this conventions, statistic $T(x)=x\otimes x$ takes values in $E$
and the parameter space $\Theta=\{\theta: Z(\theta)<+\infty\}$ is a convex subset of $E^{\ast}.$ We assume that the functions 
$1, x_i x_j: 1\leq i\leq j\leq d$ are linearly independent on ${\mathbb R}^d$ and that 
${\rm l.s.}(\Theta)=E^{\ast}.$ Under these assumptions, the model is identifiable. 
Note that 
$
\Psi (\theta)={\mathbb E}_{\theta}(X\otimes X), X\sim P_{\theta}, \theta \in \Theta,
$
which could be estimated by the sample covariance $\bar T_n := n^{-1}\sum_{j=1}^n X_j\otimes X_j.$

In addition to mean zero normal model (with non-standard parametrization), this class of exponential families 
includes Bingham distribution for directional data (with $\mu$ being the uniform distribution on the unit sphere $S^{d-1}$)
and Ising model (with $\mu$ being the uniform distribution on the binary cube $\{-1,1\}^d$).
For these two models, the support of measure $\mu$ is bounded, which implies that $\Theta=E^{\ast}$ (this, of course, could be the case even if the support of measure $\mu$ is not bounded). See Example \ref{matrix_exponential} for more detail. 
\end{example}

Note that, for $\theta\in {\rm Int}\Theta,$  $\sigma_{\rm max}^2(\theta):=\|\Sigma_{\theta}\|$ is the largest eigenvalue of the covariance $\Sigma_{\theta}$ and $\sigma_{\rm min}^2(\theta):=\frac{1}{\|\Sigma_{\theta}^{-1}\|}$ is its smallest eigenvalue. 
Let $W\subset {\rm Int}\Theta$ be an open set and let $G:=\Psi(W).$
Denote 
\begin{align*}
\sigma_{\rm max}^2 (W):= \sup_{\theta\in W}\sigma_{\rm max}^2(\theta), \ \ \sigma_{\rm min}^2 (W):= \inf_{\theta\in W}\sigma_{\rm min}^2(\theta).
\end{align*}

\begin{theorem}
\label{estim_exp_fam_simple}
Let $f:E^{\ast}\mapsto {\mathbb R}$ be a uniformly bounded functional.
Suppose that, for some $m\geq 2,$ $f\circ \Psi^{-1}$ is $m$ times continuously differentiable in open set $G$ 
and, moreover,
$(f\circ \Psi^{-1})^{(m)}$ satisfies the Hölder condition with exponent $\rho\in(0,1]$ in $G.$ 
Let numbers $\delta$ and $r$ be sufficiently large to satisfy the conditions
\begin{align*}
r\geq \frac{1}{\sigma_{\rm max}(W)}\sqrt{\frac{d}{n^{(0)}}}\ {\rm and}\ C'C\sigma_{\rm max}(W)\sqrt{\frac{d}{n^{(0)}}}\leq \delta \leq 1
\end{align*}
for a large enough numerical constant $C'>0.$ Let $U\subset G$ and $V:= \psi^{-1}(U)\subset W$ be such that, for all $t\in U,$ $B(t,\delta)\subset G$
and, for all $\theta\in V,$ $B(\theta,r)\subset W.$

(i) Assume that $n^{(0)}\asymp_m n$ and $n_j^{(k)}\asymp_m n$, $j=1,\cdots,k$, $k=1,\cdots,m.$ 
Then, for $M=\|f\|_{L_{\infty}},$ estimator $\tilde{T}_{f\circ \Psi^{-1}}= \tilde{T}_{f\circ \Psi^{-1}, M}$ satisfies the bound 
\begin{align*}
   &
   \sup_{\theta\in V} \left\|\tilde{T}_{f\circ \Psi^{-1}}(X_1,\cdots,X_n)-f(\theta)\right\|_{L_2(\mathbb{P}_\theta)}
   \lesssim_{s, C} 
  \|f\circ \Psi^{-1}\|_{C^s(G)} \Bigl[\frac{\sigma_{\rm max}(W)}{\sqrt{n}}
 \bigvee 
  \sigma_{\rm max}^s(W)\Bigl(\frac{d}{n}\Bigr)^{s/2}\Bigr].
 \end{align*}

 (ii) Assume that, for all $j=1,\dots, k, k=1,\dots, m,$ $n_j^{(k)}\asymp_m n.$
 Then, for $M=\|f\|_{L_{\infty}}+ \|f\|_{\rm Lip},$ estimator $\tilde{T}_{f\circ \Psi^{-1}}= \tilde{T}_{f\circ \Psi^{-1}, M}$ satisfies the bound 
\begin{align*}
    &
    \sup_{\theta \in V}\left\|\tilde T_{f\circ \Psi^{-1}}(X_1,\cdots,X_n)-f(\theta)- \langle \bar T_1^{(1)}-\Psi(\theta), (f\circ \Psi^{-1})'(t)\rangle\right\|_{L_2(\mathbb{P}_{\theta})}
    \\
    &
 \lesssim_{s, C} \|f\circ \Psi^{-1}\|_{C^{s}(G)} 
\Bigl[\sigma_{\rm max}^2(W) 
\frac{1}{\sqrt{n}}
\sqrt{\frac{d}{n^{(0)}}}
\vee \sigma_{\rm max}^s(W)\Bigl(\frac{d}{n^{(0)}}\Bigr)^{s/2}\Bigr].
\end{align*}    
\end{theorem}

The bounds of Theorem \ref{estim_exp_fam_simple} follow from more general bounds on the $L_p$-errors 
of estimators $\tilde T_{f\circ \Psi^{-1}}(X_1,\cdots,X_n)$ presented and proved 
in Section \ref{exp_fam_det}.

\vskip 2mm

\begin{corollary}
\label{cor_1_estim_func_exp}
Suppose the assumptions of Theorem \ref{estim_exp_fam_simple} (i) hold. 
Then, for $s>2,$
\begin{align*}
\sup_{\|f\circ \Psi^{-1}\|_{C^s(G)}\leq 1}\sup_{\theta \in V} \|\tilde{T}_{f\circ \Psi^{-1}}(X_1,\cdots,X_n) -f(\theta)\|_{L_2({\mathbb P}_{\theta})}\lesssim_{s,C}
\Bigl[\frac{\sigma_{\rm max}(W)}{\sqrt{n}}
 \bigvee 
  \sigma_{\rm max}^s(W)\Bigl(\frac{d}{n}\Bigr)^{s/2}\Bigr]\wedge 1.
\end{align*}
For $s\leq 2,$ the same bound holds for the plug-in estimator $(f\circ \Psi^{-1})(\bar T_n).$
\end{corollary}

The following result shows the minimax optimality of the upper bound of Corrolary \ref{cor_1_estim_func_exp} in the case of exponential families in the Euclidean space  
$E=E^{\ast}={\mathbb R}^d,$ provided that $\sigma_{\rm max}(W)\asymp \sigma_{\rm min}(W).$
 
\begin{proposition}
\label{prop_low_exp}
Let $s>0,$ let $\gamma>0$ be a sufficiently small constant and suppose that open set $G$ contains a ball $B_{\ell_{\infty}}(t_0,\delta)$
of radius $\delta \geq \gamma \sigma_{\rm min}(W) n^{-1/2}.$
Then 
\begin{align*}
\sup_{\|f\circ \Psi^{-1}\|_{C^s(G)}\leq 1}\inf_{\hat T_n}\sup_{\theta \in G} \|\hat T_n -f(\theta)\|_{L_2({\mathbb P}_{\theta})}\gtrsim 
\Bigl[\frac{\sigma_{\rm min}(W)}{\sqrt{n}}
 \bigvee 
  \sigma_{\rm min}^s(W)\Bigl(\frac{d}{n}\Bigr)^{s/2}\Bigr]\wedge 1,
\end{align*}
where the infimum is over all estimators $\hat T_n(X_1,\dots, X_n)$ of functional $f(\theta).$
\end{proposition}

The proof is given in Section \ref{sec:lower_bounds}.

Denote by 
\begin{align*}
\kappa (V) := \sup_{\theta\in V}\sup_{\|u\|\leq 1}
\frac{\|\langle T(X)-{\mathbb E}_{\theta} T(X),u\rangle\|_{L_4({\mathbb P}_{\theta})}^4}{\|\langle T(X)-{\mathbb E}_{\theta} T(X), u\rangle\|_{L_2({\mathbb P}_{\theta})}^4}
\end{align*}
the maximal kurtosis of r.v. $\langle T(X)-{\mathbb E}_{\theta} T(X),u\rangle, \|u\|\leq 1, \theta \in V.$
It is not hard to check that 
\begin{align*}
\kappa(V) \lesssim \sup_{\theta\in V} \sup_{u\neq 0} \frac{|\psi^{(4)}(\theta)[u,u,u,u]|}{\psi''(\theta)[u,u]^2} +1 \lesssim \frac{\|\psi^{(4)}\|_{L_{\infty}(V)}}{\sigma_{\rm min}^4(V)}+1.
\end{align*}

\begin{corollary}
\label{cor_asymp_norm}
Suppose the assumptions of Theorem \ref{estim_exp_fam_simple} (ii) hold and, in addition, $n^{(0)}\asymp \frac{n}{\log n}.$
Suppose that $\sigma_{\rm max}(W)\lesssim 1$ and $\kappa(V)\lesssim 1.$
Also suppose that $d\leq n^{\alpha}$ for some $\alpha\in (0,1)$ and $s>\frac{1}{1-\alpha}.$ 
Then    
\begin{align}
\label{asymp_normal_exp}
\sup_{\|f\circ \Psi^{-1}\|_{C^s(G)}\leq 1}\sup_{\theta\in V}
W_{2,{\mathbb P}_{\theta}}\Bigl(\sqrt{n}(\tilde{T}_{f\circ \Psi^{-1}}(X_1,\cdots,X_n)-f(\theta)), \sigma_{f,\Psi}(\Psi(\theta))Z\Bigr)\to 0\ {\rm as}\ n\to\infty
\end{align}
and 
\begin{align}
\label{mean_square_exp}
\sup_{\|f\circ \Psi^{-1}\|_{C^s(G)}\leq 1}\sup_{\theta\in V}
\Bigl|\sqrt{n}\Bigl\|\tilde{T}_{f\circ \Psi^{-1}}(X_1,\cdots,X_n)-f(\theta)\Bigr\|_{L_2({\mathbb P}_{\theta})}- \sigma_{f,\Psi}(\Psi(\theta))\Bigr|\to 0\ {\rm as}\ n\to\infty,
\end{align}
where 
\begin{align*}
\sigma_{f,\Psi}^2(t):= \langle \Sigma_{\Psi^{-1}(t)} (f\circ \Psi^{-1})'(t), (f\circ \Psi^{-1})'(t)\rangle, t\in U.
\end{align*}
\end{corollary}

The proof is given in Section \ref{sec_L_p_exp}.

\begin{remark}
\normalfont
In \cite{Koltchinskii_2022}, a result similar to Corollary \ref{cor_asymp_norm} was proved for functional estimators based on iterated 
bootstrap, but only in the case of {\it log-conave} exponential families (that is, when measure $\mu$ is log-concave). Its extension beyond 
the log-concave case seems impossible, at least with the methods developed in \cite{Koltchinskii_2022}.
\end{remark}

\begin{remark}
\normalfont
Note that 
\begin{align*}
\sigma_{f,\Psi}^2(t)= \langle {\mathcal I}(t)^{-1} (f\circ \Psi^{-1})'(t),(f\circ \Psi^{-1})'(t)\rangle,
\end{align*}
where ${\mathcal I}(t)$ is the Fisher information for the model $\{P_{\Psi^{-1}(t)}: t\in G\}$ with the natural parameter $t.$
One can apply the local minimax bound of Theorem \ref{vanTrees_bd} to this model along with \eqref{mean_square_exp}
to show that estimator $\tilde{T}_{f\circ \Psi^{-1}}(X_1,\cdots,X_n)$ is locally asymptotically minimax provided that $d\leq n^{\alpha}$
for some $\alpha\in [1/2,1)$ and $s>\frac{1}{1-\alpha}$ (for $\alpha <1/2$ and $s>\frac{1}{1-\alpha},$ the plug-in estimator $(f\circ \Psi^{-1})(\bar T_n)$
is locally asymptotically minimax). 
\end{remark}

\section{Upper bounds: further details and proofs}
\label{technical}

The proofs of the results of Section \ref{Main_results} easily follow from more general and somewhat technical facts stated and proved below.
For $k=1,\dots, m,$ we assume that estimators $\hat \theta^{(0)}, \hat \theta_j^{(k)}, j=1,\dots, k$ of parameter $\theta=\theta(P),$ based on i.i.d. observations $X_1,\dots, X_n\sim P\in {\mathcal P},$ are independent r.v.
We will use the following notations:
\begin{align*}
A_p(P,\delta):=\max_{1\leq k\leq m}\max_{1\leq j\leq k}\sup_{\|u\|\le 1}\Bigl\|\langle \hat \theta_j^{(k)}-\theta(P),u\rangle I(\|\hat \theta_j^{(k)}-\theta(P)\|<\delta)\Bigr\|_{L_p({\mathbb P}_P)}
\end{align*}
and 
\begin{align*}
B_p(P,\delta):=\max_{1\leq k\leq m}\max_{1\leq j\leq k}
\Bigl\|\|\hat \theta_j^{(k)}-\theta(P)\| I(\|\hat \theta_j^{(k)}-\theta(P)\|<\delta)\Bigr\|_{L_p({\mathbb P}_P)}.
\end{align*}
We will also denote 
\begin{align*}
\tilde \beta_p(P, \delta):= \Bigl\|\|\hat \theta^{(0)}-\theta(P)\|I(\|\hat \theta^{(0)}-\theta(P)\|<\delta)\Bigr\|_{L_p({\mathbb P}_P)}.
\end{align*}
For $\delta=+\infty,$ we set $A_p(P):=A_p(P,+\infty),$ $B_p(P):=B_p(P,+\infty)$ and $\tilde \beta_p(P):=\tilde \beta_p(P, +\infty)$
Note that 
\begin{align*}
A_p(P,\delta) \leq A_p(P)\wedge \delta,\ \  B_p(P,\delta) \leq B_p(P)\wedge \delta,\ \ \tilde \beta_p(P, \delta)\leq \tilde \beta_p(P)\wedge \delta.
\end{align*}
Based on estimators $\hat \theta^{(0)}, \hat \theta_j^{(k)}, j=1,\dots, k, k=1,\dots, m,$ we define estimator $T_f(X_1,\dots, X_n)$ by \eqref{basic_T_f}
and 
estimator $\tilde T_{f}(X_1,\dots, X_n)=\tilde T_{f,M}(X_1,\dots, X_n)$ by \eqref{basic_T_fM}.

The following result will be proved (which implies Theorem \ref{th_1_asssume_AAA}).

\begin{theorem}
\label{th_1}
Let $s:=m+\rho$ for some $m\geq 2$ and $\rho\in (0,1],$ and suppose that $f'\in C^{s-1}(E).$
Then 
\begin{align*}
    &\left\|T_f(X_1,\cdots,X_n)-f(\theta(P))\right\|_{L_p(\mathbb{P}_P)}
    \\
    &
    \leq
\sum_{k=1}^m \frac{2^{k-2}\|f^{(k)}\|_{L_{\infty}}}{(k-1)!} A_p(P) \left(B_p^{k-1}(P)+\tilde \beta_{p(k-1)}^{k-1}(P)\right)
+ \frac{1}{m!}\|f^{(m)}\|_{{\rm Lip}_{\rho}} \tilde \beta_{ps}^s (P).
\end{align*}
Moreover, 
\begin{align*}
    &\left\|T_f(X_1,\cdots,X_n)-f(\theta(P))- \langle \hat \theta_1^{(1)}-\theta(P), f'(\theta(P))\rangle\right\|_{L_p(\mathbb{P}_P)}
    \\
    &
\leq \|f''\|_{L_{\infty}}A_{p}(P) \tilde \beta_p(P)+
\sum_{k=2}^m \frac{2^{k-2}\|f^{(k)}\|_{L_{\infty}}}{(k-1)!} A_p(P) \left(B_p^{k-1}(P)+\tilde \beta_{p(k-1)}^{k-1}(P)\right)
\\
&
+ \frac{1}{m!}\|f^{(m)}\|_{{\rm Lip}_{\rho}} \tilde \beta_{ps}^s (P).
\end{align*}
\end{theorem}

For truncated estimator $\tilde T_f(X_1,\dots, X_n),$ we will prove the following local version of the previous result, which implies Theorem \ref{th_1_AA_asssume_AAA}.

\begin{theorem}
\label{th_1_AA}
Let $f:E\mapsto {\mathbb R}$ be a uniformly bounded functional and let $\theta=\theta(P)$ for some $P\in {\mathcal P}.$
Suppose $f$ is $m$ times Fr\`echet continuously differentiable in $U:=B(\theta;\delta)$ for some $\delta\in (0,1]$ and, moreover,
$f^{(m)}$ satisfies the Hölder condition with exponent $\rho\in(0,1]$ in $U.$ 
Then, for $M\geq \|f\|_{L_{\infty}}$ and for all $p'\geq p,$ 
\begin{align*}
    &
    \left\|\tilde{T}_f(X_1,\cdots,X_n)-f(\theta(P))\right\|_{L_p(\mathbb{P}_P)}
    \\
    &
    \leq  
\sum_{k=1}^m \frac{2^{k-2}\|f^{(k)}\|_{L_{\infty}(U)}}{(k-1)!} A_p(P,\delta) \left(B_p^{k-1}(P,\delta)+\tilde \beta_{p(k-1)}^{k-1}(P,\delta)\right)
+ \frac{1}{m!}\|f^{(m)}\|_{{\rm Lip}_{\rho}(U)} \tilde \beta_{ps}^s (P,\delta)    
\\
&
+(\|f\|_{L_{\infty}}+M)\biggl(\left(\frac{m(m+1)}{2}\right)^{1/p}\biggl(\frac{B_{p'}(P)}{\delta}\biggr)^{p'/p}+ \biggl(\frac{\tilde \beta_{p'}(P)}{\delta}\biggr)^{p'/p}\biggr).
\end{align*}
Moreover, if $M\geq \|f\|_{\infty}+ \|f\|_{{\rm Lip}}\delta,$ then, for all $p'\geq p,$
\begin{align*}
    &\left\|\tilde T_f(X_1,\cdots,X_n)-f(\theta(P))- \langle \hat \theta_1^{(1)}-\theta(P), f'(\theta(P))\rangle\right\|_{L_p(\mathbb{P}_P)}
   \\
   &
   \leq  \|f''\|_{L_{\infty}(U)}A_{p}(P, \delta) \tilde \beta_p(P,\delta)
   \\
   &
   +
\sum_{k=2}^m \frac{2^{k-2}\|f^{(k)}\|_{L_{\infty}(U)}}{(k-1)!} A_p(P,\delta) \left(B_p^{k-1}(P,\delta)+\tilde \beta_{p(k-1)}^{k-1}(P,\delta)\right)    
\\
&
  + \frac{1}{m!}\|f^{(m)}\|_{{\rm Lip}_{\rho}(U)} \tilde \beta_{ps}^s (P,\delta)
  \\
  &
 +C_m (\|f\|_{L_{\infty}}+\|f'\|_{L_{\infty}}+M)(1+B_p(P))
\biggl(\biggl(\frac{B_{p'}(P)}{\delta}\biggr)^{p'/p}+ \biggl(\frac{\tilde \beta_{p'}(P)}{\delta}\biggr)^{p'/p}\biggr),
 \end{align*}
where $C_m := \frac{m(m-1)}{2}\vee \left(\frac{m(m+1)}{2}\right)^{1/p}+1.$
\end{theorem}

We now turn to the proof of theorems \ref{th_1} and \ref{th_1_AA}.

\begin{proof}
The proofs rely on the following simple lemmas. 

\begin{lemma}
\label{Prop_1A}
Let $M[x_1,\cdots,x_k]$ be a bounded symmetric $k$-linear form on $E$ and $\xi_1,\cdots,\xi_k$ be independent random variables in $E$. The following bound holds for all $p\ge 1$ and $\delta>0$:
\begin{align*}
&
\mathbb{E}\left|M[\xi_1,\cdots,\xi_k]I\left(\max_{1\le j\le k}\|\xi_j\|<\delta\right)\right|^p
\le\|M\|^p \sup_{\|u\|\le 1}\mathbb{E}|\langle \xi_k,u\rangle|^p I(\|\xi_k\|<\delta)
\prod_{j=1}^{k-1}\mathbb{E}\|\xi_j\|^p I(\|\xi_j\|<\delta).
\end{align*}
\end{lemma}

\begin{proof}
Define a $(k-1)$-linear form $\tilde M[x_1,\dots,x_{k-1}]$ on $E$ with values in $E^{\ast}$ as follows:
\begin{align*}
\langle x, \tilde M[x_1,\dots,x_{k-1}]\rangle := M[x_1,\dots, x_{k-1}, x], x\in E.
\end{align*}
Clearly, 
\begin{align*}
\|\tilde M\|
&= \sup_{\|x_1\|,\dots, \|x_{k-1}\|\leq 1}\|\tilde M[x_1,\dots,x_{k-1}]\| = \sup_{\|x_1\|,\dots, \|x_{k-1}\|\leq 1}\sup_{\|x\|\leq 1}\langle x, \tilde M[x_1,\dots,x_{k-1}]\rangle
\\
&
=\sup_{\|x_1\|,\dots, \|x_{k-1}\|, \|x\|\leq 1}M[x_1,\dots, x_{k-1}, x]= \|M\|.
\end{align*}
Thus, we have 
    \begin{align*}
        &\mathbb{E}\left|M[\xi_1,\cdots,\xi_k]I\left(\max_{1\le j\le k}\|\xi_j\|< \delta\right)\right|^p\\&=\mathbb{E}\mathbb{E}_{\xi_k}\left|\left\langle\xi_k,\tilde{M}[\xi_1,\cdots,\xi_{k-1}]\right\rangle I\left(\max_{1\le j\le k}\|\xi_j\|< \delta\right)\right|^p\\
        &\le \sup_{\|u\|\le 1}\mathbb{E}|\langle\xi_k,u\rangle|^p I(\|\xi_k\|< \delta)\cdot\mathbb{E}
        \left\|\tilde{M}[\xi_1,\cdots,\xi_{k-1}]\right\|^pI\left(\max_{1\le j\le k-1}\|\xi_j\|< \delta\right)\\
        &\le \sup_{\|u\|\le 1}\mathbb{E}|\langle\xi_k,u\rangle|^p I(\|\xi_k\|< \delta)\cdot\|\tilde{M}\|^p\mathbb{E}\left(\prod_{j=1}^{k-1}\|\xi_j\|^p\right)I\left(\max_{1\le j\le k-1}\|\xi_j\|<\delta\right)\\
        &\le \|M\|^p \sup_{\|u\|\le 1}\mathbb{E}|\langle\xi_k,u\rangle|^p I(\|\xi_k\|< \delta)\prod_{j=1}^{k-1}\mathbb{E}\|\xi_j\|^pI\left(\|\xi_j\|< \delta\right),       
    \end{align*}
implying the claim of the proposition.
\qed
\end{proof}

\begin{lemma}
\label{Prop_2A}
Let $M[x_1,\cdots,x_k]$ be a bounded symmetric $k$-linear form  on $E$. Let $\xi_1,\cdots,\xi_k$ be independent random variables in $E$ and let $w\in E$. For $p\ge 1$, denote
\begin{align*}
A:=\max_{1\le j\le k}\sup_{\|u\|\le 1}\|\langle \xi_j-w,u\rangle I(\|\xi_j-w\|<\delta)\|_{L_p}
\end{align*}
and 
\begin{align*}
B:=\max_{1\le j\le k}\Big\|\|\xi_j-w\|I(\|\xi_j-w\|<\delta)\Big\|_{L_p}.
\end{align*}
The following bound holds for all $\delta>0$:
\begin{align*}
    &\left\|\left(M[\xi_1,\cdots,\xi_k]-M[w,\cdots,w]\right)I\left(\max_{1\le j\le k}\|\xi_j-w\|< \delta\right)\right\|_{L_p}
    \leq k \|M\| A (B+\|w\|)^{k-1}.
    \end{align*}
\end{lemma}

\begin{proof}
We have 
\begin{align*}
M[\xi_1,\dots, \xi_k] = \sum_{l=0}^k \sum_{0\leq j_1\dots<j_l\leq k} M[\xi_{j_1}-w,\dots, \xi_{j_k}-w, w,\underset{k-l}{\dots}, w].
\end{align*}
Denote 
\begin{align*}
M_{w,l} [x_1,\dots, x_l]:= M[x_1,\dots, x_l, w,\underset{k-l}{\dots}, w].
\end{align*}
Then 
\begin{align*}
M[\xi_1,\dots, \xi_k]-M[w,\dots, w] 
= \sum_{l=1}^k \sum_{1\leq j_1\dots<j_l\leq k} 
M_{w,l}[\xi_{j_1}-w,\dots, \xi_{j_k}-w]
\end{align*}
and also, for an $l$-linear form $M_{w,l},$ we have 
\begin{align*}
\|M_{w,l}\|\leq \|M\|\|w\|^{k-l}.
\end{align*}
Therefore,
\begin{align*}
        &\Big|M[\xi_1,\cdots,\xi_k]-M[w,\cdots,w]\Big|I\left(\max_{1\le j\le k}\|\xi_j-w\|< \delta\right)\\
        &=\left|\sum_{l=1}^k\sum_{1\le j_1<\cdots<j_l\le k}M_{w,l}[\xi_{j_1}-w,\cdots,\xi_{j_l}-w]\right|I\left(\max_{1\le j\le k}\|\xi_j-w\|<\delta\right)\\
        &\le\sum_{l=1}^k\sum_{1\le j_1<\cdots<j_l\le k}\Big|M_{w,l}[\xi_{j_1}-w,\cdots,\xi_{j_l}-w]\Big|I\left(\max_{1\le j\le k}\|\xi_j-w\|<\delta\right)\\
        &\le\sum_{l=1}^k\sum_{1\le j_1<\cdots<j_l\le k}\Big|M_{w,l}[\xi_{j_1}-w,\cdots,\xi_{j_l}-w]\Big|I\left(\max_{1\le s\le l}\|\xi_{j_s}-w\|<\delta\right).
\end{align*}
Using the bound of Proposition \ref{Prop_1A}, we get
\begin{align*}
    \left\|M_{w,l}[\xi_{j_1}-w,\cdots,\xi_{j_l}-w]I\left(\max_{1\le s\le l}\|\xi_{j_s}-w\|< \delta\right)\right\|_{L_p}\le \|M\|\|w\|^{k-l} B^{l-1}.
\end{align*}
Therefore,
 \begin{align*}
        &\left\|\left(M[\xi_1,\cdots,\xi_k]-M[w,\cdots,w]\right)I\left(\max_{1\le j\le k}\|\xi_j-w\|< \delta\right)\right\|_{L_p}
        \\
        &
        \le\|M\| A\sum_{l=1}^k\binom{k}{l}B^{l-1}\|w\|^{k-l} = \|M\| \frac{A}{B} \sum_{l=1}^k\binom{k}{l}B^{l}\|w\|^{k-l} 
        \\
        &
        = \|M\| \frac{A}{B} ((B + \|w\|)^{k} -\|w\|^k)\leq \|M\| \frac{A}{B}\ k (B+\|w\|)^{k-1} B 
        \\
        &
        \leq k \|M\| A (B+\|w\|)^{k-1}.
\end{align*}
    
 \qed   
\end{proof}

Consider now the problem of estimation of the value of a smooth functional $f$ for a value of the parameter $\theta=\theta(P)$ in a neighborhood of a given point $t\in E.$
Namely, we will assume that both $\theta$ and $t$ are in an open ball $U.$  
Suppose, for each $k=1,\dots, m,$ we are in a possession of {\it independent} estimators $\hat \theta_j^{(k)}, j=1,\dots, k$ of parameter $\theta(P)$ based on i.i.d. observations 
$X_1,\dots, X_n\sim P\in {\mathcal P}.$ Define 
\begin{align*}
T_f(t;X_1,\dots, X_n) :=  \sum_{k=0}^m \frac{f^{(k)}(t)[\hat \theta_1^{(k)}-t,\dots, \hat \theta_k^{(k)}-t]}{k!}.
\end{align*}

\begin{lemma}
\label{lm_on_T_f_t}
Under the above assumptions and notations, the following bound holds:
\begin{align}
\label{bd_on_T_f_t}
&
\nonumber
\left\|(T_f(t;X_1,\dots, X_n)-f(\theta(P)))I\left(\max_{1\leq k\leq m}\max_{1\le j\le k}\|\hat{\theta}_j^{(k)}-\theta(P)\|< \delta\right)\right\|_{L_p({\mathbb P}_P)} 
\\
&
\leq
\sum_{k=1}^m \frac{\|f^{(k)}(t)\|}{(k-1)!} A_p(P,\delta) (B_p(P,\delta)+\|t-\theta\|)^{k-1} + \frac{1}{m!}\|f^{(m)}\|_{{\rm Lip}_{\rho}(U)}\|t-\theta(P)\|^{s}.
\end{align}
Moreover,
\begin{align}
\label{bd_on_T_f_t_lin}
&
\nonumber
\left\|(T_f(t;X_1,\dots, X_n)-f(\theta(P))-\langle \hat \theta_1^{(1)}-\theta(P), f'(\theta(P))\rangle)I\left(\max_{1\leq k\leq m}\max_{1\le j\le k}\|\hat{\theta}_j^{(k)}-\theta(P)\|< \delta\right)\right\|_{L_p({\mathbb P}_P)} 
\\
&
\leq 
\|f''\|_{L_{\infty}(U)} A_p(P,\delta) \|t-\theta(P)\|
+\sum_{k=2}^m \frac{\|f^{(k)}(t)\|}{(k-1)!} A_p(P,\delta) (B_p(P,\delta)+\|t-\theta(P)\|)^{k-1} 
\\
&
\nonumber
+ \frac{1}{m!}\|f^{(m)}\|_{{\rm Lip}_{\rho}(U)}\|t-\theta(P)\|^{s}.
\end{align}
\end{lemma}

\begin{proof}
Note that, for $\theta = \theta(P)$ by the Taylor expansion,
\begin{align*}
f(\theta)= \sum_{k=0}^m \frac{f^{(k)}(t)[\theta-t,\dots, \theta-t]}{k!} + R_m,
\end{align*}
with the remainder $R_m$ satisfying the bound:
\begin{align}
\label{bd_R_k_0}
|R_m| \leq \frac{1}{m!}\|f^{(m)}\|_{{\rm Lip}_{\rho}(U)}\|\theta-t\|^{s}, s=m+\rho.
\end{align}
For estimator $T_f(t;X_1,\dots, X_n),$ we have 
\begin{align*}
&
\nonumber
T_f(t;X_1,\dots, X_n)-f(\theta)
\\
&
= \sum_{k=1}^m \Bigl(\frac{f^{(k)}(t)[\hat \theta_1^{(k)}-t,\dots, \hat \theta_k^{(k)}-t]}{k!}-\frac{f^{(k)}(t)[\theta-t,\dots, \theta-t]}{k!}\Bigr) -R_m,
\end{align*}
implying that, for $L_p= L_p({\mathbb P}_P),$
\begin{align}
\label{T_f-f_bd}
&
\nonumber
\left\|(T_f(t;X_1,\dots, X_n)-f(\theta))\left(\max_{1\leq k\leq m}\max_{1\le j\le k}\|\hat{\theta}_j^{(k)}-\theta\|< \delta\right)\right\|_{L_p} 
\\
&
\leq
\sum_{k=1}^m \frac{1}{k!}
\left\|f^{(k)}(t)[\hat \theta_1^{(k)}-t,\dots, \hat \theta_k^{(k)}-t]-f^{(k)}(t)[\theta-t,\dots, \theta-t]I\left(\max_{1\le j\le k}\|\hat{\theta}_j^{(k)}-\theta\|< \delta\right)\right\|_{L_p} 
\\
&
\nonumber
\ \ \ \ \ \ \ +|R_m|.
\end{align}
We will apply the bound of Lemma \ref{Prop_2A} to the $k$-linear form $f^{(k)}(t)$ and to r.v. $\xi_j := \hat \theta_j^{(k)}-t, j=1,\dots, k$
with $w:= \theta-t.$ We then have 
\begin{align*}
A:=A_k&:=\max_{1\le j\le k}\sup_{\|u\|\le 1}\|\langle \xi_j-w,u\rangle I(\|\xi_j-w\|<\delta)\|_{L_p}
\\
&
=\max_{1\leq j\leq k}
\sup_{\|u\|\le 1}\|\langle \hat \theta_j^{(k)}-\theta,u\rangle I(\|\hat \theta_j^{(k)}-\theta\|<\delta)\|_{L_p}
\end{align*}
and
\begin{align*}
B:=B_k&:=\max_{1\le j\le k}\|\|\xi_j-w\| I(\|\xi_j-w\|<\delta)\|_{L_p}
\\
&
=\max_{1\leq j\leq k}
\|\|\hat \theta_j^{(k)}-\theta\| I(\|\hat \theta_j^{(k)}-\theta\|<\delta)\|_{L_p},
\end{align*}
and the bound of Lemma \ref{Prop_2A} yields
\begin{align*}
&\left\|\left(f^{(k)}(t)[\hat{\theta}_1^{(k)}-t,\cdots,\hat{\theta}_k^{(k)}-t]-f^{(k)}(t)[\theta-t,\cdots,\theta-t]\right)I\left(\max_{1\le j\le k}\|\hat{\theta}_j^{(k)}-\theta\|< \delta\right)\right\|_{L_p}\\
    &\leq k \|f^{(k)}(t)\| A_k (B_k+\|t-\theta\|)^{k-1}.
\end{align*}
Substituting the last bound in \eqref{T_f-f_bd} and also using \eqref{bd_R_k_0}, we get  bound \eqref{bd_on_T_f_t}.

To prove the second bound, note that 
\begin{align*}
&
T_f(t;X_1,\cdots,X_n)-f(\theta)- \langle \hat \theta_1^{(1)}-\theta, f'(\theta)\rangle
=
\langle \hat \theta_1^{(1)}-\theta, f'(t)-f'(\theta)\rangle
\\
&
+
\sum_{k=2}^m \Bigl(\frac{f^{(k)}(t)[\hat \theta_1^{(k)}-t,\dots, \hat \theta_k^{(k)}-t]}{k!}-\frac{f^{(k)}(t)[\theta-t,\dots, \theta-t]}{k!}\Bigr) -R_m.
\end{align*}
Therefore, 
\begin{align*}
&
\left\|(T_f(t;X_1,\dots, X_n)-f(\theta)-\langle \hat \theta_1^{(1)}-\theta, f'(\theta)\rangle)I\left(\max_{1\leq k\leq m}\max_{1\le j\le k}\|\hat{\theta}_j^{(k)}-\theta\|< \delta\right)\right\|_{L_p} 
\\
&
\leq 
\left\|\langle \hat \theta_1^{(1)}-\theta, f'(t)-f'(\theta)\rangle I\left(\|\hat \theta_1^{(1)}-\theta\|<\delta\right)\right\|_{L_p}
\\
&
+\sum_{k=2}^m \frac{\|f^{(k)}(t)\|}{(k-1)!} A_p(P,\delta) (B_p(P,\delta)+\|t-\theta\|)^{k-1} 
+ \frac{1}{m!}\|f^{(m)}\|_{{\rm Lip}_{\rho}(U)}\|t-\theta\|^{s}
\\
&
\leq 
\|f'(t)-f'(\theta)\|\sup_{\|u\|\leq 1}\left\|\langle \hat \theta_1^{(1)}-\theta, u\rangle I\left(\|\hat \theta_1^{(1)}-\theta\|<\delta\right)\right\|_{L_p}
\\
&
+\sum_{k=2}^m \frac{\|f^{(k)}(t)\|}{(k-1)!} A_p(P,\delta) (B_p(P,\delta)+\|t-\theta\|)^{k-1} 
+ \frac{1}{m!}\|f^{(m)}\|_{{\rm Lip}_{\rho}(U)}\|t-\theta\|^{s}
\\
&
\leq 
\|f''\|_{L_{\infty}(U)} A_p(P,\delta) \|t-\theta\|
+\sum_{k=2}^m \frac{\|f^{(k)}(t)\|}{(k-1)!} A_p(P,\delta) (B_p(P,\delta)+\|t-\theta\|)^{k-1} 
\\
&
+ \frac{1}{m!}\|f^{(m)}\|_{{\rm Lip}_{\rho}(U)}\|t-\theta\|^{s}.
\end{align*}

\qed
\end{proof}

Since $T_f(X_1,\dots, X_n)=T_f(\hat \theta^{(0)}; X_1,\dots, X_n)$ and, for all $k=1,\dots, m,$ estimators $\hat \theta_j^{(k)}, j=1,\dots, k$ and $\hat \theta^{(0)}$
are independent, we can use bound \eqref{bd_on_T_f_t} conditionally on $\hat \theta^{(0)}$ to get that, for $\theta=\theta(P)$ on the event $\|\hat \theta^{(0)}-\theta\|<\delta$ (or, $\hat \theta^{(0)}\in U$)
\begin{align*}
&
{\mathbb E}^{1/p}\left(\left|T_f(\hat \theta^{(0)}; X_1,\dots, X_n)-f(\theta)\right|^pI\left(\max_{1\leq k\leq m}\max_{1\le j\le k}\|\hat{\theta}_j^{(k)}-\theta\|< \delta\right)\Bigl\vert\hat \theta^{(0)}\right)
\\
&
\leq
\sum_{k=1}^m \frac{\|f^{(k)}(\hat \theta^{(0)})\|}{(k-1)!} A_p(P,\delta) (B_p(P,\delta)+\|\hat \theta^{(0)}-\theta\|)^{k-1} + \frac{1}{m!}\|f^{(m)}\|_{{\rm Lip}_{\rho}(U)}\|\hat \theta^{(0)}-\theta\|^{s}
\\
&
\leq
\sum_{k=1}^m \frac{\|f^{(k)}\|_{L_{\infty}(U)}}{(k-1)!} A_p(P,\delta) (B_p(P,\delta)+\|\hat \theta^{(0)}-\theta\|)^{k-1} + \frac{1}{m!}\|f^{(m)}\|_{{\rm Lip}_{\rho}(U)}\|\hat \theta^{(0)}-\theta\|^{s}.
\end{align*}
This yields
\begin{align}
\label{bd_on_T_cond}
&
\nonumber
\left\|(T_f(X_1,\cdots,X_n)-f(\theta))I\left(\|\hat{\theta}^{(0)}-\theta\|<\delta,
\max_{1\le k\le m}\max_{1\le j\le k}\|\hat{\theta}_j^{(k)}-\theta\|< \delta\right) \right\|_{L_p}
\\
&
\nonumber
=
\left\|
\mathbb{E}^{1/p}\left(|T_f(X_1,\cdots,X_n)-f(\theta)|^pI\left(\max_{1\le k\le m}\max_{1\le j\le k}\|\hat{\theta}_j^{(k)}-\theta\|<\delta\right)\Big\vert \hat\theta^{(0)}\right)
\right\|_{L_p}
\\
&
\nonumber
\leq \sum_{k=1}^m \frac{\|f^{(k)}\|_{L_{\infty}(U)}}{(k-1)!} A_p(P,\delta) \left\|(B_p(P,\delta)+\|\hat \theta^{(0)}-\theta\|)^{k-1}I(\|\hat \theta^{(0)}-\theta\|<\delta)\right\|_{L_p} 
\\
&
\nonumber
+ \frac{1}{m!}\|f^{(m)}\|_{{\rm Lip}_{\rho}(U)}
\left\|\|\hat \theta^{(0)}-\theta\|^{s}I(\|\hat \theta^{(0)}-\theta\|<\delta)\right\|_{L_p}
\\
&
\nonumber
\leq 
\sum_{k=1}^m \frac{2^{k-1}\|f^{(k)}\|_{L_{\infty}(U)}}{(k-1)!} A_p(P,\delta) \left\|\left(\frac{B_p(P,\delta)+\|\hat \theta^{(0)}-\theta\|}{2}\right)^{k-1}I(\|\hat \theta^{(0)}-\theta\|<\delta)\right\|_{L_p} 
\\
&
\nonumber
+ \frac{1}{m!}\|f^{(m)}\|_{{\rm Lip}_{\rho}(U)}
\left\|\|\hat \theta^{(0)}-\theta\|I(\|\hat \theta^{(0)}-\theta\|<\delta)\right\|_{L_{ps}}^s
\\
&
\nonumber
\leq 
\sum_{k=1}^m \frac{2^{k-2}\|f^{(k)}\|_{L_{\infty}(U)}}{(k-1)!} A_p(P,\delta) \left(B_p^{k-1}(P,\delta)+\left\|\|\hat \theta^{(0)}-\theta\|I(\|\hat \theta^{(0)}-\theta\|<\delta)\right\|_{L_{p(k-1)}}^{k-1}\right) 
\\
&
\nonumber
+ \frac{1}{m!}\|f^{(m)}\|_{{\rm Lip}_{\rho}(U)}
\left\|\|\hat \theta^{(0)}-\theta\|I(\|\hat \theta^{(0)}-\theta\|<\delta)\right\|_{L_{ps}}^s
\\
&
=
\sum_{k=1}^m \frac{2^{k-2}\|f^{(k)}\|_{L_{\infty}(U)}}{(k-1)!} A_p(P,\delta) \left(B_p^{k-1}(P,\delta)+\tilde \beta_{p(k-1)}^{k-1}(P,\delta)\right)
+ \frac{1}{m!}\|f^{(m)}\|_{{\rm Lip}_{\rho}(U)} \tilde \beta_{ps}^s (P,\delta).
\end{align}

Similarly, we can provide bounds on 
\begin{align*}
\left\|(T_f(X_1,\cdots,X_n)-f(\theta)- \langle \hat \theta_1^{(1)}-\theta, f'(\theta)\rangle)I\left(\|\hat{\theta}^{(0)}-\theta\|<\delta,
\max_{1\le k\le m}\max_{1\le j\le k}\|\hat{\theta}_j^{(k)}-\theta\|< \delta\right) \right\|_{L_p}.
\end{align*}
To this end, we apply bound \eqref{bd_on_T_f_t_lin} to estimator $T_f(X_1,\cdots,X_n)=T_f(\hat \theta^{(0)};X_1,\cdots,X_n)$ conditionally 
on $\hat \theta^{(0)}.$ Similarly to \eqref{bd_on_T_cond}, we then have
\begin{align}
\label{bd_on_T'''}
&
\nonumber
\left\|(T_f(X_1,\cdots,X_n)-f(\theta)-\langle \hat \theta_1^{(1)}-\theta, f'(\theta)\rangle)I\left(\|\hat{\theta}^{(0)}-\theta\|<\delta,
\max_{1\le k\le m}\max_{1\le j\le k}\|\hat{\theta}_j^{(k)}-\theta\|< \delta\right) \right\|_{L_p}
\\
&
\nonumber
\leq 
\|f''\|_{L_{\infty}(U)}A_{p}(P, \delta) \tilde \beta_p(P,\delta)+
\sum_{k=2}^m \frac{2^{k-2}\|f^{(k)}\|_{L_{\infty}(U)}}{(k-1)!} A_p(P,\delta) \left(B_p^{k-1}(P,\delta)+\tilde \beta_{p(k-1)}^{k-1}(P,\delta)\right)
\\
&
+ \frac{1}{m!}\|f^{(m)}\|_{{\rm Lip}_{\rho}(U)} \tilde \beta_{ps}^s (P,\delta).
\end{align}

Setting $\delta=+\infty$ in \eqref{bd_on_T_cond} and in \eqref{bd_on_T'''} immediately yields the bound of Theorem \ref{th_1}. 

To complete the proof of Theorem \ref{th_1_AA}, note that, by the definition of truncated estimator $\tilde T_f(X_1,\dots, X_n)$ 
for $M\geq \|f\|_{\infty},$
\begin{align*}
|\tilde T_f(X_1,\dots, X_n)-f(\theta)|\leq |T_f(X_1,\dots, X_n)-f(\theta)|
\end{align*}
since $\tilde T_f (X_1,\dots, X_n)$ is either equal to $T_f(X_1,\dots, X_n),$ or it is between $T_f(X_1,\dots, X_n)$ and $f(\theta).$
Thus, bound \eqref{bd_on_T_cond} implies that 
\begin{align}
\label{bd_on_tilde_T}
&
\nonumber
\left\|(\tilde T_f(X_1,\cdots,X_n)-f(\theta))I\left(\|\hat{\theta}^{(0)}-\theta\|<\delta,
\max_{1\le k\le m}\max_{1\le j\le k}\|\hat{\theta}_j^{(k)}-\theta\|< \delta\right)\right\|_{L_p}
\\
&
\leq \sum_{k=1}^m \frac{2^{k-2}\|f^{(k)}\|_{L_{\infty}(U)}}{(k-1)!} A_p(P,\delta) \left(B_p^{k-1}(P,\delta)+\tilde \beta_{p(k-1)}^{k-1}(P,\delta)\right)
+ \frac{1}{m!}\|f^{(m)}\|_{{\rm Lip}_{\rho}(U)} \tilde \beta_{ps}^s (P,\delta).
\end{align}
On the other hand, since $|\tilde T_f(X_1,\cdots,X_n)-f(\theta)|\leq \|f\|_{L_{\infty}}+M,$ we have 
\begin{align}
\label{bd>delta}
&
\nonumber
\left\|(\tilde T_f(X_1,\cdots,X_n)-f(\theta))I\left(\|\hat{\theta}^{(0)}-\theta\|\geq \delta\ {\rm or}\ 
\max_{1\le k\le m}\max_{1\le j\le k}\|\hat{\theta}_j^{(k)}-\theta\|\geq \delta\right)\right\|_{L_p}
\\
&
\nonumber
\leq (\|f\|_{L_{\infty}}+M) {\mathbb P}^{1/p} \biggl\{\|\hat{\theta}^{(0)}-\theta\|\geq \delta\ {\rm or}\ 
\max_{1\le k\le m}\max_{1\le j\le k}\|\hat{\theta}_j^{(k)}-\theta\|\geq \delta \biggr\}
\\
&
\leq (\|f\|_{L_{\infty}}+M) \biggl({\mathbb P}\{\|\hat{\theta}^{(0)}-\theta\|\geq \delta\}+ 
\sum_{k=1}^m \sum_{j=1}^k {\mathbb P}\{\|\hat{\theta}_j^{(k)}-\theta\|\geq \delta\} \biggr)^{1/p}.
\end{align}
Note that, by Markov inequality,
\begin{align*}
{\mathbb P}_{\theta}\{\|\hat{\theta}^{(0)}-\theta\|\geq \delta\} \leq \frac{{\mathbb E}_{\theta}\|\hat \theta^{(0)}-\theta\|^{p'}}{\delta^{p'}}
=\biggl(\frac{\tilde \beta_{p'}(P)}{\delta}\biggr)^{p'}
\end{align*}
and, similarly, 
\begin{align*}
\max_{1\le k\le m}\max_{1\le j\le k}{\mathbb P}_{\theta}\{\|\hat{\theta}_j^{(k)}-\theta\|\geq \delta\} 
\leq \biggl(\frac{B_{p'}(P)}{\delta}\biggr)^{p'}.
\end{align*}
This yields the bound 
\begin{align*}
&
\left\|(\tilde T_f(X_1,\cdots,X_n)-f(\theta))I\left(\|\hat{\theta}^{(0)}-\theta\|\geq \delta\ {\rm or}\ 
\max_{1\le k\le m}\max_{1\le j\le k}\|\hat{\theta}_j^{(k)}-\theta\|\geq \delta\right)\right\|_{L_p}
\\
&
\leq (\|f\|_{L_{\infty}}+M)\biggl(\left(\frac{m(m+1)}{2}\right)^{1/p}\biggl(\frac{B_{p'}(P)}{\delta}\biggr)^{p'/p}+ \biggl(\frac{\tilde \beta_{p'}(P)}{\delta}\biggr)^{p'/p}\biggr),
\end{align*}
which, together with \eqref{bd_on_tilde_T}, implies the first bound of Theorem \ref{th_1_AA}.

To prove the second bound, observe that, under the assumption $M\geq \|f\|_{L_{\infty}}+ \|f\|_{\rm Lip} \delta,$
the definition of estimator $\tilde T (X_1,\dots, X_n)$ implies that 
\begin{align*}
|\tilde T_f(X_1,\dots, X_n)-f(\theta)-\langle \hat \theta_1^{(1)}-\theta, f'(\theta)\rangle|
\leq |T(X_1,\dots, X_n)-f(\theta)-\langle \hat \theta_1^{(1)}-\theta, f'(\theta)\rangle|.
\end{align*}
This follows from the fact that $|f(\theta)+\langle \hat \theta_1^{(1)}-\theta, f'(\theta)\rangle|\leq M$ and, hence, 
either $\tilde T_f (X_1,\dots, X_n)=T_f(X_1,\dots, X_n),$ 
or $\tilde T_f (X_1,\dots, X_n)$ is between $T_f(X_1,\dots, X_n)$ 
and $f(\theta)+\langle \hat \theta_1^{(1)}-\theta, f'(\theta)\rangle.$ Therefore, it follows from \eqref{bd_on_T'''} that 
\begin{align}
\label{bd_on_T_ABC}
&
\left\|(\tilde T_f(X_1,\cdots,X_n)-f(\theta)-\langle \hat \theta_1^{(1)}-\theta, f'(\theta)\rangle)I\left(\|\hat{\theta}^{(0)}-\theta\|<\delta,
\max_{1\le k\le m}\max_{1\le j\le k}\|\hat{\theta}_j^{(k)}-\theta\|< \delta\right) \right\|_{L_p}
\nonumber
\\
&
\nonumber
\leq 
\|f''\|_{L_{\infty}(U)}A_{p}(P, \delta) \tilde \beta_p(P,\delta)+
\sum_{k=2}^m \frac{2^{k-2}\|f^{(k)}\|_{L_{\infty}(U)}}{(k-1)!} A_p(P,\delta) \left(B_p^{k-1}(P,\delta)+\tilde \beta_{p(k-1)}^{k-1}(P,\delta)\right)
\\
&
+ \frac{1}{m!}\|f^{(m)}\|_{{\rm Lip}_{\rho}(U)} \tilde \beta_{ps}^s (P,\delta).
\end{align}
On the other hand, bound \eqref{bd>delta} could be rewritten as follows:
\begin{align}
\label{bd>delta_AAA}
&
\nonumber
\left\|(\tilde T_f(X_1,\cdots,X_n)-f(\theta)-\langle \hat \theta_1^{(1)}-\theta, f'(\theta)\rangle)
I\left(\|\hat{\theta}^{(0)}-\theta\|\geq \delta\ {\rm or}\ 
\max_{1\le k\le m}\max_{1\le j\le k}\|\hat{\theta}_j^{(k)}-\theta\|\geq \delta\right)\right\|_{L_p}
\\
&
\nonumber
\leq (\|f\|_{L_{\infty}}+M) \biggl({\mathbb P}\{\|\hat{\theta}^{(0)}-\theta\|\geq \delta\}+ 
\sum_{k=1}^m \sum_{j=1}^k {\mathbb P}\{\|\hat{\theta}_j^{(k)}-\theta\|\geq \delta\} \biggr)^{1/p}
\\
&
\nonumber
+ \|f'(\theta)\| \left\|\|\hat \theta_1^{(1)}-\theta\| I(\|\hat \theta^{(0)}-\theta\|\geq \delta)\right\|_{L_p}
+\|f'(\theta)\| \left\|\|\hat \theta_1^{(1)}-\theta\| I(\|\hat \theta_1^{(1)}-\theta\|\geq \delta)\right\|_{L_p}
\\
&
+\|f'(\theta)\|\sum_{k=2}^m \sum_{j=1}^k \left\|\|\hat \theta_1^{(1)}-\theta\| I(\|\hat \theta_j^{(k)}-\theta\|\geq \delta)\right\|_{L_p}.
\end{align}
Note that 
\begin{align*}
\left\| \|\hat \theta_1^{(1)}-\theta\| I(\|\hat \theta_1^{(1)}-\theta\|\geq \delta)\right\|_{L_p}
\leq {\mathbb E}^{1/p} \left(\|\hat \theta_1^{(1)}-\theta\|^p \frac{\|\hat \theta_1^{(1)}-\theta\|^{p'-p}}{\delta^{p'-p}}\right)
\leq \frac{B_{p'}^{p'/p}(P)}{\delta^{p'/p-1}}.
\end{align*}
In addition, using independence of estimators $\hat \theta^{(0)}$ and $\hat \theta_1^{1},$ we get
\begin{align*} 
\left\| \|\hat \theta_1^{(1)}-\theta\| I(\|\hat \theta^{(0)}-\theta\|\geq \delta)\right\|_{L_p}
&
= \left\| \|\hat \theta_1^{(1)}-\theta\|\right\|_{L_p}{\mathbb P}_{\theta}^{1/p}\{\|\hat \theta^{(0)}-\theta\|\geq \delta\}
\\
&
\leq  \frac{B_{p}(P)\tilde \beta_{p'}^{p'/p}(P)}{\delta^{p'/p}}.
\end{align*}
Similarly, for all $k=2,\dots, m, j=1,\dots, k,$
\begin{align*}
\left\| \|\hat \theta_1^{(1)}-\theta\| I(\|\hat \theta_j^{(k)}-\theta\|\geq \delta)\right\|_{L_p}
&= 
\left\| \|\hat \theta_1^{(1)}-\theta\|\right\|_{L_p}{\mathbb P}_{\theta}^{1/p}\{\|\hat \theta_j^{(k)}-\theta\|\geq \delta\}
\\
&
\leq 
 \frac{B_{p}(P)B_{p'}^{p'/p}(P)}{\delta^{p'/p}}.
\end{align*}
 Therefore,
\begin{align*}
&
\nonumber
\left\|(\tilde T_f(X_1,\cdots,X_n)-f(\theta)-\langle \hat \theta_1^{(1)}-\theta, f'(\theta)\rangle)
I\left(\|\hat{\theta}^{(0)}-\theta\|\geq \delta\ {\rm or}\ 
\max_{1\le k\le m}\max_{1\le j\le k}\|\hat{\theta}_j^{(k)}-\theta\|\geq \delta\right)\right\|_{L_p}
\\
&
\leq (\|f\|_{L_{\infty}}+M)\biggl(\left(\frac{m(m+1)}{2}\right)^{1/p}\biggl(\frac{B_{p'}(P)}{\delta}\biggr)^{p'/p}+ \biggl(\frac{\tilde \beta_{p'}(P)}{\delta}\biggr)^{p'/p}\biggr)
\\
&
+  \|f'\|_{L_{\infty}}\biggl(\frac{B_{p'}^{p'/p}(P)}{\delta^{p'/p-1}}+ B_p(P) \biggl(\frac{\tilde \beta_{p'}(P)}{\delta}\biggr)^{p'/p}
+ \frac{m(m-1)}{2}B_{p}(P)\biggl(\frac{B_{p'}(P)}{\delta}\biggr)^{p'/p}\biggr).
\end{align*}
Note also that 
\begin{align*}
\|f'\|_{L_{\infty}}\frac{B_{p'}^{p'/p}(P)}{\delta^{p'/p-1}} =\|f'\|_{L_{\infty}}\delta \biggl(\frac{B_{p'}(P)}{\delta}\biggr)^{p'/p} \leq M \biggl(\frac{B_{p'}(P)}{\delta}\biggr)^{p'/p},
\end{align*}
which implies 
\begin{align*}
&
\nonumber
\left\|(\tilde T_f(X_1,\cdots,X_n)-f(\theta)-\langle \hat \theta_1^{(1)}-\theta, f'(\theta)\rangle)
I\left(\|\hat{\theta}^{(0)}-\theta\|\geq \delta\ {\rm or}\ 
\max_{1\le k\le m}\max_{1\le j\le k}\|\hat{\theta}_j^{(k)}-\theta\|\geq \delta\right)\right\|_{L_p}
\\
&
\leq C_m (\|f\|_{L_{\infty}}+\|f'\|_{L_{\infty}}+M)(1+B_p(P))
\biggl(\biggl(\frac{B_{p'}(P)}{\delta}\biggr)^{p'/p}+ \biggl(\frac{\tilde \beta_{p'}(P)}{\delta}\biggr)^{p'/p}\biggr).
\end{align*}
where $C_m := \frac{m(m-1)}{2}\vee \left(\frac{m(m+1)}{2}\right)^{1/p}+1.$
Together with \eqref{bd_on_T_ABC}, this completes the proof of the second bound.

\qed
\end{proof}

\section{Remarks on Bernstein type bounds in functional estimation}
\label{functionals_Bernstein}

For the sample mean $\bar X_n$ of  i.i.d. observations $X_1,\dots, X_n$ of real valued r.v. $X$ uniformly bounded by constant $U>0$ 
and with variance $\sigma,$ the Bernstein inequality could be written 
in the following form: for all $t\geq 0$ with probability at least $1-e^{-t}$
\begin{align*}
|\bar X_n-{\mathbb E}X| \lesssim \sigma \sqrt{\frac{t}{n}}\vee U\frac{t}{n}.
\end{align*}
It is equivalent to the following bounds on the $L_p$-norms: for all $p\geq 1,$
\begin{align*}
\|\bar X_n-{\mathbb E}X\|_{L_p} \lesssim \sigma \sqrt{\frac{p}{n}}\vee U\frac{p}{n}.
\end{align*}
Moreover, if $X$ is a subexponential r.v. (not necessarily bounded) with $\|X\|_{\psi_1}\asymp \sigma,$
then similar bounds hold with $U=\sigma.$ If $X$ is a subgaussian r.v. with $\|X\|_{\psi_2}\asymp \sigma,$ then we 
have that, for all $t\geq 0$ with probability at least $1-e^{-t},$
$
|\bar X_n-{\mathbb E}X| \lesssim \sigma \sqrt{\frac{t}{n}}
$
and, for all $p\geq 1,$
$
\|\bar X_n-{\mathbb E}X\|_{L_p} \lesssim \sigma \sqrt{\frac{p}{n}}.
$

More generally, for a number of important estimators $\hat \theta_n$ of parameter $\theta(P),$ similar bounds on the $L_p$ norms of 
$\langle \hat \theta_n-\theta(P), u\rangle, u\in E^{\ast}$ 
and $\|\hat \theta_n-\theta(P)\|$ hold for all $p\geq 1$
and imply Bernstein type inequalities. 


The properties of base estimators $\hat \theta_n$ are postulated in the following assumptions.

\begin{assumption}
\label{assume_Bernstein_lin}
\normalfont
There exist functions $\sigma:{\mathcal P}\mapsto {\mathbb R}_+$ and $U:{\mathcal P}\mapsto {\mathbb R}_+$
such that, for all $p\geq 1,$ 
 \begin{align*}
 \sup_{\|u\|\leq 1} \Bigl\|\langle \hat \theta_n-\theta(P), u\rangle\Bigr\|_{L_p({\mathbb P}_{P})} \lesssim 
 \sigma(P)\sqrt{\frac{p}{n}}\vee U(P)\frac{p}{n}, P \in {\mathcal P}.
 \end{align*}
 \end{assumption}
 
 In this case, it will be said that $\hat \theta_n$ is a {\it Bernstein type estimator}. 
 If $U(P)=0, P\in {\mathcal P}$ the above bound simplifies as follows:
 \begin{align*}
 \sup_{\|u\|\leq 1} \Bigl\|\langle \hat \theta_n-\theta(P), u\rangle\Bigr\|_{L_p({\mathbb P}_{P})} \lesssim \
 \sigma(P)\sqrt{\frac{p}{n}}, P \in {\mathcal P},
 \end{align*}
 and it will be said that $\hat \theta_n$ is a {\it subgaussian type estimator}. 

\begin{assumption}
\label{assume_Bernstein_on_norm}
\normalfont
There exist functions $\sigma:{\mathcal P}\mapsto {\mathbb R}_+,$ $U:{\mathcal P}\mapsto {\mathbb R}_+,$
$d_1: {\mathcal P}\mapsto {\mathbb R}_+$ and $d_2: {\mathcal P}\mapsto {\mathbb R}_+$
such that, for all $p\geq 1,$ 
 \begin{align*}
 \Bigl\|\|\hat \theta_n-\theta(P)\|\Bigr\|_{L_p({\mathbb P}_{P})} \lesssim 
 \sqrt{\frac{d_1(P)}{n}}\vee \frac{d_2(P)}{n}\vee \sigma(P)\sqrt{\frac{p}{n}}\vee U(P)\frac{p}{n}, P \in {\mathcal P}.
 \end{align*}
\end{assumption}

As in Section \ref{Main_results}, we consider the problem of estimation of functionals $f(\theta(P))$ based on i.i.d. observations $X_1,\dots, X_n\sim P, P\in {\mathcal P}.$ 
We use base estimator $\hat \theta_n$ of $\theta(P)$ and the sample split to construct 
estimator $T_f(X_1,\dots, X_n)$ of $f(\theta(P))$ defined by \eqref{basic_T_f} and its truncated version $\tilde T_f(X_1,\dots, X_n)$ defined by \eqref{basic_T_fM}.
For base estimators $\hat \theta_n$ satisfying Assumptions \ref{assume_Bernstein_lin} and \ref{assume_Bernstein_on_norm}, 
the following result holds.

\begin{theorem}
\label{th_1_Bern_AA}
Let $f:E\mapsto {\mathbb R}$ be a uniformly bounded functional and let $\theta=\theta(P)$ for some $P\in {\mathcal P}.$
Suppose $f$ is $m$ times Fr\`echet continuously differentiable in $U:=B(\theta;\delta)$ for some $\delta\in (0,1]$ and some $m\geq 2,$ and, moreover,
$f^{(m)}$ satisfies the Hölder condition with exponent $\rho\in(0,1]$ in $U.$ 
Assume that for a sufficiently large constant $C>0,$
\begin{align*}
C\left(\sqrt{\frac{d_1(P)}{n^{(0)}}}\vee\frac{d_2(P)}{n^{(0)}}\right)\leq \delta
\end{align*}
and also that assumptions \ref{assume_Bernstein_lin} and \ref{assume_Bernstein_on_norm} hold. 

(i) Suppose that $n^{(0)}\asymp_m n$ and $n_j^{(k)}\asymp_m n$, $j=1,\cdots,k$, $k=1,\cdots,m.$
Then, for $M\geq \|f\|_{L_{\infty}},$ $p\geq 1$ and for a constant $c>0$ depending on $m,C,$     
\begin{align*}
    &\left\|\tilde{T}_f(X_1,\cdots,X_n)-f(\theta(P))\right\|_{L_p(\mathbb{P}_P)}
    \lesssim_{s, C} \max_{1\leq k\leq m} \|f^{(k)}\|_{L_\infty(U)} \left(\sigma(P)\sqrt{\frac{p}{n}}\vee U(P)\frac{p}{n}\right)  
    \\
    &  
   +\|f^{(m)}\|_{{\rm Lip}_\rho(U)} \left(\Bigl(\frac{d_1(P)}{n}\Bigr)^{s/2}\vee \Bigl(\frac{d_2(P)}{n}\Bigr)^s\vee   \sigma^s(P) \Bigl(\frac{p}{n}\Bigr)^{s/2}\vee U^s(P)\Bigl(\frac{p}{n}\Bigr)^s   
   \right)
   \\
   &
   +(\|f\|_{L_\infty}+M)\exp\Bigl\{-c\frac{n}{p}\Bigl(\frac{\delta^2}{\sigma^2(P)}\wedge \frac{\delta}{U(P)}\Bigr)\Bigr\}.
    \end{align*}

 (ii) Suppose that, for all $j=1,\dots, k, k=1,\dots, m,$ $n_j^{(k)}\asymp_m n.$
If $M\geq \|f\|_{\infty}+ \|f\|_{{\rm Lip}}\delta,$ then 
\begin{align*}
    &\left\|\tilde T_f(X_1,\cdots,X_n)-f(\theta(P))- \langle \hat \theta_1^{(1)}-\theta, f'(\theta(P))\rangle\right\|_{L_p(\mathbb{P}_P)}
    \\
    &
    \lesssim_{s, C} \max_{2\leq k\leq m} \|f^{(k)}\|_{L_\infty(U)} \left(\sigma(P)\sqrt{\frac{p}{n}}\vee U(P)\frac{p}{n}\right) \times
    \\
    &
    \ \ \ \ \ \ \ \ \ \ \ \ \ \ \ \ \ \ \ \ \ \ \ \ \ \ \ \ \ 
    \left(\sqrt{\frac{d_1(P)}{n^{(0)}}}\vee \frac{d_2(P)}{n^{(0)}}\vee 
    \sigma(P)\sqrt{\frac{p}{n^{(0)}}}\vee U(P)\frac{p}{n^{(0)}}\right) 
 \\
 &
 +\|f^{(m)}\|_{{\rm Lip}_\rho(U)} \left(\Bigl(\frac{d_1(P)}{n^{(0)}}\Bigr)^{s/2}\vee \Bigl(\frac{d_2(P)}{n^{(0)}}\Bigr)^s\vee
  \sigma^s(P) \Bigl(\frac{p}{n^{(0)}}\Bigr)^{s/2}\vee U^s(P)\Bigl(\frac{p}{n^{(0)}}\Bigr)^s \right)
 \\
 &
 + 
\left(\|f\|_{L_\infty}+M + \|f'(\theta)\|\left( \sqrt{\frac{d_1(P)}{n}}\vee \frac{d_2(P)}{n}\vee \sigma(P)\sqrt{\frac{p}{n}}\vee U(P)\frac{p}{n}\right) \right)\times 
\\
&
  \ \ \ \ \ \ \ \ \ \ \ \ \ \ \ \ \ \ \ \ \ \ \ \ \ \ \ \ \ \ \ \ \ \ \ \ \ \ \ \ \ \ \ \ \ \ \ \ \ \ \ \ \ \ \ \ \ \  
  \exp\Bigl\{-c\frac{n^{(0)}}{p}\Bigl(\frac{\delta^2}{\sigma^2(P)}\wedge \frac{\delta}{U(P)}\Bigr)\Bigr\}.
\end{align*}    
\end{theorem}

Since $e^{-x}\leq x^{-1}, x>0,$ the bound of Theorem \ref{th_1_Bern_AA}, (i) implies the following somewhat simplified inequality (for $M=\|f\|_{L_{\infty}}$):
\begin{align*}
   &
    \left\|\tilde{T}_f(X_1,\cdots,X_n)-f(\theta(P))\right\|_{L_p(\mathbb{P}_P)}
   \lesssim_{s, C} \max_{0\leq k\leq m} \|f^{(k)}\|_{L_\infty(U)} \left(\sigma(P)\sqrt{\frac{p}{n}}\vee \frac{\sigma^2(P)}{\delta^2} \frac{p}{n}\vee\frac{U(P)}{\delta}\frac{p}{n}\right)    
   \\
   &
+\|f^{(m)}\|_{{\rm Lip}_\rho(U)} \left(\Bigl(\frac{d_1(P)}{n}\Bigr)^{s/2}\vee \Bigl(\frac{d_2(P)}{n}\Bigr)^s\right)
+\|f^{(m)}\|_{{\rm Lip}_\rho(U)} \left(\sigma^s(P) \Bigl(\frac{p}{n}\Bigr)^{s/2}\vee U^s(P)\Bigl(\frac{p}{n}\Bigr)^s\right),
 \end{align*}
which also implies that 
\begin{align*}
    \left\|\tilde{T}_f(X_1,\cdots,X_n)-f(\theta(P))\right\|_{L_{\psi_{1/s}}(\mathbb{P}_P)}
   & \lesssim_{s, C} \max_{0\leq k\leq m} \|f^{(k)}\|_{L_\infty(U)} \left( \frac{\sigma(P)}{\sqrt{n}}\vee \frac{\sigma^2(P)}{\delta^2 n} \vee\frac{U(P)}{\delta n}\right)    
   \\
   &
  +\|f^{(m)}\|_{{\rm Lip}_\rho(U)} \left(\Bigl(\frac{d_1(P)}{n}\Bigr)^{s/2}\vee \Bigl(\frac{d_2(P)}{n}\Bigr)^s\right).
 \end{align*}
Moreover, in the subgaussian case (when $U(P)=0$ and $d_2(P)=0$), we have for all $p\geq 1$
\begin{align*}
   &
    \left\|\tilde{T}_f(X_1,\cdots,X_n)-f(\theta(P))\right\|_{L_p(\mathbb{P}_P)}
   \lesssim_{s, C} \max_{0\leq k\leq m} \|f^{(k)}\|_{L_\infty(U)} 
   \left(\sigma(P)\sqrt{\frac{p}{n}}\vee \frac{\sigma^2(P)}{\delta^2} \frac{p}{n}\right)    
   \\
   &
   +\|f^{(m)}\|_{{\rm Lip}_\rho(U)} \Bigl(\frac{d_1(P)}{n}\Bigr)^{s/2}
    +\|f^{(m)}\|_{{\rm Lip}_\rho(U)} \sigma^s(P) \Bigl(\frac{p}{n}\Bigr)^{s/2}          
 \end{align*}
and
\begin{align*}
&
  \left\|\tilde{T}_f(X_1,\cdots,X_n)-f(\theta(P))\right\|_{L_{\psi_{1/s}}(\mathbb{P}_P)}
   \\
   & 
   \lesssim_{s, C} \max_{0\leq k\leq m} \|f^{(k)}\|_{L_\infty(U)} \left( \frac{\sigma(P)}{\sqrt{n}}\vee \frac{\sigma^2(P)}{\delta^2 n}\right)    
  +\|f^{(m)}\|_{{\rm Lip}_\rho(U)}  \Bigl(\frac{d_1(P)}{n}\Bigr)^{s/2}.
 \end{align*}
 
 To prove Theorem \ref{th_1_Bern_AA}, one just needs to modify the part of the argument in the proof of Theorem \ref{th_1_AA} related to bounding the probability of event 
\begin{align*}
\biggl\{\|\hat{\theta}^{(0)}-\theta\|\geq \delta\ {\rm or}\ 
\max_{1\le k\le m}\max_{1\le j\le k}\|\hat{\theta}_j^{(k)}-\theta\|\geq \delta \biggr\},
\end{align*}
where $\theta=\theta(P).$

\begin{proof}
Note that, under Assumption \ref{assume_Bernstein_on_norm} for a sufficiently large constant $C>0,$ we have 
\begin{align*}
 {\mathbb P}_{P}\left\{\|\hat \theta_n-\theta(P)\|\geq C\left(\sqrt{\frac{d_1(P)}{n}} \vee \frac{d_2(P)}{n}\vee 
\sigma(P)\sqrt{\frac{t}{n}} \vee U(P)\frac{t}{n} \right)\right\}\leq e^{-t}.
\end{align*}
Under the assumption that 
\begin{align*}
C\left(\sqrt{\frac{d_1(P)}{n}} \vee \frac{d_2(P)}{n}\right) \leq \delta,
\end{align*}
we can set 
\begin{align*}
t:= n\left(\frac{1}{C^2}\frac{\delta^2}{\sigma^2(P)}\wedge \frac{1}{C} \frac{\delta}{U(P)}\right).
\end{align*}
This yields the bound 
\begin{align*}
{\mathbb P}_{P}\{\|\hat \theta_n-\theta(P)\|\geq \delta\}\leq \exp\left\{- c n \left(\frac{\delta^2}{\sigma^2(P)}\wedge \frac{\delta}{U(P)}\right)\right\}
\end{align*}
that holds for a sufficiently small constant $c>0.$ Under the assumption 
that $n^{(0)}\asymp_m n$ and $n_j^{(k)}\asymp_m n$, $j=1,\cdots,k$, $k=1,\cdots,m,$ we have 
\begin{align*}
&
{\mathbb P}_{P}^{1/p}\biggl\{\|\hat{\theta}^{(0)}-\theta(P)\|\geq \delta\ {\rm or}\ 
\max_{1\le k\le m}\max_{1\le j\le k}\|\hat{\theta}_j^{(k)}-\theta(P)\|\geq \delta \biggr\}
\\
&
\lesssim_m \exp\left\{- c\frac{n}{p} \left(\frac{\delta^2}{\sigma^2(P)}\wedge \frac{\delta}{U(P)}\right)\right\}
\end{align*}
for some $c>0.$
Together with the rest of the argument of the proof of Theorem  \ref{th_1_AA}, it is now easy to complete the proof of statement (i) of Theorem \ref{th_1_Bern_AA}
and, similarly, conclude the proof of statement (ii).

\qed
\end{proof}

It is well known that Assumption \ref{assume_Bernstein_lin} is equivalent to the following Bernstein type inequality: for all $P\in {\mathcal P},$ 
for all $u\in E^{\ast}$ with $\|u\|\leq 1$ and for all $t\geq 1,$
with probability ${\mathbb P}_{P}$ at least $1-e^{-t}$
 \begin{align}
 \label{Bernstein_bd_a}
 |\langle \hat \theta_n-\theta(P), u\rangle| \lesssim \sigma(P)\sqrt{\frac{t}{n}}\vee U(P)\frac{t}{n}.
 \end{align}
 
 To obtain Bernstein type bounds on $\|\hat \theta_n-\theta(P)\|,$ the following assumption on the norm of Banach space $E$ is useful. 
 
 \begin{assumption}
 \label{assum_norm_d}
 \normalfont
 Suppose ${\mathcal M}$ is a finite set of linear functionals $u\in E^{\ast}$ with $\|u\|\leq 1$ such that 
\begin{align}
\label{assume_on_nrm}
\|x\| \leq C \max_{u\in {\mathcal M}}|\langle x,u\rangle|, x\in E
\end{align} 
with a numerical constant $C>0$ and $\log {\rm card}({\mathcal M})\leq d$ for some $d\geq 1.$
 \end{assumption}

\begin{proposition}
\label{Bernstein_on_norms}
Under Assumptions \ref{assume_Bernstein_lin} and \ref{assum_norm_d} for all  $P\in {\mathcal P}$ and for all $p\geq 1,$
\begin{align*}
\left\| \|\hat \theta_n - \theta(P)\|\right\|_{L_p({\mathbb P}_P)}\lesssim 
C \left[\sigma(P)\sqrt{\frac{d}{n}} \vee  U(P)\frac{d}{n}\vee \sigma(P)\sqrt{\frac{p}{n}} \vee U(P)\frac{p}{n}\right]. 
\end{align*}
In particular, under Assumption \ref{assume_Bernstein_lin} with $U(P)=0,$ 
\begin{align*}
\left\| \|\hat \theta_n - \theta(P)\|\right\|_{L_p({\mathbb P}_P)}\lesssim 
C \left[\sigma(P)\sqrt{\frac{d}{n}} \vee \sigma(P)\sqrt{\frac{p}{n}}\right]. 
\end{align*}
\end{proposition}

\begin{proof}
It is enough to replace in bound \eqref{Bernstein_bd_a} $t$ by $t+\log M$ for $M:={\rm card}({\mathcal M})$ and apply this bound to all $u\in {\mathcal M}$
and, finally, use the union bound to prove that, for all $t\geq 1,$ with probability at least $1-e^{-t},$
\begin{align*}
\|\hat \theta_n - \theta(P)\|\lesssim 
C \sigma(P)\sqrt{\frac{\log M}{n}} \vee  C U(P)\frac{\log M}{n}\vee C \sigma(P)\sqrt{\frac{t}{n}} \vee CU(P)\frac{t}{n}. 
\end{align*}
The claim now follows from a standard derivation of the bounds on the $L_p$-norms from the exponential inequality. 

\qed
\end{proof}

Thus, under Assumptions \ref{assume_Bernstein_lin} and \ref{assum_norm_d},
the bound of Assumption \ref{assume_Bernstein_on_norm} holds with $d_1(P):= \sigma^2(P)d$ and $d_2(P):= U(P) d.$

\begin{example}
\normalfont
If ${\rm dim}(E)=d,$ then, by a standard volumetric argument, there exists a $1/2$-net ${\mathcal M}$ for the unit ball $\{u\in E^{\ast}: \|u\|\leq 1\}$
of cardinality $M\leq 5^d,$ and we have 
\begin{align*}
\|x\| \leq 2\max_{u\in {\mathcal M}} |\langle x,u\rangle|, x\in E.
\end{align*} 
This yields the bound 
\begin{align}
\label{Bernstein_on_norm_L_p}
   \left\| \|\hat \theta_n-\theta(P)\| \right\|_{L_p({\mathbb P}_P)}\lesssim 
   \sigma(P)\sqrt{\frac{d}{n}} \vee  U(P)\frac{d}{n}\vee \sigma(P)\sqrt{\frac{p}{n}} \vee U(P)\frac{p}{n}
\end{align}
that holds for all $p\geq 1.$
\end{example}

\begin{example}
\normalfont
Let $E$ be the space of linear operators $A:{\mathbb R}^d\mapsto {\mathbb R}^d$ equipped with the operator norm.
Let ${\mathcal M}$ be the set of operators $u\otimes v, u,v\in {\mathcal N},$ where ${\mathcal N}$ is a $1/4$-net for the unit ball of ${\mathbb R}^d$ of cardinality 
$\leq 9^d.$ Then, $M={\rm card}({\mathcal M})\leq 9^{2d}$ and 
\begin{align*}
\|A\| \leq 2\max_{B\in {\mathcal M}} |\langle A,B\rangle| = 2\max_{u,v\in {\mathcal N}}|\langle Au,v\rangle|, A\in E.
\end{align*} 
Thus, we again get bound \eqref{Bernstein_on_norm_L_p} for all $p\geq 1.$
Note that, in this example, ${\rm dim}(E)=d^2,$ but the bound rather depends on the dimension $d$ of space ${\mathbb R}^d.$ 
\end{example}

As the above examples show, the quantity $\log {\rm card}({\mathcal M})$ is often upper bounded by the dimension or other similar complexity parameters $d.$

Bernstein type bounds on the norm $\|\hat \theta-\theta(P)\|$ could be also often deduced from concentration inequalities for  
$\|\hat \theta_n-\theta(P)\|$ around its expectation and the bounds on ${\mathbb E}_P\|\hat \theta_n-\theta(P)\|.$

\begin{example}
\normalfont
Let ${\mathcal P}={\mathcal P}(S)$ be the set of all probability 
measures on a measurable space $(S,{\mathcal A})$ and let the goal be to estimate the value of $f(P)$ based 
on i.i.d. observations $X_1,\dots, X_n\sim P$ for a real valued functional $f:{\mathcal P}(S)\mapsto {\mathbb R}.$
Denote by $W(S)$ the space of all signed measures on $(S,{\mathcal A})$ with a bounded total variation. 
Clearly, ${\mathcal P}(S)\subset W(S).$
Let ${\mathcal G}$ be a class of ${\mathcal A}$-measurable functions $g:S\mapsto {\mathbb R}$ such 
that 
\begin{align*}
|g(x)|\leq U, x\in S, g\in {\mathcal G}.
\end{align*}
For all $\nu\in W(S),$
\begin{align*}
\|\nu\|_{\mathcal G}:= \sup_{g\in {\mathcal G}}|\nu (g)|<+\infty,
\end{align*}
where $\nu (g):= \nu g:=\int_S gd\nu.$
In what follows, we equip the space $W(S)$ with the norm $\|\cdot\|_{\mathcal G}$ (instead of the usual total variation norm).
To emphasize the dependence of the norm on ${\mathcal G},$ we will also call this space $W_{\mathcal G}.$
We will view $(W_{\mathcal G}, \|\cdot\|_{\mathcal G})$ as a Banach space $E$ and set $\theta(P):=P, P\in {\mathcal P}(S).$ 

The empirical distribution $P_n:= n^{-1}\sum_{k=1}^n \delta_{X_k},$ where $\delta_x$ denotes the unit point mass at $x\in S,$
will be used as an estimator $\hat \theta_n$ of $\theta(P)=P.$ 
Let 
\begin{align*}
\sigma^2=\sigma_P^2({\mathcal G}):= \sup_{g\in {\mathcal G}}(Pg^2 -(Pg)^2).
\end{align*}

We will also assume that there exists $V:= V({\mathcal G})\geq 1$ such that 
\begin{align}
\label{emp_pr}
{\mathbb E}_P\|P_n-P\|_{\mathcal G}\lesssim  \sigma \sqrt{\frac{V}{n}} \vee U\frac{V}{n}.
\end{align}
Conditions of this type often hold for empirical processes. For instance, if the uniform $L_2$-covering numbers of class 
${\mathcal G}$ satisfy the bound 
\begin{align*}
\sup_{Q\in {\mathcal P}(S)}N({\mathcal G};L_2(Q), \eps) \leq \Bigl(\frac{AU}{\eps}\Bigr)^v, \eps>0,
\end{align*}
for some $v>0, A>0,$ then bound \eqref{emp_pr} holds with $V:= v \log \frac{AU}{\sigma}$ (see \cite{Kol_oracle}, p. 46).

Suppose that ${\mathcal G}$ is equipped with some topology $\tau$ such that it is a first countable compact Hausdorff space 
and functions ${\mathcal G}\ni g\mapsto g(x)$ are continuous for all $x\in S.$

Let $f:{\mathcal P}(S)\mapsto {\mathbb R}$ be a bounded Lipschitz functional with respect to the norm $\|\cdot\|_{\mathcal G}.$
By McShane-Whitney extension theorem, $f$ can be extended to a bounded Lipschitz functional on the whole space 
$(W_{\mathcal G}, \|\cdot\|_{\mathcal G})$
with preservation of its sup-norm and its Lipschitz constant (such an extension will be still denoted by $f$). Let $P\in {\mathcal P}(S)$ and let $V\subset W_{\mathcal G}=W(S)$ be a neighborhood of $P.$ We will assume that there exists a bounded Lipschitz extension of $f$
such that $f\in C^s(V)$ for a given $s>0$ (some weaker notions of smoothness of $f$ would be also sufficient for our purposes, but, for simplicity,
we will stick to the above definition). We are now in a position to apply the bounds of Theorem \ref{th_1_Bern_AA} to estimator 
$\tilde T_f(X_1,\dots, X_n)$ of $f(P)$ based on the empirical distribution $P_n.$ For simplicity, we state the result only for $p=2.$

\begin{proposition}
Let $\delta \in (0,1]$ and let $V:=B(P,\delta)\subset W_{\mathcal G}.$ 
Suppose $s:=m+\rho$ for some $m\geq 2$ and $\rho\in (0,1].$
Assume that, 
for a sufficiently large constant $C>0,$
\begin{align*}
C\left(\sigma \sqrt{\frac{V}{n^{(0)}}}\vee U\frac{V}{n^{(0)}}\right)\leq \delta.
\end{align*}
Then, for estimator $\tilde T_f(X_1,\dots, X_n)$ with $M=\|f\|_{L_{\infty}},$ the following bound holds:
\begin{align*}
    &
    \left\|\tilde{T}_f(X_1,\cdots,X_n)-f(P)\right\|_{L_2(\mathbb{P}_P)}
    \\
    &
    \lesssim_{s, C} \|f\|_{C^s(V)}\left[\left(\frac{\sigma}{\sqrt{n}} \vee \frac{U}{n}\right) + \left(\sigma \sqrt{\frac{V}{n}}\vee U\frac{V}{n}\right)^s
    + \exp\Bigl\{-c_1 n\Bigl(\frac{\delta^2}{\sigma^2}\wedge \frac{\delta}{U}\Bigr)\Bigr\}
    \right].
  \end{align*}
\end{proposition}

\begin{proof}
We start with the following lemma.

\begin{lemma}
\label{prop_G_compact}
For all $P\in {\mathcal P}(S)$ and all $p\geq 1,$
\begin{align*}
\sup_{u\in W_{\mathcal G}^{\ast}, \|u\|\leq 1} \Bigl\|\langle P_n-P, u\rangle\Bigr\|_{L_p({\mathbb P}_{P})}
\lesssim \sigma \sqrt{\frac{p}{n}} \vee U\frac{p}{n}.
\end{align*}
\end{lemma}

\begin{proof} 
As a consequence of Bernstein inequality, for all $P\in {\mathcal P}(S),$ all $p\geq 1$ and all $g\in {\mathcal G},$
\begin{align}
\label{emp_Bern}
\|(P_n -P)g\|_{L_p({\mathbb P}_P)}\lesssim \sigma \sqrt{\frac{p}{n}} \vee U\frac{p}{n}.
\end{align}
Uniform boundedness of ${\mathcal G}$ and continuity of functions ${\mathcal G}\ni g\mapsto g(x), x\in S$ implies 
that ${\mathcal G}\ni g\mapsto \int_S gd\nu$ is continuous in topology $\tau$ for any signed measure $\nu\in W(S).$ 
Thus, $W_{\mathcal G}$ could be viewed as a linear subspace of the space of continuous functions $C({\mathcal G})$
equipped with the sup-norm. By Hahn-Banach Theorem, any bounded linear functional $u\in W_{\mathcal G}^{\ast}$
could be extended to a bounded linear functional on $C({\mathcal G})$ with preservation of its norm. By Riesz-Markov Theorem, the dual space 
of $C({\mathcal G})$ can be identified with the space of Borel signed measures on 
${\mathcal G}$ equipped with the total variation norm. Thus, functional  $u\in W_{\mathcal G}^{\ast}$ can 
be represented as an integral with respect to a Borel signed measure $\gamma$ of total variation $\|\gamma\|=\|u\|$ on 
the space ${\mathcal G}.$ Setting $\xi_n:= P_n-P,$ we get 
\begin{align*}
 &
 \sup_{u\in W_{\mathcal G}^{\ast}, \|u\|\leq 1} \Bigl\|\langle P_n-P, u\rangle\Bigr\|_{L_p({\mathbb P}_{P})}\leq 
 \sup_{\gamma\in C({\mathcal G})^{\ast}, \|\gamma\|\leq 1} \Bigl\|\int_{{\mathcal G}} \xi_n(g)\gamma(dg)\Bigr\|_{L_p({\mathbb P}_{P})}
 \\
 &
 \leq 
\sup_{\gamma\in C({\mathcal G})^{\ast},\|\gamma\|\leq 1} {\mathbb E}^{1/p}\Bigl|\int_{{\mathcal G}} |\xi_n(g)| |\gamma|(dg)\Bigr|^p
\leq 
\sup_{\gamma\in C({\mathcal G})^{\ast},\|\gamma\|=1} \Bigl(\int_{{\mathcal G}}{\mathbb E} |\xi_n(g)|^p |\gamma|(dg)\Bigr)^{1/p}
 \\
 &
 \leq \sup_{g\in {\mathcal G}}\|(P_n-P)(g)\|_{L_p({\mathbb P}_P)}
 \lesssim \sigma \sqrt{\frac{p}{n}} \vee U\frac{p}{n}.
 \end{align*}

 \qed
 \end{proof}
 
 Lemma \ref{prop_G_compact} implies the bound of Assumption \ref{assume_Bernstein_lin} with $\sigma(P)=\sigma, U(P)=U.$

It easily follows from bound \eqref{emp_pr} and Talagrand's concentration inequality for empirical processes that, for all $p\geq 1,$ 
\begin{align}
\label{emp_pr_L_p}
\Bigl\| \|P_n-P\|_{{\mathcal G}} \Bigr\|_{L_p({\mathbb P}_P)}\lesssim \sigma \sqrt{\frac{V}{n}} \vee U\frac{V}{n}\vee \sigma\sqrt{\frac{p}{n}} 
\vee U\frac{p}{n}.
\end{align}
Bound \eqref{emp_pr_L_p} implies the bound of Assumption \ref{assume_Bernstein_on_norm} with $d_1(P)=\sigma^2 V, d_2(P)= UV, \sigma(P)=\sigma,
U(P)=U.$ 

The claim now follows from the first bound of Theorem \ref{th_1_Bern_AA}.

\qed
\end{proof}
  
\end{example}

Other examples include estimation of functionals of sample covariance operators considered in Section \ref{func_cov_op}
and functionals of parameters of high-dimensional exponential families discussed in Section \ref{exp_fam_det}.


\section{Exponential families: general background, Bernstein inequalities and further results}
\label{exp_fam_det}

In this section, we provide a review of well known facts on exponential families (most of them, could be found in Chapter 2 of the book by Chencov \cite{Chencov}),
discuss Bernstein type bounds for estimators of the mean parameter and bounds on the $L_p$-errors of functional estimators.  

\subsection{Background on exponential families}
\label{background}

Let $(S,{\mathcal A})$ be a measurable space and let $\mu$ be a $\sigma$-finite measure on $(S,{\mathcal A}).$
Define probability measures $P_{\theta}, \theta \in \Theta\subset {\mathbb R}^N$ on the space $(S,{\mathcal A})$ to be absolutely continuous with respect to $\mu$ 
with densities
\begin{align}
\label{exp_fam}
p_{\theta}(x)=\frac{dP_{\theta}}{d\mu}(x) = \frac{1}{Z(\theta)} \exp\{\langle T(x), \theta\rangle\}, x\in S,
\end{align}
where $T:S\mapsto {\mathbb R}^N,$ $T(x)=(T_1(x),\dots, T_N(x)), x\in S$ is a measurable function, $\langle\cdot,\cdot\rangle$ is the canonical inner product in ${\mathbb R}^N$ and 
\begin{align*} 
Z(\theta) := \int_S \exp\{\langle T(x), \theta\rangle\} \mu(dx).
\end{align*}
We assume that 
\begin{align*}
\Theta:= \{\theta\in {\mathbb R}^N: Z(\theta)<+\infty\} \neq \emptyset. 
\end{align*}
It is well known and easy to check that $Z: {\mathbb R}^N \mapsto (0,+\infty]$ is a convex function, implying that $\Theta$ is a convex set. 

The statistical model $\{P_{\theta}: \theta \in \Theta\}$ is called {\it the exponential family} generated by statistic $T$ and measure $\mu.$  
Recall that $T(X), X\sim P_{\theta}, \theta \in \Theta$ is a sufficient statistic and it is complete when the interior of $\Theta$ is nonempty: ${\rm Int}\Theta\neq \emptyset.$ 
Clearly, the same exponential family could be generated by multiple choices of $T, \mu$ and parameter $\theta.$ 
For instance, one can obtain many different parametrizations of exponential family by applying invertible affine transformations 
to parameter $\theta$ or to statistic $T.$

Note that, 
for any fixed $\theta_0\in \Theta,$ we have 
\begin{align*}
\frac{dP_{\theta}}{dP_{\theta_0}}(x) =  \frac{Z(\theta_0)}{Z(\theta)} \exp\{\langle T(x), \theta-\theta_0\rangle\},
\end{align*}
implying that $\{P_{\theta}: \theta \in \Theta\}$ can be also viewed as an exponential family generated by $T$ and probability measure $P_{\theta_0}$
with parameter $\theta-\theta_0\in \Theta-\theta_0.$ Thus, without loss of generality, we can and will assume in what follows that $\mu$ is a probability 
measure, which implies that $0\in \Theta$ and $Z(0)=1.$ 

If $L:={\rm l.s.}(\Theta)$ is a proper subspace of ${\mathbb R}^N,$ one can replace statistic $T$ by its orthogonal projection onto $L.$
Moreover, using the coordinate representation of $L$ in an arbitrary orthonormal basis of $L,$ subspace $L$ can be identified with the Euclidean space 
${\mathbb R}^{{\rm dim}(L)}$ of dimension smaller than $N.$ Thus, without loss of generality, we will assume in what follows that ${\rm l.s.}(\Theta)={\mathbb R}^N.$
It is easy to see that such an exponential model is identifiable if the functions $1, T_1,\dots, T_N$ are linearly independent in the linear space of random variables 
on $(S,{\mathcal A}, \mu)$ and such a linear independence can be always achieved by further reducing the number of functions $T_1,\dots, T_N$ and the dimension 
of parameter $\theta.$ So, again without loss of generality, we assume in what follows that functions $1,T_1,\dots, T_N$ are linearly independent, implying the identifiability 
of the model $\{P_{\theta}:\theta \in \Theta\}.$ 

To summarize, we end up with the following assumption.

\begin{assumption}
\normalfont
\label{assume_exp_fam}
Suppose the exponential family \eqref{exp_fam} satisfies:
\begin{enumerate}[label=(\roman*)]
\item $\mu$ is a probability measure on $(S,{\mathcal A});$
\item ${\rm l.s.}(\Theta)={\mathbb R}^N;$
\item functions $1, T_1,\dots, T_N$ are linearly independent in the linear space of random variables 
on $(S,{\mathcal A}, \mu).$
\end{enumerate}
\end{assumption}

Let 
\begin{align*}
\psi (\theta) := \log Z(\theta) = \log \int_{S}\exp\{\langle T(x),\theta\rangle\}\mu(dx), \theta \in {\mathbb R}^N
\end{align*}
be the cumulant generating function of r.v. $T(X), X\sim \mu.$ The following properties of this function are 
well known:

\begin{itemize}
\item Function $\psi: {\mathbb R}^N \mapsto {\mathbb R}\cup {+\infty}$ is convex and lower semi-continuous.
\item Moreover, under Assumption \ref{assume_exp_fam} (iii), $\psi$ is strictly convex on $\Theta.$ 
\item $Z$ and $\psi$ are both analytic functions in the tube domain $\{z\in {\mathbb C}^N: {\rm Re}(z)\in {\rm Int}\Theta\}$
and their derivatives could be calculated by differentiation under the integral sign.
\item In particular, the following formulas hold
\begin{align*}
\psi'(\theta) = (\nabla \psi)(\theta) = {\mathbb E}_{\theta} T(X), \theta \in {\rm Int}\Theta
\end{align*}
and 
\begin{align*}
\psi^{''}(\theta) = {\mathbb E}_{\theta} (T(X)-{\mathbb E}_{\theta} T(X))\otimes (T(X)-{\mathbb E}_{\theta} T(x))= {\rm Cov}_{\theta}(T(X)), \theta \in {\rm Int}\Theta.
\end{align*}
\item
Under Assumption \ref{assume_exp_fam}, $\psi'$ is strictly monotone on ${\rm Int}\Theta,$ i.e., 
\begin{align*}
\langle \psi'(\theta_1)-\psi'(\theta_2), \theta_1-\theta_2\rangle>0, \theta_1,\theta_2\in {\rm Int}\Theta, \theta_1\neq \theta_2,
\end{align*}
and $\psi^{''}(\theta)$ is positively definite for $\theta \in {\rm Int}\Theta.$
\end{itemize}

Denote $\Psi (\theta):= \psi'(\theta)={\mathbb E}_{\theta} T(X), \theta \in {\rm Int}\Theta.$ It follows from the last property that $\Psi$ is a bijection between 
${\rm Int} \Theta$ and $\Psi ({\rm Int} \Theta),$ which allows one to view $t=\Psi(\theta)={\mathbb E}_{\theta} T(X)$ as an alternative parametrization of the sub-model $\{P_{\theta}: \theta \in {\rm Int}\Theta\}.$ Note that parameter $t={\mathbb E}_{\theta} T(X)$ is well defined for all $\theta\in \Theta_1,$ where 
\begin{align*}
\Theta_1 := \{\theta \in \Theta: {\mathbb E}_{\theta}|T_j(X)|<\infty, j=1,\dots, N\}.
\end{align*}
Clearly, $\Theta_1$ is a convex set and $\Theta_1\supset {\rm Int}\Theta.$
Thus, the mapping $\Psi$ can be extended to the set $\Theta_1$ and it could be proved that this extension (which will be still denoted by $\Psi$)
is a strictly monotone map on $\Theta_1$ (see \cite{Chencov}, Chapter 4), implying that $\Psi$ is a bijection between $\Theta_1$ and its image 
$\Psi(\Theta_1).$ Moreover, $\Psi$ is infinitely differentiable (even analytic) in ${\rm Int}\Theta$ with positive Jacobian at each point and 
its inverse is also infinitely differentiable (\cite{Chencov}, Chapter 4). The parameter $t=\Psi(\theta)={\mathbb E}_{\theta} T(X)$ of exponential 
family is usually called {\it the mean parameter}, but we will rather call it {\it the natural parameter} (following the terminology of Chencov \cite{Chencov}).

Let ${\rm csupp}(\nu)$ denote {\it the convex support} of a Borel probability measure $\nu$ on ${\mathbb R}^N$ (that is, the smallest closed convex set of full measure $\nu,$
or, in other words, the intersection of all closed half-spaces of full measure $\nu$). Since all measures $P_{\theta}, \theta \in \Theta$ are 
equivalent to $\mu,$ it is easy to see that ${\rm csupp}(P_{\theta}\circ T^{-1})= {\rm csupp}(\mu\circ T^{-1}), \theta \in \Theta.$ The following properties can be also found in \cite{Chencov}, Chapter 4:

\begin{enumerate}
\item $\Psi(\Theta_1)\subset {\rm Int\ csupp}(\mu\circ T^{-1});$
\item $\Psi (\Theta_1)= {\rm Int\ csupp}(\mu\circ T^{-1})$ if and only if $\Theta_1={\rm Int}\Theta.$ 
\end{enumerate}

It will be said that the exponential model is {\it regular} iff $\Theta_1={\rm Int} \Theta$ 
(or, equivalently, $\Psi (\Theta_1)= {\rm Int\ csupp}(\mu\circ T^{-1})$). For regular models, 
the mapping $\Psi$ is a $C^{\infty}$-diffeomorphism between open convex sets $\Theta_1={\rm Int}\Theta$ and  
${\rm Int\ csupp}(\mu\circ T^{-1}).$ 
In this case, it makes sense to study the problem 
of estimation of smooth functionals of natural parameter $t$ in order to develop the methods of estimation of smooth 
functionals of canonical parameter $\theta$ (with the relationship between the two problems described in terms of diffeomorphism $\Psi$). 
In particular, one can exploit the fact that there exists a simple unbiased estimator of natural parameter $t$ based on i.i.d. observations 
$X_1,\dots, X_n\sim P_{\Psi^{-1}(t)},$ namely, the sample mean $\bar T_n:= \frac{T(X_1)+\dots+T(X_n)}{n}.$  
It is also known that, if $\bar T_n\in \Psi(\Theta_1),$ then $\hat \theta_n:=\Psi^{-1}(\bar T_n)$ is the unique maximum likelihood 
estimator of parameter $\theta\in \Theta_1.$ Note also that the covariance ${\rm Cov}_{\theta}(T(X))=\Psi'(\theta)$ coincides with the Fisher 
information matrix $I(\theta)$ for the model $\{P_{\theta}: \theta\in \Theta_1\}$ whereas ${\mathcal I}(t):=I(\Psi^{-1}(t))^{-1}$ is the Fisher 
information for the model $P_{\Psi^{-1}(t)}, t\in \Psi(\Theta_1).$ This immediately implies that $\bar T_n$ is an efficient estimator of natural parameter $t$
and it provides a way to design asymptotically efficient estimators of $f(t).$

It will be convenient for our purposes to assume now that statistic $T$ takes values in an $N$-dimensional 
linear normed space $E$ with the dual space $E^{\ast}.$ Recall that by choosing by-orthogonal bases $e_1,\dots, e_N$ and $f_1,\dots, f_N$ 
in $E$ and $E^{\ast},$ respectively, and using these bases to introduce coordinates in $E$ and $E^{\ast},$ we can identify these spaces 
with the two copies of coordinate space ${\mathbb R}^N$ equipped with the corresponding norms.  
With parameter $\theta\in E^{\ast},$ densities $p_{\theta}$ are still given by formula \eqref{exp_fam} and the parameter space 
$\Theta=\{\theta\in E^{\ast}: Z(\theta)<+\infty\}$ is now a convex subset of $E^{\ast}.$
We can still define $T_j(x):=\langle T(x), f_j\rangle, j=1,\dots, N.$
Assumption \ref{assume_exp_fam} holds with ${\rm l.s.}(\Theta)=E^{\ast}$ in part (ii). We still have the same properties of cumulant generating 
function $\psi: E^{\ast}\mapsto {\mathbb R}\cup {+\infty}$ with $\psi'(\theta)$ now taking values in space $E$ and $\psi^{''}(\theta)$ being a symmetric 
linear operator from $E^{\ast}$ into $E$ (the covariance operator of $T(X)$). Thus, $\Psi$ is now a map from $\Theta\subset E^{\ast}$ into $E,$
the natural parameter $t$ takes values in $E,$ $\Theta_1$ is a subset of $E^{\ast}$ and ${\rm Int\ csupp}(\mu\circ T^{-1})$ is a subset of $E.$ 
As before, the regularity means that $\Psi (\Theta_1)= {\rm Int\ csupp}(\mu\circ T^{-1}),$ which holds if and only if $\Theta_1={\rm Int}\Theta.$

\begin{example}
\label{example_spherical}
\normalfont
Consider an exponential model 
\begin{align*}
P_{\theta}(dx)= \frac{1}{Z(\theta)} \exp\{\langle x, \theta\rangle\}\mu(dx), x\in {\mathbb R}^N,
\end{align*}
where $\mu$ is a Borel probability measure on ${\mathbb R}^N$ invariant with respect to the group of orthogonal transformations. 
In this case, the function $Z(\theta)$ is also orthogonally invariant and 
$
Z(\theta) = \tilde Z(\|\theta\|),
$
where
\begin{align*}
\tilde Z(\rho):= \int_{{\mathbb R}^N} e^{\rho x_1} \mu(dx), \rho \geq 0.
\end{align*}
Let 
\begin{align*}
r:= \sup\{\rho \geq 0: \tilde Z(\rho)<\infty\}.
\end{align*}
Clearly, if $r=+\infty,$ then $\Theta:=\{\theta\in {\mathbb R}^N: Z(\theta)<\infty\}={\mathbb R}^N$ and, if $r=0,$ then 
$\Theta=\{0\}.$ 
Otherwise, for $r\in (0,+\infty),$ 
$\Theta$ is either the open ball $B(0,r)$ with center $0$ and radius $r$ 
(if $\tilde Z(r)=+\infty$), or closed ball $\bar B(0,r)$ (if $\tilde Z(r)<+\infty$). We will assume that $r>0,$ then 
${\rm Int} \Theta=B(0,r)\neq \emptyset.$ 

Clearly, 
$\psi(\theta)=\varphi (\|\theta\|),$
where $\varphi(\rho)= \log \tilde Z(\rho), \rho\geq 0$ is a strictly convex function on $[0,r).$ Denote 
\begin{align*}
\Phi(\rho):=\varphi'(\rho), \rho \in [0,\rho),
 \end{align*}
which is a strictly increasing continuous function. Then we have 
\begin{align*}
\Psi(\theta) = \Phi(\|\theta\|)\frac{\theta}{\|\theta\|}, \theta\in B(0,r)
\end{align*}
and 
\begin{align*}
\Psi^{-1}(t) = \Phi^{-1}(\|t\|)\frac{t}{\|t\|}, t\in B(0, \Phi(r)).
\end{align*}
Note that both $\Psi$ and $\Psi^{-1}$ are spherically symmetric strictly monotone vector fields on the open balls $B(0,r)$
and $B(0,\Phi(r)),$ respectively. It is also easy to see that $\Phi(0)=0$ (otherwise, $\Psi$ could not be a one-to-one mapping).

Since measure $\mu$ is orthogonally invariant, its convex support is a closed ball $\bar B(0,r^{\ast})$ with center $0$ and radius 
$r^{\ast}\in [0,\infty].$  If $r^{\ast}<\infty,$ then $r=+\infty$ and $\Theta={\mathbb R}^N.$
In this case, $\Theta_1=\Theta= {\rm Int}\Theta={\mathbb R}^N,$ so the exponential model is regular and 
$\Psi ({\mathbb R}^N)=B(0,r^{\ast})={\rm Int}\ {\rm csupp}(\mu).$
On the other hand, if $r^{\ast}=+\infty,$ there are two possibilities. If $\tilde Z(r)=+\infty,$ then $\Theta=B(0,r)={\rm Int} \Theta,$
so again $\Theta_1={\rm Int} \Theta,$ the model is regular 
and  $\Psi(B(0,r))= B(0,r^{\ast})={\mathbb R}^N.$ If $\tilde Z(r)<\infty,$ then $\Theta=\bar B(0,r)$ and 
$\Theta_1$ could be 
either $B(0,r),$ or $\bar B(0,r),$ depending on divergence or convergence of the integral 
\begin{align*}
\int_{x\in {\mathbb R}^N, x_1\geq 0} x_1 e^{r x_1} \mu (dx).
\end{align*}
If the integral diverges, then $\Theta_1=B(0,r)= {\rm Int} \Theta,$ the model is regular and $\Psi(B(0,r))={\mathbb R}^N.$
Otherwise, $\Psi (B(0,r))=B(0,\Phi(r))$ is a ball of finite radius in ${\mathbb R}^N.$  
Clearly, in all the regular cases, $r^{\ast}=\Phi(r),$ otherwise $r^{\ast}>\Phi(r).$

The smoothness properties of mappings $\Psi$ and $\Psi^{-1}$ clearly depend only on the smoothness of function 
$\Phi :{\mathbb R}_+\mapsto {\mathbb R}_+$ and its inverse $\Phi^{-1}.$ Since the derivatives of $\Phi$ can be represented in terms 
of integrals $\int_{{\mathbb R}^N} x_1^k e^{\rho x_1}\mu(dx),$ the smoothness of $\Phi$ and $\Phi^{-1}$ is related to integrability 
properties of measure $\mu.$

Note that, when $\mu$ is a uniform distribution on the unit sphere $S^{N-1},$ the exponential family discussed in this example is called von Mises-Fisher distribution 
and it provides an important statistical model for directional data.
\end{example}

\subsection{Bernstein type inequalities for exponential families}

We will consider an exponential family
\begin{align*} 
P_{\theta}(dx) = \frac{1}{Z(\theta)} \exp\{\langle T(x), \theta\rangle\} \mu(dx), \theta\in \Theta\subset {\mathbb R}^N
\end{align*}
under the assumptions and notations of Section \ref{background}.
Denote 
\begin{align*}
\Sigma_{\theta}:= \psi''(\theta) = {\mathbb E}_{\theta}(T(X)-{\mathbb E}_{\theta} T(X))\otimes (T(X)-{\mathbb E}_{\theta}T(X)), \theta \in {\rm Int}\Theta.
\end{align*}
For $\theta\in {\rm Int}\Theta$ and $u\in E^{\ast},$ let
\begin{align*}
\rho(\theta;u):=\sup\{r>0:\theta+r u\in {\rm Int}\Theta\}.
\end{align*}
Clearly, $\rho(\theta,u)>0, \theta\in {\rm Int}\Theta, u\in E^{\ast}.$ 
Also define 
\begin{align*}
\sigma^2(\theta;u):= \sup_{\lambda \in (0,\rho(\theta;u))}\frac{1}{\lambda} \int_{0}^{\lambda}\langle \Sigma_{\theta+s u}u,u\rangle ds, 
\ \theta\in {\rm Int}\Theta, u\in E^{\ast}.
\end{align*}

We will start with the following simple Bernstein type inequality (see \cite{Siri, Trivellato} for other results on Bernstein type bounds for exponential models).

\begin{proposition}
For all $\theta \in {\rm Int}\Theta, u\in E^{\ast}$ and $t>0$ with probability at least $1-e^{-t},$
\begin{align*}
    \left|\langle \bar T_n-\mathbb{E}_\theta T(X), u\rangle\right|\leq
\sqrt{2}\sigma(\theta;u)\sqrt{\frac{t}{n}}\bigvee \frac{2}{\rho(\theta;u)}\frac{t}{n}.
\end{align*}
\end{proposition}

\begin{proof}
Recall that $\psi(\theta)= \log Z(\theta)$ and $\psi'(\theta)=\Psi(\theta)={\mathbb E}_{\theta} T(X).$
This implies 
\begin{align*}
   \gamma(\lambda):= \log\mathbb{E}_\theta e^{\lambda\langle T(X)-\mathbb{E}_\theta T(X), u\rangle}
    &=\log \frac{Z(\theta+\lambda u)}{Z(\theta)} -\lambda \langle\mathbb{E}_\theta T(X),u\rangle\\
    &=\psi(\theta+\lambda u)-\psi(\theta)-\lambda \langle \psi'(\theta),u\rangle.
\end{align*}    
Clearly, 
\begin{align*}
\gamma'(\lambda)= \langle \psi' (\theta+\lambda u)-\psi'(\theta), u\rangle
\end{align*}
and 
\begin{align*}
\gamma''(\lambda) = \langle \psi''(\theta+\lambda u)u,u\rangle = \langle \Sigma_{\theta+\lambda u}u,u\rangle. 
\end{align*}
Since $\gamma(0)=0$ and $\gamma'(0)=0,$ we get 
\begin{align*}
\gamma(\lambda) &= \int_0^{\lambda} \int_{0}^s \gamma''(\tau) d\tau ds = \int_0^{\lambda} \int_{0}^s \langle \Sigma_{\theta+\tau u}u,u\rangle d\tau ds 
\\
&
\leq \int_0^{\lambda} \sigma^2(\theta;u) s ds = \frac{\sigma^{2}(\theta;u)\lambda^2}{2}, \lambda \in (0,\rho(\theta;u)).
\end{align*}
The above inequality implies that, for $\lambda \in (0, n\rho(\theta;u)),$
\begin{align*}
\mathbb{E}_\theta \exp\Bigl\{\lambda\langle \bar T_n-\mathbb{E}_\theta T(X), u\rangle\Bigr\} &=  
\prod_{j=1}^n {\mathbb E}_{\theta} \exp\Bigl\{\frac{\lambda}{n} \langle T(X_j)-{\mathbb E}_{\theta} T(X), u\rangle\Bigr\}
\\
&
=\exp\Bigl\{n\gamma\Bigl(\frac{\lambda}{n}\Bigr)\Bigr\} \leq \exp\Bigl\{\frac{\sigma^{2}(\theta;u)\lambda^2}{2n}\Bigr\},
\end{align*}
implying that, for all $\delta>0,$ 
\begin{align*}
\mathbb{P}_\theta\{\langle \bar T_n-\mathbb{E}_\theta T(X), u\rangle\geq \delta\}\leq 
\exp\Bigl\{\frac{\sigma^{2}(\theta;u)\lambda^2}{2n} - \lambda \delta\Bigr\}.
\end{align*}
Minimizing the last bound over $\lambda \in (0, n\rho(\theta;u)),$ we get
\begin{equation}
\nonumber
    \mathbb{P}_\theta\{\langle \bar T_n-\mathbb{E}_\theta T(X), u\rangle\geq \delta\}\leq 
    \begin{cases}
    \begin{aligned}
    &\exp\left\{-\frac{n \delta^2}{2 \sigma^2(\theta;u)}\right\},\ \  \delta<\rho(\theta;u) \sigma^2 (\theta;u)\\
    &\exp\left\{-\frac{n\rho(\theta;u)\delta}{2}\right\},\ \ \delta\ge\rho(\theta;u)\sigma^2 (\theta;u).  
    \end{aligned}
    \end{cases}
\end{equation}
It remains to take 
\begin{align*}
\delta:= \sqrt{2}\sigma(\theta;u)\sqrt{\frac{t}{n}}\bigvee \frac{2}{\rho(\theta;u)}\frac{t}{n}
\end{align*}
to complete the proof.

\qed
\end{proof}

One could easily see from the proof that the bound also holds with $\rho(\theta;u)$ replaced by an arbitrary $0<\rho<\rho(\theta;u)$ and $\sigma^2(\theta;u)$ replaced by 
\begin{align*}
\sigma_{\rho}^2(\theta;u):= \sup_{\lambda \in (0,\rho)}\frac{1}{\lambda} \int_{0}^{\lambda}\langle \Sigma_{\theta+s u}u,u\rangle ds, 
\ \theta\in {\rm Int}\Theta, u\in E^{\ast}.
\end{align*}
If $B(\theta,\rho)=\{y\in E^{\ast}: \|y-\theta\|<\rho\}\subset {\rm Int}\Theta,$ then $\rho(\theta;u)>\rho$ for all $u\in E^{\ast}, \|u\|\leq 1.$
Also, for all such $u,$
\begin{align*}
\sigma_{\rho}^2(\theta;u) \leq \sup_{\theta'\in B(\theta;\rho)}\|\Sigma_{\theta'}\|=: \sigma_{\rho}^2(\theta).
\end{align*}
Therefore, under the assumption that $B(\theta,\rho)\subset {\rm Int}\Theta,$ for all $u\in E^{\ast}, \|u\|\leq 1$ and for all  $t>0$ with probability at least $1-e^{-t},$
\begin{align*}
    \left|\langle \bar T_n-\mathbb{E}_\theta T(X), u\rangle\right|\leq
\sqrt{2}\sigma_{\rho}(\theta)\sqrt{\frac{t}{n}}\bigvee \frac{2}{\rho}\frac{t}{n},
\end{align*}
which also implies that, for all $p\geq 1,$ 
\begin{align}
\label{bd_UVW}
    \sup_{\|u\|\leq 1}\left\|\langle \bar T_n-\mathbb{E}_\theta T(X), u\rangle\right\|_{L_p}\lesssim 
\sigma_{\rho}(\theta)\sqrt{\frac{p}{n}}\bigvee \frac{1}{\rho}\frac{p}{n}.
\end{align}

The next statement follows from Proposition \ref{Bernstein_on_norms}. 

\begin{proposition}
\label{prop_about_norm}
Suppose Assumption \ref{assum_norm_d} on the norm  of Banach space $E$ holds.
Also assume that for some $\rho>0,$ $B(\theta,\rho)\subset {\rm Int}\Theta.$
Then,  for all  $p\geq 1$ 
\begin{align}
\label{bd_UVW_sup}
    \left\|\|\bar T_n-\mathbb{E}_\theta T(X)\|\right\|_{L_p}\lesssim 
C\sigma_{\rho}(\theta)\sqrt{\frac{d}{n}}\vee \frac{C}{\rho}\frac{d}{n} \vee 
C\sigma_{\rho}(\theta)\sqrt{\frac{p}{n}}\vee \frac{C}{\rho}\frac{p}{n}. 
\end{align}
\end{proposition}

Finally, note that, if $\Theta=E^{\ast},$ we can take $\rho=\infty$ and set $\sigma_{\infty}^2:= \sup_{\theta\in E^{\ast}}\|\Sigma_{\theta}\|.$
In this case, bounds \eqref{bd_UVW} and \eqref{bd_UVW_sup} become subgaussian: 
\begin{align}
\label{bd_UVW''}
 \sup_{\|u\|\leq 1}\left\|\langle \bar T_n-\mathbb{E}_\theta T(X), u\rangle\right\|_{L_p}\lesssim 
\sigma_{\infty} \sqrt{\frac{p}{n}}
\end{align}
and 
\begin{align}
\label{bd_UVW_sup''}
    \left\|\|\bar T_n-\mathbb{E}_\theta T(X)\|\right\|_{L_p}\lesssim 
C\sigma_{\infty}\Bigl(\sqrt{\frac{d}{n}}\vee \sqrt{\frac{p}{n}}\Bigr). 
\end{align}

\begin{example}
\label{matrix_exponential}
\normalfont
An important example of exponential model with matrix parameter is 
\begin{align*}
P_{\theta}(dx) = \frac{1}{Z(\theta)} \exp\{\langle x\otimes x, \theta\rangle\}\mu(dx), x\in {\mathbb R}^d, 
\end{align*}
where the parameter $\theta$ is a symmetric $d\times d$ matrix and $\mu$ is a Borel probability measure in 
${\mathbb R}^d.$ The statistic generating this exponential family is $T(x)= x\otimes x.$ It also takes values in the 
space of symmetric matrices. This space will be equipped 
with the operator norm and the resulting Banach space will be denoted by $E.$ Its dual space $E^{\ast}$ is again the linear space of 
symmetric $d\times d$ matrices equipped with the nuclear norm.  With these conventions, statistic $T(x)=x\otimes x$ takes values in $E$
and the parameter space $\Theta=\{\theta: Z(\theta)<+\infty\}$ is a convex subset of $E^{\ast}.$ For Assumption \ref{assume_exp_fam} to be satisfied, it would be enough 
to assume that the functions $1, x_i x_j: 1\leq i\leq j\leq d$ are linearly independent on ${\mathbb R}^d$ and that 
${\rm l.s.}(\Theta)=E^{\ast}.$ 
Note that for our exponential model 
$
\Psi (\theta)={\mathbb E}_{\theta}(X\otimes X), X\sim P_{\theta}, \theta \in \Theta.
$

In addition to mean zero normal model (with non-standard parametrization), this class of exponential families 
includes Bingham distribution for directional data (with $\mu$ being the uniform distribution on the unit sphere $S^{d-1}$)
and Ising model (with $\mu$ being the uniform distribution on the binary cube $\{-1,1\}^d$).
For these two models, the support of measure $\mu$ is bounded, which implies that $\Theta=E^{\ast}$ (this, of course, could be the case even if the support of measure $\mu$ is not bounded).
The condition $\Theta=E^{\ast}$ also implies that $\Theta_1=E^{\ast}$ and, as a consequence of the results by Chencov discussed 
in Section \ref{background}, ${\rm Int}\ {\rm csupp}(\mu\circ T^{-1})=\Psi (\Theta_1).$ 

Let 
$
\bar T_n :=  n^{-1} \sum_{j=1}^n X_j \otimes X_j.
$
Since $E^{\ast}$ is the space of symmetric matrices equipped with the nuclear norm, using the spectral decomposition of matrix $U\in E^{\ast},$ it is easy to check that 
\begin{align*}
\sup_{U\in E^{\ast}, \|U\|\leq 1} \Bigl\|\langle\bar T_n - {\mathbb E}_{\theta} T(X), U\rangle\Bigr\|_{L_p}
=
\sup_{y\in {\mathbb R}^d, \|y\|\leq 1}
\Bigl\|\langle(\bar T_n - {\mathbb E}_{\theta} T(X))y, y\rangle\Bigr\|_{L_p}.
\end{align*}
For $\Sigma_{\theta}:= {\rm cov}_{\theta}(T(X)),$ we have 
\begin{align*}
&
\|\Sigma_{\theta}\|=\sup_{U\in E^{\ast}, \|U\|\leq 1}\langle \Sigma_{\theta}U,U\rangle = \sup_{y,y'\in {\mathbb R}^d, \|y\|, \|y'\|\leq 1}
\langle \Sigma_{\theta} (y\otimes y), (y'\otimes y')\rangle
\\
&
\leq \sup_{y\in {\mathbb R}^d, \|y\|\leq 1}{\mathbb E}_{\theta}\langle X\otimes X, y\otimes y\rangle^2
= \sup_{y\in {\mathbb R}^d, \|y\|\leq 1}{\mathbb E}_{\theta}\langle X,y\rangle^4.
\end{align*}
Therefore,
\begin{align*}
\sigma_r^2 (\theta) \leq \sup_{\theta'\in B(\theta, r)}\sup_{y\in {\mathbb R}^d, \|y\|\leq 1}{\mathbb E}_{\theta'}\langle X,y\rangle^4:= \bar \sigma_r^2(\theta)
\end{align*}
and 
\begin{align*}
\sigma_{\infty}^2  \leq \sup_{\theta\in \Theta}\sup_{y\in {\mathbb R}^d, \|y\|\leq 1}{\mathbb E}_{\theta}\langle X,y\rangle^4:= 
\bar \sigma_{\infty}^2.
\end{align*}

If $B(\theta,r)\subset {\rm Int}\Theta,$ then the following Bernstein type bound holds for all $y\in {\mathbb R}^d$ with $\|y\|\leq 1$ and all $p\geq 1:$
\begin{align*}
\Bigl\|\langle(\bar T_n - {\mathbb E}_{\theta} T(X))y, y\rangle\Bigr\|_{L_p({\mathbb P}_{\theta})}
\lesssim \bar \sigma_r(\theta) \sqrt{\frac{p}{n}} \vee \frac{1}{r} \frac{p}{n}.
\end{align*}
Moreover, if ${\mathcal N}$ is a $1/4$-net for $S^{d-1}$ of cardinality $\leq 9^d,$ then 
\begin{align*}
\|A\| \leq 2\max_{y\in {\mathcal N}}|\langle Ay, y\rangle|, A\in E.
\end{align*}
Thus, the bound of Proposition \ref{prop_about_norm} with ${\mathcal M}:= \{y\otimes y: y\in {\mathcal N}\}$ implies that 
\begin{align*}
\Bigl\|\|\bar T_n - {\mathbb E}_{\theta} T(X)\|\Bigr\|_{L_p({\mathbb P}_{\theta})}  
\lesssim  \bar \sigma_r(\theta) \sqrt{\frac{d}{n}} \vee \frac{1}{r} \frac{d}{n}\vee \bar \sigma_r(\theta) \sqrt{\frac{p}{n}} \vee \frac{1}{r} \frac{p}{n}.
\end{align*}

\end{example}

\subsection{Bounds on the $L_p$-errors in functional estimation}
\label{sec_L_p_exp}

Since, for regular exponential families, $\Psi$ is a $C^{\infty}$-diffeomorphism between ${\rm Int}\Theta$ and its image $\Psi ({\rm Int} \Theta)$
(which coincides with ${\rm Int}\ {\rm csupp}(\mu\circ T^{-1})$) and, for natural parameter $t=\Psi(\theta)$ there exists a simple estimator $\bar T_n,$
it is reasonable to reduce the problem of estimation of functional $f(\theta)$ based on i.i.d. observations $X_1,\dots, X_n\sim P_{\theta}$
to the problem of estimation of functional $(f\circ \Psi^{-1})(t)$ of natural parameter $t.$ With such an approach, the smoothness conditions 
will be imposed on $f\circ \Psi^{-1},$ which, under proper smoothness of map $\Psi^{-1},$ could be reduced to the smoothness of $f.$
We use the sample split described in Section \ref{Main_results} to define estimators $\bar T_{n^{(0)}}$ and $\bar T_{n^{(k)}_j}, j=1,\dots, k, k=1,\dots, m$
and, based on these estimators, to construct estimator $T_{f\circ \Psi^{-1}}(X_1,\dots, X_n)$ (see \eqref{basic_T_f}) and its truncated version 
$\tilde T_{f\circ \Psi^{-1}, M}(X_1,\dots, X_n)$  (see \eqref{basic_T_fM}).

The next result is a corollary of Theorem \ref{th_1_Bern_AA}. Of course, to apply this theorem, one has to consider the identifiable model 
${\mathcal P}:= \{P_{\theta}: \theta \in \Theta\}$ and set $\theta(P):=\theta$ for $P=P_{\theta}\in {\mathcal P}.$ Denote 
$\Sigma_{\theta}:={\rm cov}_{\theta}(T(X))=\Psi'(\theta), \theta\in {\rm Int}\Theta$ and 
\begin{align*}
\sigma_r^2(\theta):=\sup_{\theta'\in B(\theta,r)\cap {\rm Int}\Theta}\|\Sigma_{\theta'}\|, r>0.
\end{align*}
For $r=\infty,$ $\sigma_{\infty}^2:=\sup_{\theta \in {\rm Int}\Theta} \|\Sigma_{\theta}\|.$

\begin{theorem}
\label{estim_exp_fam}
Suppose the norm of space $E$ satisfies Assumption \ref{assum_norm_d}. Let $f:E^{\ast}\mapsto {\mathbb R}$ be a uniformly bounded functional.
Suppose that $f\circ \Psi^{-1}$ is $m$ times continuously differentiable in a neighborhood $U=B(t,\delta)$ for some $t\in \Psi ({\rm Int}\Theta)$ and $\delta\in (0,1],$ 
and, moreover,
$(f\circ \Psi^{-1})^{(m)}$ satisfies the Hölder condition with exponent $\rho\in(0,1]$ in $U.$ 
Let $\theta:= \Psi^{-1}(t)$ and suppose that, for some $r>0,$ $B(\theta; r)\subset {\rm Int}\Theta.$

(i) Assume that $n^{(0)}\asymp_m n$ and $n_j^{(k)}\asymp_m n$, $j=1,\cdots,k$, $k=1,\cdots,m.$
Moreover, assume that for a sufficiently large numerical constant $C'>0,$
\begin{align*}
C'\left(C\sigma_{r}(\theta)\sqrt{\frac{d}{n}}\vee \frac{C}{r}\frac{d}{n}\right)\leq \delta.
\end{align*}
Then, with some constant $c_1>0$ depending on $m, C,$ for $M\geq \|f\|_{L_{\infty}}$ and $p\geq 1,$
\begin{align*}
   &
    \left\|\tilde{T}_{f\circ \Psi^{-1}}(X_1,\cdots,X_n)-f(\theta)\right\|_{L_p(\mathbb{P}_\theta)}
   \\
   &
   \lesssim_{s, C} 
   \max_{0\leq k\leq m} \|(f\circ \Psi^{-1})^{(k)}\|_{L_\infty(U)} \left(\sigma_{r}(\theta)\sqrt{\frac{p}{n}}\vee\frac{1}{r}\frac{p}{n}\right)  
   \\
   &  
  +\|(f\circ \Psi^{-1})^{(m)}\|_{{\rm Lip}_\rho(U)} \left(\sigma_{r}^s(\theta) \Bigl(\frac{d}{n}\Bigr)^{s/2}\vee \frac{1}{r^s}\Bigl(\frac{d}{n}\Bigr)^s\vee 
  \sigma_{r}^s(\theta) \Bigl(\frac{p}{n}\Bigr)^{s/2}\vee \frac{1}{r^s}\Bigl(\frac{p}{n}\Bigr)^s
  \right)
  +
  \\
  &
  +(\|f\|_{L_{\infty}}+M) \exp\Bigl\{-c_1 \frac{n}{p}\Bigl(\frac{\delta^2}{\sigma_r^2(\theta)}\wedge \delta r\Bigr)\Bigr\}.
 \end{align*}
 
  (ii) Assume that, for all $j=1,\dots, k, k=1,\dots, m,$ $n_j^{(k)}\asymp_m n.$
If $M\geq \|f\|_{\infty}+ \|f\|_{{\rm Lip}}\delta,$ then 
\begin{align*}
    &\left\|\tilde T_{f\circ \Psi^{-1}}(X_1,\cdots,X_n)-f(\theta)- \langle \bar T_1^{(1)}-\Psi(\theta), (f\circ \Psi^{-1})'(t)\rangle\right\|_{L_p(\mathbb{P}_{\theta})}
    \\
    &
    \lesssim_{s, C} \max_{2\leq k\leq m} \|(f\circ \Psi^{-1})^{(k)}\|_{L_\infty(U)} \left(\sigma_r(\theta)\sqrt{\frac{p}{n}}\vee \frac{1}{r}\frac{p}{n}\right) 
    \left(\sigma_r(\theta)\sqrt{\frac{d}{n^{(0)}}}\vee \frac{1}{r}\frac{d}{n^{(0)}}\vee \sigma_r(\theta)\sqrt{\frac{p}{n^{(0)}}}\vee \frac{1}{r}\frac{p}{n^{(0)}}\right) 
 \\
 &
 +\|(f\circ \Psi^{-1})^{(m)}\|_{{\rm Lip}_\rho(U)} \left(\sigma_r^s(\theta)\Bigl(\frac{d}{n^{(0)}}\Bigr)^{s/2}\vee \frac{1}{r^s}\Bigl(\frac{d}{n^{(0)}}\Bigr)^s\vee  \sigma_r^s(\theta) \Bigl(\frac{p}{n^{(0)}}\Bigr)^{s/2}\vee \frac{1}{r^s}\Bigl(\frac{p}{n^{(0)}}\Bigr)^s     
\right)
 \\
 &
 + 
\left(\|f\|_{L_\infty}+M + \|(f\circ \Psi^{-1})'(t)\|\left(\sigma_r(\theta)\sqrt{\frac{d}{n}}\vee \frac{1}{r}\frac{d}{n}\vee \sigma_r(\theta)\sqrt{\frac{p}{n}}\vee \frac{1}{r}\frac{p}{n}
\right) \right)
  \exp\Bigl\{-c_1 \frac{n^{(0)}}{p}\Bigl(\frac{\delta^2}{\sigma_r^2(\theta)}\wedge \delta r\Bigr)\Bigr\}.
\end{align*}    
 \end{theorem}

\vskip 2mm

\begin{remark}
\label{exp_rem_1}
\normalfont
Note that, if $\Theta=E^{\ast},$ we can set $r:=\infty.$ Then, 
under the assumption that, for a sufficiently large constant $C'>0,$
$
C' C\sigma \sqrt{\frac{d}{n}}\leq \delta,
$
the following bound holds:
\begin{align*}
   &
    \left\|\tilde{T}_{f\circ \Psi^{-1}}(X_1,\cdots,X_n)-f(\theta)\right\|_{L_p(\mathbb{P}_\theta)}
   \lesssim_{s, C} 
   \max_{0\leq k\leq m} \|(f\circ \Psi^{-1})^{(k)}\|_{L_\infty(U)} \sigma \sqrt{\frac{p}{n}}  
\\
&
   +\|(f\circ \Psi^{-1})^{(m)}\|_{{\rm Lip}_\rho(U)} \sigma_{\infty}^s  \Bigl(\Bigl(\frac{d}{n}\Bigr)^{s/2}\vee \Bigl(\frac{p}{n}\Bigr)^{s/2}\Bigr)
   + (\|f\|_{L_{\infty}}+M) \exp\Bigl\{-c_1 \frac{n}{p}\frac{\delta^2}{\sigma_{\infty}^2}\Bigr\}. 
   \end{align*}
\end{remark}

\begin{remark}
\label{exp_rem_2}
\normalfont
Under the assumptions that, for large enough constant $C'>0,$ 
\begin{align*}
r\geq \frac{1}{\sigma_r(\theta)}\sqrt{\frac{d}{n}}\ {\rm and}\ C'C\sigma_r(\theta)\sqrt{\frac{d}{n}}\leq \delta \leq 1,
\end{align*}
the first bound of the theorem takes the following form:
\begin{align*}
   &
    \left\|\tilde{T}_{f\circ \Psi^{-1}}(X_1,\cdots,X_n)-f(\theta)\right\|_{L_p(\mathbb{P}_\theta)}
   \\
   &
   \lesssim_{s, C} 
   \max_{0\leq k\leq m} \|(f\circ \Psi^{-1})^{(k)}\|_{L_\infty(U)} \sigma_{r}(\theta)\sqrt{\frac{p}{n}}\Bigl(1 \vee \sqrt{\frac{p}{d}}\Bigr) 
   \\
   &  
  +\|(f\circ \Psi^{-1})^{(m)}\|_{{\rm Lip}_\rho(U)} \left( \sigma_{r}^s(\theta) \Bigl(\frac{d}{n}\Bigr)^{s/2}\vee \sigma_{r}^s(\theta) \Bigl(\frac{p}{n}\Bigr)^{s/2} \Bigl(1 \vee \sqrt{\frac{p}{d}}\Bigr)^s
 \right)
  +
  \\
  &
 + (\|f\|_{L_{\infty}}+M) \exp\Bigl\{-c_1 \frac{\sqrt{nd}}{p} \frac{\delta}{\sigma_r(\theta)}\Bigr\}.
 \end{align*}
 Moreover, if for large enough constant $C'>0,$
 \begin{align*}
r\geq \frac{1}{\sigma_r(\theta)}\sqrt{\frac{d}{n}}\ {\rm and}\ C'C\sigma_r(\theta)\sqrt{\frac{d}{n}}\Bigl(1\vee \frac{\log \frac{n}{\sigma_r^2(\theta)d}}{d}\Bigr)\leq \delta \leq 1,
\end{align*}
then, for $M=\|f\|_{L_{\infty}},$
\begin{align}
\label{bd_exp_exp}
   &
    \left\|\tilde{T}_{f\circ \Psi^{-1}}(X_1,\cdots,X_n)-f(\theta)\right\|_{L_2(\mathbb{P}_\theta)}
   \lesssim_{s, C} 
  \|f\circ \Psi^{-1}\|_{C^s(U)} \Bigl[\frac{\sigma_{r}(\theta)}{\sqrt{n}}
 \vee 
  \sigma_{r}^s(\theta)\Bigl(\frac{d}{n}\Bigr)^{s/2}\Bigr].
 \end{align}
 \end{remark}

\begin{remark}
\label{exp_rem_3}
\normalfont
If 
\begin{align*}
r\geq \frac{1}{\sigma_r(\theta)}\sqrt{\frac{d}{n^{(0)}}}\ {\rm and}\ C'C\sigma_r(\theta)\sqrt{\frac{d}{n^{(0)}}}\leq \delta \leq 1,
\end{align*}
then the second bound of the theorem takes the form
\begin{align*}
    &\left\|\tilde T_{f\circ \Psi^{-1}}(X_1,\cdots,X_n)-f(\theta)- \langle \bar T_1^{(1)}-\Psi(\theta), (f\circ \Psi^{-1})'(t)\rangle\right\|_{L_p(\mathbb{P}_{\theta})}
    \\
    &
    \lesssim_{s, C} \max_{2\leq k\leq m} \|(f\circ \Psi^{-1})^{(k)}\|_{L_\infty(U)} \left(\sigma_r(\theta)\sqrt{\frac{p}{n}}\Bigl(1\vee \sqrt{\frac{p}{d}}\Bigr)\right) 
    \left(\sigma_r(\theta)\sqrt{\frac{d}{n^{(0)}}}\vee \sigma_r(\theta)\sqrt{\frac{p}{n^{(0)}}}\Bigl(1\vee \sqrt{\frac{p}{d}}\Bigr)
    \right) 
 \\
 &
 +\|(f\circ \Psi^{-1})^{(m)}\|_{{\rm Lip}_\rho(U)} \left(\sigma_r^s(\theta)\Bigl(\frac{d}{n^{(0)}}\Bigr)^{s/2}\vee \sigma_r^s(\theta) \Bigl(\frac{p}{n^{(0)}}\Bigr)^{s/2} \Bigl(1\vee \sqrt{\frac{p}{d}}\Bigr)^s   
\right)
 \\
 &
 + 
\left(\|f\|_{L_\infty}+M + \|(f\circ \Psi^{-1})'(t)\|\left(\sigma_r(\theta)\sqrt{\frac{d}{n}}\vee \sigma_r(\theta)\sqrt{\frac{p}{n}}\Bigl(1\vee \sqrt{\frac{p}{d}}\Bigr)\right) \right)
  \exp\Bigl\{-c_1 \frac{\sqrt{n^{(0)}d}}{p} \frac{\delta}{\sigma_r(\theta)}\Bigr\}.
\end{align*}    
Moreover, if for large enough constant $C'>0,$
 \begin{align*}
r\geq \frac{1}{\sigma_r(\theta)}\sqrt{\frac{d}{n^{(0)}}}\ {\rm and}\ C'C\sigma_r(\theta)\sqrt{\frac{d}{n^{(0)}}}\Bigl(1\vee \frac{\log \frac{n^{(0)}}{\sigma_r^2(\theta)d}}{d}\Bigr)\leq \delta \leq 1,
\end{align*}
then, for $M=\|f\|_{L_{\infty}}+ \|f\|_{\rm Lip},$
\begin{align*}
    &
    \left\|\tilde T_{f\circ \Psi^{-1}}(X_1,\cdots,X_n)-f(\theta)- \langle \bar T_1^{(1)}-\Psi(\theta), (f\circ \Psi^{-1})'(t)\rangle\right\|_{L_2(\mathbb{P}_{\theta})}
    \\
    &
 \lesssim_{s, C} \|f\circ \Psi^{-1}\|_{C^{s}(U)} 
\Bigl[\sigma_{r}^2(\theta) 
\frac{1}{\sqrt{n}}
\sqrt{\frac{d}{n^{(0)}}}
\vee \sigma_{r}^s(\theta)\Bigl(\frac{d}{n^{(0)}}\Bigr)^{s/2}\Bigr].
\end{align*}    
\end{remark}

\begin{remark}
\label{exp_rem_4}
\normalfont
It is easy to see that the quantity $\sigma_r(\theta)$ in the statement of Theorem \ref{estim_exp_fam} and in remarks \ref{exp_rem_1}, \ref{exp_rem_2}
and \ref{exp_rem_3} could be replaced by an arbitrary upper bound.
\end{remark}

Note that Theorem \ref{estim_exp_fam_simple} of Section \ref{exponential_section} easily follows from bound \eqref{bd_exp_exp}.

Finally, we provide the proof of Corollary \ref{cor_asymp_norm}.

\begin{proof}
If $n^{(0)}\asymp \frac{n}{\log n},$ we would have $n_1^{(1)}=(1+o(1))n.$ 
Using the bounds on the rates of convergence in CLT in Wasserstein distance \cite{Rio}, we get 
\begin{align*}
\sup_{\theta\in V}\sup_{\|u\|\leq 1}W_{2, P_{\theta}} \Bigl(\frac{\langle n^{1/2}(\bar T_n -\Psi(\theta)), u\rangle}{\langle\Sigma_{\theta}u,u\rangle^{1/2}}, Z\Bigr)
\lesssim 
\sup_{\theta\in V}\sup_{\|u\|\leq 1}\frac{\|\langle T(X)-{\mathbb E}_{\theta} T(X),u\rangle\|_{L_4({\mathbb P}_{\theta})}^2}{\|\langle T(X)-{\mathbb E}_{\theta} T(X), u\rangle\|_{L_2({\mathbb P}_{\theta})}^2}
\frac{1}{\sqrt{n}} = \frac {\sqrt{\kappa (V)}}{\sqrt{n}}.
\end{align*}
This implies that 
\begin{align*}
\sup_{\theta\in V}\sup_{\|u\|\leq 1}W_{2, P_{\theta}} \Bigl(\langle n^{1/2}(\bar T_n -\Psi(\theta)), u\rangle, \langle\Sigma_{\theta}u,u\rangle^{1/2} Z\Bigr)
\lesssim \sigma_{\max}(W)\frac {\sqrt{\kappa (V)}}{\sqrt{n}}.
\end{align*}
Therefore, it easily follows that
\begin{align*}
\sup_{\|(f\circ \Psi^{-1})'\|_{L_{\infty}(G)}\leq 1}\sup_{\theta\in V}W_{2,{\mathbb P}_{\theta}}\Bigl(\sqrt{n}\Bigl\langle \bar T_1^{(1)}-\Psi(\theta), (f\circ \Psi^{-1})'(t)\Bigr\rangle, \sigma_{f,\Psi}(\Psi(\theta))Z\Bigr)\to 0\ {\rm as}\ n\to\infty.
\end{align*}
Together with the bound of part (ii) of Theorem \ref{estim_exp_fam_simple}, this implies, 
under the assumptions that $d\leq n^{\alpha}$ and $s>\frac{1}{1-\alpha},$ the asymptotic normality \eqref{asymp_normal_exp}
of estimator $\tilde T_{f\circ \Psi^{-1}}(X_1,\cdots,X_n)$
as well as the limit of its $L_2$-error \eqref{mean_square_exp}.

\qed
\end{proof}

\section{Proofs of the upper bounds for high-dimensional models with independent low-dimensional components}
\label{sec:HDLD}

We will provide below the proofs of the results of Section \ref{HD_comps}.
We start with the proof of Theorem \ref{cor_cor_th_1_asssume_AAA}.

\begin{proof}
The proof is based on the following proposition that provides a way to check Assumption \ref{assume_AAA} for estimator $\hat \theta_n$ in terms of similar 
assumptions on its components $\hat \theta_n^{(j)}, j=1,\dots, n.$

\begin{proposition}
\label{bd_ind_comp}
Let $p\geq 2.$ Then 
\begin{enumerate}[label=(\alph*)]
\item For all $P\in {\mathcal P},$
 \begin{align*}
 &
 \sup_{\|u\|\leq 1, u\in E^{\ast}} \Bigl\|\langle \hat \theta_n-\theta(P), u\rangle\Bigr\|_{L_p({\mathbb P}_P)} 
\\
&
 \leq 
 C\frac{p}{\log p} \Bigl(\sqrt{\frac{\max_{1\leq j\leq d} a_{p,j}(P_j)}{n}} + \frac{\max_{1\leq j\leq d}b_{j}(P_j)}{n}\Bigr) + 
\frac{\max_{1\leq j\leq d}b_{j}(P)}{\sqrt{n}} \sqrt{\frac{d}{n}},
 \end{align*}
 where $C>0$ is a numerical constant.
 \item 
 For all $P\in P,$
 \begin{align*}
 \Bigl\|\|\hat \theta_n-\theta(P)\|\Bigr\|_{L_{p}({\mathbb P}_P)} \leq 
  \sqrt{\max_{1\leq j\leq d}d_{p,j}(P_j)} \sqrt{\frac{d}{n}}.
 \end{align*}
 \end{enumerate}
\end{proposition}

\begin{proof}
First note that 
\begin{align}
\label{bou_1_one}
\Bigl\|\langle \hat \theta_n-\theta(P), u\rangle\Bigr\|_{L_p({\mathbb P}_P)}\leq 
 \Bigl\|\langle \hat \theta_n-{\mathbb E}_P \hat \theta_n, u\rangle\Bigr\|_{L_p({\mathbb P}_P)}  
 +
\Bigl\|\langle {\mathbb E}_P \hat \theta_n - \theta(P), u\rangle\Bigr\|_{L_p({\mathbb P}_P)}. 
\end{align}
We have 
\begin{align*}
\Bigl\|\langle \hat \theta_n-{\mathbb E}_P \hat \theta_n, u\rangle\Bigr\|_{L_p({\mathbb P}_P)}  =
\Bigl\|\sum_{j=1}^d \langle \hat \theta_n^{(j)}-{\mathbb E}_{P_j} \hat \theta_n^{(j)}, u^{(j)}\rangle\Bigr\|_{L_p({\mathbb P}_P)}.  
\end{align*}
Denote 
$
\zeta_j :=  \langle \hat \theta_n^{(j)}-{\mathbb E}_{P_j} \hat \theta_n^{(j)}, u^{(j)}\rangle, j=1,\dots, d.
$
Since $\zeta_j, j=1,\dots, d$ are independent mean zero r.v., we can use the bound of Th. 1.5.11 in \cite{Pena_Gine} to get that, for all $p\geq 2,$ 
\begin{align*} 
\Bigl\|\langle \hat \theta_n-{\mathbb E}_P \hat \theta_n, u\rangle\Bigr\|_{L_p({\mathbb P}_P)}  
&= \Bigl\|\sum_{j=1}^d \zeta_j\Bigr\|_{L_p}
\lesssim \frac{p}{\log p} \biggl(\biggl(\sum_{j=1}^d {\mathbb E} \zeta_j^2\biggr)^{1/2}\bigvee 
\biggl(\sum_{j=1}^d {\mathbb E} |\zeta_j|^p\biggr)^{1/p}\biggr).
\end{align*}
Note that, for $p\geq 2,$ 
\begin{align*}
\|\zeta_j\|_{L_2} \leq \|\zeta_j\|_{L_p}\leq \gamma_{n,p}(P)\|u^{(j)}\|, j=1,\dots, d,
\end{align*}
where 
\begin{align*}
\gamma_{n, p} (P) := \max_{1\leq j\leq d} \sup_{\|u_j\|\leq 1, u_j\in E_j^{\ast}}\Bigl\| \langle \hat \theta_n^{(j)}- {\mathbb E}_{P_j}\hat \theta_n^{(j)}, u_j\rangle\Bigr\|_{L_p({\mathbb P}_P)}, P\in {\mathcal P}.
\end{align*}
Therefore, 
\begin{align}
\label{bou_2_two}
\nonumber 
\Bigl\|\langle \hat \theta_n-{\mathbb E}_P \hat \theta_n, u\rangle\Bigr\|_{L_p({\mathbb P}_P)}  
&
\lesssim \frac{p}{\log p}  \gamma_{n,p}(P)\biggl(\biggl(\sum_{j=1}^d \|u^{(j)}\|^2\biggr)^{1/2}\bigvee \biggl(\sum_{j=1}^d \|u^{(j)}\|^p\biggr)^{1/p}\biggr)
\\
&
\nonumber
\lesssim \frac{p}{\log p}  \gamma_{n,p}(P)\biggl(\sum_{j=1}^d \|u^{(j)}\|^2\biggr)^{1/2} 
\\
&
= \frac{p}{\log p}  \gamma_{n,p}(P)\|u\|.
\end{align}
On the other hand,
\begin{align}
\label{bou_3_three}
&
\nonumber
\Bigl\|\langle {\mathbb E}_P \hat \theta_n - \theta(P), u\rangle\Bigr\|_{L_p({\mathbb P}_P)}
= \Bigl| \sum_{j=1}^d \langle {\mathbb E}_P \hat \theta_n^{(j)}-\theta^{(j)}(P_j), u^{(j)}\rangle\Bigr|
\\
&
\nonumber
\leq 
\biggl(\sum_{j=1}^d \Bigl\|{\mathbb E}_{P_j} \hat \theta_{n}^{(j)} - \theta^{(j)}(P_j)\Bigr\|^2\biggr)^{1/2}
\biggl(\sum_{j=1}^d \|u^{(j)}\|^2\biggr)^{1/2}
\\
&
= \beta_n(P) \|u\|,
\end{align}
where 
\begin{align*}
\beta_{n} (P):= \|{\mathbb E}_P \hat \theta_n-\theta(P)\|=
\biggl(\sum_{j=1}^d \Bigl\|{\mathbb E}_{P_j} \hat \theta_{n}^{(j)} - \theta^{(j)}(P_j)\Bigr\|^2\biggr)^{1/2}, P\in {\mathcal P}.
\end{align*}
It immediately follows from bounds \eqref{bou_1_one}, \eqref{bou_2_two}
and \eqref{bou_3_three} that 
\begin{align*}
\sup_{\|u\|\leq 1, u\in E^{\ast}} \Bigl\|\langle \hat \theta_n-\theta(P), u\rangle\Bigr\|_{L_p({\mathbb P}_P)} 
\leq C\frac{p}{\log p} \gamma_{n,p}(P) + \beta_{n}(P)
\end{align*}
with some numerical constant $C>0.$

Also note that, for $p\geq 2,$ 
\begin{align*}
&
\Bigl\|\|\hat \theta_n-\theta(P)\|\Bigr\|_{L_{p}({\mathbb P}_P)} = 
\Bigl\|\Bigl(\sum_{j=1}^d\|\hat \theta_n^{(j)}-\theta^{(j)}(P_j)\|^{2}\Bigr)^{1/2}\Bigr\|_{L_{p}({\mathbb P}_P)}
\\
&
=  \Bigl\|\sum_{j=1}^d\|\hat \theta_n^{(j)}-\theta^{(j)}(P_j)\|^{2}\Bigr\|_{L_{p/2}({\mathbb P}_P)}^{1/2}
\leq \Bigl(\sum_{j=1}^d \Bigl\|\|\hat \theta_n^{(j)}-\theta^{(j)}(P_j)\|^{2}\Bigr\|_{L_{p/2}({\mathbb P}_P)}\Bigr)^{1/2}
\\
&
= \Bigl(\sum_{j=1}^d \Bigl\|\|\hat \theta_n^{(j)}-\theta^{(j)}(P_j)\|\Bigr\|_{L_{p}({\mathbb P}_P)}^2\Bigr)^{1/2}
\leq \tau_{n,p}(P)\sqrt{d},
\end{align*}
where 
\begin{align*}
\tau_{n, p} (P) := \max_{1\leq j\leq d} \Bigl\| \|\hat \theta_n^{(j)}- \theta_j(P_j)\|\Bigr\|_{L_p({\mathbb P}_P)}, P\in {\mathcal P}.
\end{align*}

To complete the proof, it remains to further bound the quantities $\beta_n(P), \gamma_{n,p}(P)$ and $\tau_{n,p}(P)$
under the conditions of the proposition.
We easily get that 
\begin{align*}
\beta_{n} (P)= \|{\mathbb E}_P \hat \theta_n-\theta(P)\| \leq \frac{\max_{1\leq j\leq d}b_j(P_j)}{\sqrt{n}} \sqrt{\frac{d}{n}}.
\end{align*}
Also, 
\begin{align*}
\gamma_{n,p}(P) \leq \sqrt{\frac{\max_{1\leq j\leq d} a_{p,j}(P_j)}{n}} + \frac{\max_{1\leq j\leq d}b_j(P_j)}{n}.
\end{align*}
and 
\begin{align*}
\tau_{n,p}(P) \leq \sqrt{\frac{\max_{1\leq j\leq d}d_{p,j}(P_j)}{n}}.
\end{align*}
This implies the bounds of the proposition.

\qed
\end{proof}

Note that, if $d\lesssim n$ and  
\begin{align*}
\max_{1\leq j\leq d}b_j(P_j)\lesssim 1,\ \max_{1\leq j\leq d} a_{p,j}(P_j)\lesssim 1,\  \max_{1\leq j\leq d}d_{ps,j}(P_j)\lesssim 1,
\end{align*}
then, for estimator $\hat \theta_n,$ we have $a_p(P)\lesssim_p 1$ and $d_{ps}(P)\lesssim_{ps} d.$
The claims of the theorem now follow from the bounds of Theorem \ref{th_1_asssume_AAA}.

\qed
\end{proof}

Next we provide a proof of Corollary \ref{cor_cor_th_1_asssume_regular}.

\begin{proof}
First note that 
\begin{align*} 
\Bigl\|{\mathbb E}_{\theta_j} \hat \theta_{n}^{(j)} - \theta^{(j)}\Bigr\|\leq 
\frac{W_{1,{\mathbb P}_{\theta}} \Bigl(\sqrt{n}(\hat \theta_n^{(j)}-\theta^{(j)}), \xi^{(j)}(\theta^{(j)})\Bigr)}{\sqrt{n}} \leq \frac{C_{p,j}(\theta^{(j)})}{n}.
\end{align*}
Moreover,
\begin{align*}
&
\sup_{\|u^{(j)}\|\leq 1} \Bigl\|\langle \hat \theta_n^{(j)}-\theta^{(j)}, u^{(j)}\rangle\Bigr\|_{L_p({\mathbb P}_{\theta})} 
\\
&
\leq 
\frac{\sup_{\|u^{(j)}\|\leq 1}\Bigl\|\langle \xi^{(j)}(\theta^{(j)}), u^{(j)}\rangle\Bigr\|_{L_p({\mathbb P}_{P})}}{\sqrt{n}} +
\frac{W_{p,{\mathbb P}_{\theta}} 
\Bigl(\sqrt{n}(\hat \theta_n^{(j)}-\theta^{(j)}), \xi^{(j)}(\theta^{(j)})\Bigr)}{\sqrt{n}}
\\
&
\leq \frac{\sup_{\|u^{(j)}\|\leq 1}\|\langle  Z,I_j(\theta^{(j)})^{-1/2}u^{(j)}\rangle\|_{L_p}}{\sqrt{n}}
+  \frac{C_{p,j}(\theta^{(j)})}{n}
\\
&
\lesssim_p \frac{\|I_j(\theta^{(j)})^{-1/2}\|}{\sqrt{n}} + \frac{C_{p,j}(\theta^{(j)})}{n}
\end{align*}
and, similarly, 
\begin{align*}
\Bigl\|\|\hat \theta_n^{(j)}-\theta^{(j)}\|\Bigr\|_{L_{p}({\mathbb P}_{\theta})} \lesssim_p
\|I_j(\theta^{(j)})^{-1/2}\|\sqrt{\frac{l_j}{n}} + \frac{C_{p,j}(\theta^{(j)})}{n}.
\end{align*}
The above bounds allow one to check condition \eqref{cond_a_b_d} of Theorem \ref{cor_cor_th_1_asssume_AAA}, which yields the claims of the corollary. 

\qed
\end{proof}

Finally, we provide a proof of Proposition \ref{asym_eff_many_comp}.

\begin{proof}
We need to check that Assumption \ref{assume_normal_approx_A} holds for estimator 
$\hat \theta_n = (\hat \theta_n^{(1)}, \dots, \hat \theta_n^{(d)}).$ To this end, we obtain a bound on the accuracy of approximation 
of $\sqrt{n}(\hat \theta_n-\theta)$ by $\xi (\theta):= (\xi_1(\theta^{(1)}),\dots, \xi_d(\theta^{(d)}))$
in the $W_p$-distance.

\begin{lemma}
\label{prop_normal_appr_many_comp}
For all $p\geq 2,$ the following bound holds:
\begin{align*}
\sup_{\|u\|\leq 1} W_{p,{\mathbb P}_{\theta}}\Big(\sqrt{n}\langle \hat \theta_n-\theta, u\rangle, \langle \xi(\theta), u\rangle\Bigr)
\lesssim \frac{p}{\log p}\max_{1\leq j\leq d}C_{p,j}(\theta^{(j)})\sqrt{\frac{d}{n}}.
\end{align*}
Thus, under the assumption $\max_{1\leq j\leq d}\sup_{\theta^{(j)}\in \Theta_j}C_{p,j}(\theta^{(j)})\lesssim 1$ and $d=d_n=o(n)$ as $n\to \infty,$
we have 
\begin{align*}
\sup_{\|u\|\leq 1} W_{p,{\mathbb P}_{\theta}}\Big(\sqrt{n}\langle \hat \theta_n-\theta, u\rangle, \langle \xi(\theta), u\rangle\Bigr)\to 0.
\end{align*}
\end{lemma}

\begin{proof}
Let $(\eta^{(1)}, \zeta^{(1)}), \dots, (\eta^{(d)}, \zeta^{(d)})$ be independent r.v. such that $\eta^{(j)}\overset{d}{=} \sqrt{n}(\hat \theta_n^{(j)}-\theta^{(j)})$
and $\zeta^{(j)} \overset{d}{=} \xi^{(j)}(\theta^{(j)})$ for all $j=1,\dots, d.$ Denote $\eta:=(\eta^{(1)},\dots, \eta^{(d)})$ and $\zeta:=(\zeta^{(1)},\dots, \zeta^{(d)}).$
We have 
\begin{align*}
&
W_{p,{\mathbb P}_{\theta}}\Big(\sqrt{n}\langle \hat \theta_n-\theta, u\rangle, \langle \xi(\theta), u\rangle\Bigr)\leq 
\|\langle \eta-\zeta, u\rangle\|_{L_p} = \Bigl\|\sum_{j=1}^d \langle \eta^{(j)}-\zeta^{(j)}, u^{(j)}\rangle\Bigr\|_{L_p}
\\
&
\leq \Bigl|\sum_{j=1}^d \langle {\mathbb E}(\eta^{(j)}-\zeta^{(j)}), u^{(j)}\rangle\Bigr| + 
\Bigl\|\sum_{j=1}^d (\langle \eta^{(j)}-\zeta^{(j)}, u^{(j)}\rangle- {\mathbb E} \langle \eta^{(j)}-\zeta^{(j)}, u^{(j)}\rangle)\Bigr\|_{L_p}.
\end{align*}
The first term in the right hand side could be bounded as follows:
\begin{align*}
&
\Bigl|\sum_{j=1}^d \langle {\mathbb E}(\eta^{(j)}-\zeta^{(j)}), u^{(j)}\rangle\Bigr| \leq \Bigl(\sum_{j=1}^d \|{\mathbb E}\eta^{(j)}-{\mathbb E}\zeta^{(j)}\|^2\Bigr)^{1/2}
\Bigl(\sum_{j=1}^d \|u^{(j)}\|^2\Bigr)^{1/2}
\\
&
\leq \|u\|\Bigl(\sum_{j=1}^d W_{p, {\mathbb P}_{\theta}}^2\Bigl(\sqrt{n}(\hat \theta_n^{(j)}-\theta^{(j)}), \xi^{(j)}(\theta^{(j)})\Bigr)\Bigr)^{1/2} 
\leq \|u\|\max_{1\leq j\leq d}C_{p,j}(\theta^{(j)})\sqrt{\frac{d}{n}}.
\end{align*}
For the second term, we argue as in the proof of Proposition \ref{bd_ind_comp} to get, for $p\geq 2,$ the bound 
\begin{align*}
&
\Bigl\|\sum_{j=1}^d (\langle \eta^{(j)}-\zeta^{(j)}, u^{(j)}\rangle- {\mathbb E} \langle \eta^{(j)}-\zeta^{(j)}, u^{(j)}\rangle)\Bigr\|_{L_p}
\\
&
\lesssim \frac{p}{\log p} \Bigl(\Bigl(\sum_{j=1}^d {\mathbb E}\langle \eta^{(j)}-\zeta^{(j)}, u^{(j)}\rangle^2\Bigr)^{1/2}\vee 
\Bigl(\sum_{j=1}^d {\mathbb E}|\langle \eta^{(j)}-\zeta^{(j)}, u^{(j)}\rangle|^p\Bigr)^{1/p}\Bigr)
\\
&
\lesssim \frac{p}{\log p} \Bigl(\Bigl(\sum_{j=1}^d{\mathbb E}\|\eta^{(j)}-\zeta^{(j)}\|^2 \|u^{(j)}\|^2\Bigr)^{1/2} 
\vee \Bigl(\sum_{j=1}^d{\mathbb E}\|\eta^{(j)}-\zeta^{(j)}\|^p \|u^{(j)}\|^p \Bigr)^{1/p}\Bigr)
\\
&
\lesssim \frac{p}{\log p} \Bigl(\max_{1\leq j\leq d} {\mathbb E}^{1/2} \|\eta^{(j)}-\zeta^{(j)}\|^2 \Bigl(\sum_{j=1}^d \|u^{(j)}\|^2\Bigr)^{1/2} \vee 
\max_{1\leq j\leq d} {\mathbb E}^{1/p} \|\eta^{(j)}-\zeta^{(j)}\|^p \Bigl(\sum_{j=1}^d \|u^{(j)}\|^p\Bigr)^{1/p}\Bigr)
\\
&
\lesssim  \frac{p}{\log p} \|u\| \max_{1\leq j\leq d} {\mathbb E}^{1/p} \|\eta^{(j)}-\zeta^{(j)}\|^p.
\end{align*}
Thus, we get 
\begin{align*}
\sup_{\|u\|\leq 1} W_{p,{\mathbb P}_{\theta}}\Big(\sqrt{n}\langle \hat \theta_n-\theta, u\rangle, \langle \xi(\theta), u\rangle\Bigr)
\lesssim \max_{1\leq j\leq d}C_{p,j}(\theta^{(j)})\sqrt{\frac{d}{n}} + \frac{p}{\log p}\max_{1\leq j\leq d} {\mathbb E}^{1/p} \|\eta^{(j)}-\zeta^{(j)}\|^p.
\end{align*}
Minimizing the right hand side over all $\eta^{(j)}\overset{d}{=} \sqrt{n}(\hat \theta_n^{(j)}-\theta^{(j)})$
and $\zeta^{(j)} \overset{d}{=} \xi^{(j)}(\theta^{(j)})$ yields the bound 
\begin{align*}
&
\sup_{\|u\|\leq 1} W_{p,{\mathbb P}_{\theta}}\Big(\sqrt{n}\langle \hat \theta_n-\theta, u\rangle, \langle \xi(\theta), u\rangle\Bigr)
\\
&
\lesssim 
\max_{1\leq j\leq d}C_{p,j}(\theta^{(j)})\sqrt{\frac{d}{n}} + \frac{p}{\log p}\max_{1\leq j\leq d} 
W_{p,{\mathbb P}_{\theta}}\Bigl(\sqrt{n}(\hat \theta_n^{(j)}-\theta^{(j)}), \xi^{(j)}(\theta^{(j)})\Bigr) 
\\
&
\lesssim 
\max_{1\leq j\leq d}C_{p,j}(\theta^{(j)})\sqrt{\frac{d}{n}} + \frac{p}{\log p}\frac{\max_{1\leq j\leq d} C_{p,j}(\theta^{(j)})}{\sqrt{n}} 
\lesssim \frac{p}{\log p}\max_{1\leq j\leq d}C_{p,j}(\theta^{(j)})\sqrt{\frac{d}{n}}.
\end{align*}

\qed
\end{proof}

Applying this lemma for $u=f'(\theta)$ and observing that ${\mathbb E}\langle \xi(\theta), f'(\theta)\rangle^2= \sigma_f^2(\theta),$
it is easy to deduce the claim of Proposition \ref{asym_eff_many_comp} from Corollary \ref{cor_cor_th_1_asssume_regular}(ii).

\qed
\end{proof}

\section{Proofs of upper bounds for functionals of covariance}
\label{func_cov_proofs}

We provide a proof of Theorem \ref{th_1_Bern_AA_covariance} of Section \ref{func_cov_op}.

\begin{proof} We will use the following simple fact (probably, well known, but we will sketch its proof for completeness). 

\begin{proposition}
\label{subgauss_series}
Let $X$ be a centered subgaussian and pre-gaussian r.v. in $E$ with covariance operator $\Sigma$ and let $\{x_n\}$ be 
the sequence of vectors from representation of $Y\sim N(0,\Sigma)$ in Proposition \ref{Gauss_repres}. Denote by ${\mathbb H}_X$
the closure of the set $\{\langle X,u\rangle: u\in E^{\ast}\}$ in the space $L_2({\mathbb P})$ of real valued random variables. Then, there exists 
an orthonormal sequence $\{\xi_n\}\subset {\mathbb H}_X\subset L_2({\mathbb P})$ of r.v. such that, for all $u\in E^{\ast},$ 
\begin{align*}
\langle X,u\rangle = \sum_{n=1}^{\infty} \xi_n \langle x_n,u\rangle\ {\rm in}\ L_2({\mathbb P})
\end{align*}
and, moreover, 
\begin{align*}
{\mathbb E} \Bigl\|X-\sum_{n=1}^N \xi_n x_n\Bigr\|\to 0\ {\rm as}\ N\to\infty.
\end{align*}
\end{proposition}

\begin{proof}
It is easy to check that ${\mathbb H}_X$ is a linear subspace and all r.v. in ${\mathbb H}_X$ are centered and subgaussian: 
for all $\eta\in {\mathbb H}_X,$ ${\mathbb E}\eta=0$
and 
\begin{align*}
\|\eta\|_{\psi_2}\leq C\|\eta\|_{L_2}.
\end{align*}
For $u\in E^{\ast},$ define $\bar u:= (\langle x_n,u\rangle: n\geq 1)\in \ell_2.$ Since vectors $\{x_n: n\geq 1\}$ are linearly independent,
it is easy to check that the closure of the set $\{\bar u: u\in E^{\ast}\}$ coincides with $\ell_2.$ Recall that 
\begin{align*}
{\mathbb E}\langle X,u\rangle \langle X,v\rangle= \langle \Sigma u,v\rangle = \sum_{n\geq 1}\langle x_n,u\rangle \langle x_n,v\rangle = 
\langle \bar u, \bar v\rangle_{\ell_2}, u,v\in E^{\ast}.
\end{align*}
Thus, the mapping $\bar u\mapsto \langle X,u \rangle$ is an isometry between linear spaces $\{\bar u: u\in E^{\ast}\}$ 
and $\{\langle X,u\rangle: u\in E^{\ast}\},$ and it could be extended by continuity to an isometry $J$ between their closures $\ell_2$ and ${\mathbb H}_X.$
Let $\{e_n: n\geq 1\}$ be the canonical basis of $\ell_2$ and denote $\xi_n:= J e_n, n\geq 1.$ Then 
$\{\xi_n: n\geq 1\}$ is an orthonormal basis of subspace ${\mathbb H}_X$ of $L_2({\mathbb P}),$ and, moreover, for all $u\in E^{\ast},$
\begin{align*}
\langle X,u\rangle= J \bar u = \sum_{n\geq 1} \xi_n \langle x_n,u \rangle 
\end{align*}
with convergence of the series in $L_2({\mathbb P}).$

To conclude the proof, denote $X_{>N}:= X-\sum_{n=1}^N \xi_n x_n$ and $Y_{>N}:= Y-\sum_{n=1}^N g_n x_n.$
Note that $\langle X_{>N}, u\rangle, u\in E^{\ast}$ is a centered subgaussian process whereas $\langle Y_{>N}, u\rangle, u\in E^{\ast}$
is a centered Gaussian process. Moreover, we have 
\begin{align*}
&
\|\langle X_{>N}, u\rangle - \langle X_{>N}, v\rangle\|_{\psi_2} \leq C \|\langle X_{>N}, u\rangle - \langle X_{>N}, v\rangle\|_{L_2}
\\
&
=C \Bigl(\sum_{n>N} \langle x_n, u-v\rangle^2\Bigr)^{1/2}
=C \|\langle Y_{>N}, u\rangle - \langle Y_{>N}, v\rangle\|_{L_2},\ u, v\in E^{\ast}.
\end{align*}
Using Talagrand's upper generic chaining bound on subgaussian process$\langle X_{>N}, u\rangle, u\in E^{\ast}, \|u\|\leq 1$ along with his lower generic chaining bound on Gaussian process  $\langle Y_{>N}, u\rangle, u\in E^{\ast}, \|u\|\leq 1$ yields
\begin{align*}
{\mathbb E}\Bigl\|X-\sum_{n=1}^N \xi_n x_n\Bigr\|= {\mathbb E}\|X_{>N}\| = {\mathbb E}\sup_{\|u\|\leq 1}\langle X_{>N}, u\rangle
\lesssim C  {\mathbb E}\sup_{\|u\|\leq 1}\langle Y_{>N}, u\rangle = C{\mathbb E}\Bigl\|X-\sum_{n=1}^N g_n x_n\Bigr\|.
\end{align*}
It easily follows from the claims of Proposition \ref{Gauss_repres} and Gaussian concentration inequality that 
\begin{align*}
{\mathbb E}\Bigl\|X-\sum_{n=1}^N g_n x_n\Bigr\| \to 0\ {\rm as}\ n\to\infty,
\end{align*}
implying the last statement. 

\qed
\end{proof}

To apply Theorem \ref{th_1_Bern_AA}, we need to obtain bounds on the $L_p$-norms 
$
\sup_{\|U\|\leq 1} \|\langle \hat \Sigma_n-\Sigma, U\rangle\|_{L_p({\mathbb P}_P)}
$
and 
$\| \|\hat \Sigma_n-\Sigma\|\|_{L_p({\mathbb P}_P)},
$
in the first of these quantities, the supremum being over the unit ball of the dual space of $L(E^{\ast}, E).$
The following simple proposition holds.

\begin{proposition} 
\label{Prop_Bern_cov}
Under the above assumptions and notations, for all $p\geq 1,$
\begin{align}
\label{cov_odin}
\sup_{\|U\|\leq 1} \Bigl\|\langle \hat \Sigma_n-\Sigma, U\rangle\Bigr\|_{L_p({\mathbb P}_P)}
\lesssim C^2\|\Sigma\| \Bigl(\sqrt{\frac{p}{n}}\vee \frac{p}{n}\Bigr)
\end{align}
and 
\begin{align}
\label{cov_dva}
\Bigl\| \|\hat \Sigma_n-\Sigma\|\Bigr\|_{L_p({\mathbb P}_P)}
\lesssim C^2\|\Sigma\| \Bigl(\sqrt{\frac{{\bf r}(\Sigma)}{n}}\vee \frac{{\bf r}(\Sigma)}{n}\vee \sqrt{\frac{p}{n}}\vee \frac{p}{n}\Bigr).
\end{align}
\end{proposition}

\begin{proof}
Since 
\begin{align*}
\langle \hat \Sigma_n-\Sigma, U\rangle = n^{-1}\sum_{j=1}^n \langle X_j\otimes X_j, U\rangle - {\mathbb E}\langle X\otimes X, U\rangle, 
\end{align*}
it is enough to show that $\|\langle X\otimes X, U\rangle\|_{\psi_1}\lesssim C^2\|\Sigma\|\|U\|.$
Then, bound \eqref{cov_odin} would follow from Bernstein inequality for subexponential r.v. 
Moreover, it would be enough to prove the bound on $\|\langle X\otimes X, U\rangle\|_{\psi_1}$ only in the case 
of r.v. $X$ taking values in a finite-dimensional subspace $L\subset E.$ Indeed, in the general case, one can use approximation 
of $X$ by $X^{(N)}:=\sum_{n=1}^N \xi_n x_n.$ By Proposition \ref{subgauss_series}, ${\mathbb E}\|X-X^{(N)}\|\to 0$ as $N\to \infty.$
Thus, there exists a subsequence $N_k, k\geq 1$ such that $\|X-X^{(N_k)}\|\to 0$ as $k\to\infty$ a.s., which also 
implies $X^{(N_k)}\otimes X^{(N_k)}\to X\otimes X$ as $k\to \infty$ a.s. in the operator norm. If, for some numerical constant $D>0$
and for all $N\geq 1,$ we have  
\begin{align*}
\|\langle X^{(N)}\otimes X^{(N)}, U\rangle\|_{\psi_1}\leq D C^2\|\Sigma^{(N)}\|\|U\|,
\end{align*}
then we also have, for some numerical constant $D'>0$ and for all $p\geq 1, N\geq 1,$ 
\begin{align*}
\|\langle X^{(N)}\otimes X^{(N)}, U\rangle\|_{L_p}\leq D' p C^2\|\Sigma^{(N)}\|\|U\|,
\end{align*}
where $\Sigma^{(N)}= \sum_{k=1}^N x_k\otimes x_k$ is the covariance operator of $X^{(N)}.$ 
Since $\|\Sigma^{(N)}\|\leq \|\Sigma\|,$ we can write 
\begin{align*}
\|\langle X^{(N)}\otimes X^{(N)}, U\rangle\|_{L_p}\leq D' p C^2\|\Sigma\|\|U\|,
\end{align*}
and, setting $N=N_k,$ it is easy to justify passing to the limit as $k\to \infty$ to get 
\begin{align*}
\|\langle X\otimes X, U\rangle\|_{L_p}\leq D' p C^2\|\Sigma\|\|U\|.
\end{align*}
This, clearly, implies the bound $\|\langle X\otimes X, U\rangle\|_{\psi_1}\lesssim C^2\|\Sigma\|\|U\|.$

Thus, we can now consider the case when $X=\sum_{j=1}^N \xi_j x_j$ with $x_1,\dots, x_N\in E$ being linearly independent vectors and 
$\xi_1,\dots, \xi_N\in {\mathbb H}_X$ being centered orthonormal r.v.. In this case, we can write   
\begin{align*}
\langle X\otimes X, U\rangle = \sum_{i,j=1}^{N} \xi_i \xi_j \langle x_i\otimes x_j, U\rangle = \sum_{i,j=1}^{N} a_{ij} \xi_i \xi_j,
\end{align*}
where 
$
a_{ij} := \frac{1}{2}\langle x_i\otimes x_j+x_j\otimes x_i, U\rangle.
$
We will now prove the following bound on the nuclear norm of symmetric matrix $A:= (a_{ij})_{i,j=1}^{N}:$ 
$\|A\|_1 \leq \|U\|\|\Sigma\|.$ Indeed, by the duality between the operator norm and the nuclear norm,
$
\|A\|_1 = \sup\{\langle A,B\rangle: \|B\|\leq 1\}.
$
We have 
\begin{align*}
\langle A,B\rangle &=\sum_{i,j=1}^{N} a_{ij} b_{ij} = \frac{1}{2} \Bigl\langle \sum_{i,j=1}^{N} b_{ij}(x_i\otimes x_j+x_j\otimes x_i), U\Bigr\rangle
\\
&
=\Bigl\langle \sum_{i,j=1}^{N} b_{ij} x_i\otimes x_j, U\Bigr\rangle \leq \|U\| \Bigl\|\sum_{i,j=1}^{N} b_{ij} x_i\otimes x_j\Bigr\|
\\
&
\leq \|U\| \sup_{\|u\|, \|v\|\leq 1} \sum_{i,j} b_{ij} \langle x_i,u\rangle \langle x_i, v\rangle 
\leq \|U\|\|B\| \sup_{\|u\|\leq 1} \sum_{j=1}^{N} \langle x_j, u\rangle^2 
\\
&
= \|U\|\|B\| \sup_{\|u\|\leq 1}\langle \Sigma u, u\rangle \leq \|U\|\|B\|\|\Sigma\|,
\end{align*}
which implies that $\|A\|_1 \leq \|U\|\|\Sigma\|.$  

The condition that $X=\sum_{j=1}^n \xi_j x_j$ is a centered subgaussian random vector implies 
that $\xi:=(\xi_1,\dots, \xi_N)$ is a centered subgaussian r.v. in ${\mathbb R}^N,$ for which we have the bound 
\begin{align*}
\|\langle \xi, t\rangle\|_{\psi_2}\leq C \|t\|_{\ell_2}, t\in {\mathbb R}^N. 
\end{align*}
Let $\lambda_1\geq \dots \geq \lambda_N$
be the eigenvalues of symmetric matrix $A$ and $\phi_1,\dots, \phi_d$ be the corresponding orthonormal eigenvectors. 
Then 
$
 \sum_{i,j=1}^{N} a_{ij} \xi_i \xi_j = \sum_{j=1}^N \lambda_j \langle \xi, \phi_j \rangle^2.
$
Therefore, we get 
\begin{align*}
\|\langle X\otimes X, U\rangle\|_{\psi_1} &= \Bigl\| \sum_{j=1}^N \lambda_j \langle \xi, \phi_j \rangle^2\Bigr\|_{\psi_1} 
\leq \sum_{j=1}^N  |\lambda_j| \|\langle \xi, \phi_j \rangle^2\|_{\psi_1}
\\
&
= \sum_{j=1}^N  |\lambda_j| \|\langle \xi, \phi_j \rangle\|_{\psi_2}^2 \leq C^2 \|A\|_1 \leq C^2\|\Sigma\|\|U\|.
\end{align*}
As a result, we proved bound \eqref{cov_odin}.

Since r.v. $X$ is both subgaussian and pregaussian, the proof of \eqref{cov_dva} follows from the Bernstein type bounds on 
$\|\hat \Sigma_n-\Sigma\|$ of Theorem 9 in Koltchinskii and Lounici (2017) (which itself relies on generic chaining bounds obtained 
earlier in the papers by Mendelson, Bednorz and Dirksen): for all $t\geq 1$ with probability at least $1-e^{-t}$
\begin{align*}
\|\hat \Sigma_n-\Sigma\|\lesssim C^2\|\Sigma\| \Bigl(\sqrt{\frac{{\bf r}(\Sigma)}{n}}\vee \frac{{\bf r}(\Sigma)}{n}\vee \sqrt{\frac{t}{n}}\vee \frac{t}{n}\Bigr).
\end{align*}

\qed
\end{proof}

It follows from Proposition \ref{Prop_Bern_cov} that, for $X\sim P$ with covariance $\Sigma=\Sigma_P$ satisfying the above conditions,
Bernstein type bounds on $\sup_{\|U\|\leq 1}\|\langle \hat \Sigma_n-\Sigma, U\rangle\|_{L_p({\mathbb P}_P)}$ and on 
$\|\|\hat \Sigma_n-\Sigma\|\|_{L_p({\mathbb P}_P)}$ of Theorem \ref{th_1_Bern_AA} hold 
with $\sigma(P)=U(P)=\|\Sigma\|$ and $d_1(P)=\|\Sigma\|^2 {\bf r}(\Sigma), d_2(P)=\|\Sigma\|{\bf r}(\Sigma).$ Using Theorem \ref{th_1_Bern_AA}, we 
can conclude the proof.

\qed
\end{proof}

\section{Lower bounds}
\label{sec:lower_bounds}

In this section, we provide the proofs of minimax lower bounds of sections \ref{Main_results}, \ref{HD_comps} and \ref{exponential_section}. 
We start with the proof of local minimax lower bound of Theorem \ref{vanTrees_bd}.

\begin{proof}
Note that it is enough to prove the claim under the assumptions that $\|I(\theta_0)^{-1}\| \omega_I(\theta_0, \delta)\leq 1/2$ and 
$\frac{\omega_{f'}(\theta_0,\delta)}{\|f'(\theta_0)\|}\leq 1.$ Otherwise, since $D\geq 2$ and $\|I(\theta_0)\|\|I(\theta_0)^{-1}\|\geq 1,$
the right hand side of the inequality would be non-positive and the bound would trivially hold. 

Let $h:= I(\theta_0)^{-1} f'(\theta_0)$
and define $\theta_t := \theta_0+ \frac{t}{\sqrt{n}}h, t\in [-c,c].$ We will choose $c>0$ such that 
\begin{align}
\label{cond_c_delta}
\frac{c}{\sqrt{n}}\|I(\theta_0)^{-1}\|\|f'(\theta_0)\|=\delta,
\end{align}
which implies $\theta_t \in B(\theta_0,\delta), t\in [-c,c].$
Consider the following one-dimensional model: $X_1,\dots, X_n$ i.i.d. $\sim P_{\theta_t}, t\in [-c,c].$ It is easy to check that its Fisher 
information is 
\begin{align*}
J(t)= n \Bigl\langle \frac{h}{\sqrt{n}}, I(\theta_t) \frac{h}{\sqrt{n}}\Bigr\rangle = \langle h, I(\theta_t) h\rangle.
\end{align*}
We will now consider the problem of estimation of function $\varphi(t)=f(\theta_t), t\in [-c,c]$ based on $X_1,\dots, X_n\sim P_{\theta_t}$
and use van Trees inequality (see \cite{Tsybakov}, Theorem 2.13) to obtain a minimax lower bound for this estimation problem. 
Let $\pi$ be a smooth probability density on $[-1,1]$ with $\pi(1)=\pi(-1)=0$ and such that $J_{\pi}:= \int_{-1}^1 \frac{(\pi'(s))^2}{\pi(s)}ds<\infty.$
Denote $\pi_c(t):= \frac{1}{c}\pi\Bigl(\frac{t}{c}\Bigr), t\in [-c,c].$ Then, it follows from van Trees inequality that 
\begin{align}
\label{van_trees_bound}
\inf_{T_n} \sup_{t\in [-c,c]} {\mathbb E}_t (T_n(X_1,\dots, X_n)-\varphi (t))^2\geq  \frac{(\int_{-c}^c \varphi'(t) \pi_c(t) dt)^2}{\int_{-c}^c J(t) \pi_c(t) dt + 
\frac{J_{\pi}}{c^2}}. 
\end{align}
Note that 
$
\varphi'(t)= \frac{\langle h, f'(\theta_t)\rangle}{\sqrt{n}}
$
and 
\begin{align*}
\varphi'(0) = \frac{\langle h, f'(\theta_0)\rangle}{\sqrt{n}}= \frac{\langle I(\theta_0)^{-1} f'(\theta_0), f'(\theta_0)\rangle}{\sqrt{n}}= \frac{\sigma_f^2(\theta_0)}{\sqrt{n}}.
\end{align*}
Next, we have 
\begin{align*}
&
\sqrt{n}|\varphi'(t)-\varphi'(0)| = |\langle h, f'(\theta_t)-f'(\theta_0)\rangle| \leq \|h\|\|f'(\theta_t)-f'(\theta_0)\|
\leq \|I(\theta_0)^{-1}\| \|f'(\theta_0)\| \omega_{f'}(\theta_0, \delta).
\end{align*}
Therefore, 
\begin{align*}
&
n\Bigl(\int_{-c}^c \varphi'(t) \pi_c(t) dt\Bigr)^2 \geq  
\Bigl(\Bigl(\sigma_f^2(\theta_0) -\|I(\theta_0)^{-1}\| \|f'(\theta_0)\| \omega_{f'}(\theta_0, \delta)\Bigr)\vee 0\Bigr)^2
\\
&
=\sigma_f^4(\theta_0) \Bigl(\Bigl(1 -\frac{\|I(\theta_0)^{-1}\| \|f'(\theta_0)\| \omega_{f'}(\theta_0, \delta)\Bigr)}{\sigma_f^2(\theta_0)}\Bigr)\vee 0\Bigr)^2
\geq \sigma_f^4(\theta_0)\Bigl(1- \frac{2\|I(\theta_0)^{-1}\| \|f'(\theta_0)\| \omega_{f'}(\theta_0, \delta)}{\sigma_f^2(\theta_0)}\Bigr).
\end{align*}
Also, since $\|\theta_t-\theta_0\|\leq \delta,$ we have
\begin{align*}
&
\int_{-c}^c J(t) \pi_c(t) dt = \int_{-c}^c \langle  h,I(\theta_t) h\rangle \pi_c(t) dt = \langle h, I(\theta_0) h\rangle +\int_{-c}^c \langle  h,(I(\theta_t)-I(\theta_0)) h\rangle \pi_c(t) dt
\\
&
\leq \langle I(\theta_0)^{-1}f'(\theta_0), f'(\theta_0)\rangle + \omega_{I}(\theta_0, \delta)\|h\|^2\leq \sigma_f^2(\theta_0) + \|I(\theta_0)^{-1}\|^2\|f'(\theta_0)\|^2
\omega_{I}(\theta_0, \delta)
\\
&
= \sigma_f^2 (\theta_0)\Bigl(1 + \frac{\|I(\theta_0)^{-1}\|^2\|f'(\theta_0)\|^2}{\sigma_f^2(\theta_0)} \omega_{I}(\theta_0, \delta)\Bigr).
\end{align*}
It then follows from \eqref{van_trees_bound} that 
\begin{align}
\label{van_trees_3}
\inf_{T_n} \sup_{t\in [-c,c]} \frac{n{\mathbb E}_t (T_n(X_1,\dots, X_n)-\varphi (t))^2}{\sigma_f^2(\theta_0)}
\geq  
\frac{1- \frac{2\|I(\theta_0)^{-1}\| \|f'(\theta_0)\| \omega_{f'}(\theta_0, \delta)}{\sigma_f^2(\theta_0)}}
{1 + \frac{\|I(\theta_0)^{-1}\|^2\|f'(\theta_0)\|^2}{\sigma_f^2(\theta_0)} \omega_{I}(\theta_0, \delta)+ \frac{J_{\pi}}{c^2 \sigma_f^2(\theta_0)}}.
\end{align}

We will use the following simple lemma.

\begin{lemma}
For all $u\in E^{\ast},$ 
\begin{align*}
\langle I(\theta_0)^{-1} u,u\rangle \geq \|I(\theta_0)\|^{-1} \|u\|^2.
\end{align*}
\end{lemma}

\begin{proof}
Since $I(\theta_0)$ is a covariance operator, there exists a Euclidean space ${\mathbb H},$ ${\rm dim}({\mathbb H})={\rm dim}(E),$
and a linear operator $A:E\mapsto {\mathbb H}$ such that $I(\theta_0)=A^{\ast} A.$ Since $I(\theta_0)$ is invertible, it is easy to check 
that operator $A$ is also invertible. This implies that its adjoint operator $A^{\ast}:{\mathbb H}\mapsto E^{\ast}$ is invertible and $(A^{\ast})^{-1}=(A^{-1})^{\ast}.$ 
Moreover, we have $I(\theta_0)^{-1}=A^{-1} (A^{-1})^{\ast}.$ It is also easy to see that $\|I(\theta_0)\|= \|A\|^2=\|A^{\ast}\|^2.$
Therefore, for all $u\in E^{\ast},$
\begin{align*}
\|u\|^2 &= \|A^{\ast} (A^{\ast})^{-1} u\|^2 \leq \|A^{\ast}\|^2 \|(A^{\ast})^{-1} u\|^2 = \|I(\theta_0)\| \langle (A^{\ast})^{-1}u, (A^{\ast})^{-1}u\rangle
\\
&
= \|I(\theta_0)\| \langle A^{-1} (A^{-1})^{\ast}u,u\rangle = \|I(\theta_0)\| \langle I(\theta_0)^{-1} u,u\rangle,
\end{align*}
implying the claim.
\qed
\end{proof}

It follows from the lemma that $\sigma_f^2(\theta_0)\geq \|I(\theta_0)\|^{-1}\|f'(\theta_0)\|^2.$ We plug in this lower bound in the right hand side of inequality \eqref{van_trees_3} to get 
\begin{align}
\label{van_trees_4}
\inf_{T_n} \sup_{t\in [-c,c]} \frac{n{\mathbb E}_t (T_n(X_1,\dots, X_n)-\varphi (t))^2}{\sigma_f^2(\theta_0)}
\geq  
\frac{1- 2\|I(\theta_0)\| \|I(\theta_0)^{-1}\| \frac{\omega_{f'}(\theta_0, \delta)}{\|f'(\theta_0)\|}}
{1 + \|I(\theta_0)\|\|I(\theta_0)^{-1}\|^2 \omega_{I}(\theta_0, \delta)+ \frac{J_{\pi}\|I(\theta_0)\|}{c^2 \|f'(\theta_0)\|^2}}.
\end{align}
Recall also that, by \eqref{cond_c_delta}, 
\begin{align*}
c^2 \|f'(\theta_0)\|^2= \frac{n\delta^2}{\|I(\theta_0)^{-1}\|^2}.
\end{align*}
Thus, we can rewrite \eqref{van_trees_4} as follows:
\begin{align}
\label{van_trees_5}
\inf_{T_n} \sup_{t\in [-c,c]} \frac{n{\mathbb E}_t (T_n(X_1,\dots, X_n)-\varphi (t))^2}{\sigma_f^2(\theta_0)}
\geq  
\frac{1- 2\|I(\theta_0)\| \|I(\theta_0)^{-1}\| \frac{\omega_{f'}(\theta_0, \delta)}{\|f'(\theta_0)\|}}
{1 + \|I(\theta_0)\|\|I(\theta_0)^{-1}\|^2 \omega_{I}(\theta_0, \delta)+ \frac{J_{\pi}\|I(\theta_0)\|\|I(\theta_0)^{-1}\|^2}{n\delta^2}},
\end{align}
which easily implies that
\begin{align}
\label{van_trees_6}
&
\nonumber
\inf_{T_n} \sup_{t\in [-c,c]} \frac{n{\mathbb E}_t (T_n(X_1,\dots, X_n)-\varphi (t))^2}{\sigma_f^2(\theta_0)}
\\
&
\geq  
1- D' \|I(\theta_0)\|\|I(\theta_0)^{-1}\| \Bigl( \frac{\omega_{f'}(\theta_0, \delta)}{\|f'(\theta_0)\|}+ \|I(\theta_0)^{-1}\| \omega_{I}(\theta_0, \delta)
+\frac{\|I(\theta_0)^{-1}\|}{n\delta^2}\Bigr)
\end{align}
with a constant $D'\geq 2.$ The last bound is trivial when the numerator of the right hand side of \eqref{van_trees_5} is negative. 
Otherwise, the bound is a consequence of the following simple inequality: for all $A\in [0,1]$ and $B\geq 0,$ $\frac{1-A}{1+B} \geq 1-A-2B.$
Since $\|\theta_t-\theta_0\|\leq \delta$ for all $t\in [-c,c],$ it immediately follows from \eqref{van_trees_6} that 
\begin{align}
\label{van_trees_7}
&
\nonumber
\inf_{T_n} \sup_{\|\theta-\theta_0\|\leq \delta} \frac{n{\mathbb E}_{\theta}(T_n(X_1,\dots, X_n)-f(\theta))^2}{\sigma_f^2(\theta_0)}
\\
&
\geq  
1- D' \|I(\theta_0)\|\|I(\theta_0)^{-1}\| \Bigl( \frac{\omega_{f'}(\theta_0, \delta)}{\|f'(\theta_0)\|}+ \|I(\theta_0)^{-1}\| \omega_{I}(\theta_0, \delta)
+\frac{\|I(\theta_0)^{-1}\|}{n\delta^2}\Bigr).
\end{align}

It remains to replace in bound \eqref{van_trees_7} $\sigma_f^2(\theta_0)$ with $\sigma_f^2(\theta).$ To this end, we 
will prove the following lemma.

\begin{lemma}
\label{ratio_bd}
Suppose that $\|I(\theta_0)^{-1}\|\omega_I(\theta_0,\delta)\leq 1/2$ and $\frac{\omega_{f'}(\theta_0,\delta)}{\|f'(\theta_0)\|}\leq 1.$
Then 
\begin{align*}
\sup_{\|\theta-\theta_0\|\leq \delta}
\Bigl|\frac{\sigma_f^2(\theta)}{\sigma_f^2(\theta_0)}-1\Bigr|
\leq 
2\|I(\theta_0)\|\|I(\theta_0)^{-1}\|^2 \omega_I(\theta_0,\delta) + 6 \|I(\theta_0)\|\|I(\theta_0)^{-1}\| \frac{\omega_{f'}(\theta_0,\delta)}{\|f'(\theta_0)\|}.
\end{align*}
\end{lemma} 

\begin{proof}
Let $B:= I(\theta)-I(\theta_0).$ Then 
$
I(\theta)^{-1} - I(\theta_0)^{-1} = I(\theta_0)^{-1}((I+ B I(\theta_0)^{-1})^{-1}-I),
$
where $I=I_{E^{\ast}}$ is the identity operator in $E^{\ast}.$ Under the assumptions that $\|\theta-\theta_0\|\leq \delta$ and $\|I(\theta_0)^{-1}\|\omega_I(\theta_0,\delta)\leq 1/2,$
we have $\|B I(\theta_0)^{-1}\|\leq 1/2$ and $(I+ B I(\theta_0)^{-1})^{-1}= I-B I(\theta_0)^{-1}+(B I(\theta_0)^{-1})^2-\dots .$ This easily implies 
that, for $\|\theta-\theta_0\|\leq \delta,$ 
\begin{align*}
\|I(\theta)^{-1} - I(\theta_0)^{-1}\| \leq  \|I(\theta_0)^{-1}\| \frac{\|B I(\theta_0)^{-1}\|}{1-\|B I(\theta_0)^{-1}\|}\leq 2\|I(\theta_0)^{-1}\|^2\|I(\theta)-I(\theta_0)\|
\leq 2\|I(\theta_0)^{-1}\|^2 \omega_{I}(\theta_0,\delta).
\end{align*}
We also have 
\begin{align*}
\sigma_f^2(\theta) - \sigma_f^2 (\theta_0) = \langle (I(\theta)^{-1} - I(\theta_0)^{-1})f'(\theta), f'(\theta)\rangle
+ \langle I(\theta_0)^{-1} f'(\theta), f'(\theta)\rangle -  \langle I(\theta_0)^{-1} f'(\theta_0), f'(\theta_0)\rangle.
\end{align*}
For $\|\theta-\theta_0\|\leq \delta,$ the first term in the right hand side of this equality is bounded as follows:
\begin{align*}
&
|\langle (I(\theta)^{-1} - I(\theta_0)^{-1})f'(\theta), f'(\theta)\rangle| \leq 2 \|f'(\theta)\|^2 \|I(\theta_0)^{-1}\|^2 \omega_{I}(\theta_0,\delta)
\\
&
\leq 2\|I(\theta_0)^{-1}\|^2 \omega_{I}(\theta_0,\delta)\Bigl( \|f'(\theta_0)\|^2  + \|f'(\theta)-f'(\theta_0)\|(2\|f'(\theta_0)\|+ \|f'(\theta)-f'(\theta_0)\|)\Bigr)
\\
&
\leq 2\|I(\theta_0)^{-1}\|^2 \omega_{I}(\theta_0,\delta)\Bigl( \|f'(\theta_0)\|^2  + 2\|f'(\theta_0)\|\omega_{f'}(\theta_0,\delta)+ \omega_{f'}^2(\theta_0,\delta)\Bigr)
\\
&
= 2\|f'(\theta_0)\|^2 \|I(\theta_0)^{-1}\|^2 \omega_{I}(\theta_0,\delta)\Bigl(1  + 2\frac{\omega_{f'}(\theta_0,\delta)}{\|f'(\theta_0)\|}+ 
\frac{\omega_{f'}^2(\theta_0,\delta)}{\|f'(\theta_0)\|^2}\Bigr).
\end{align*}
For the second term, the following bound holds:
\begin{align*}
&
|\langle I(\theta_0)^{-1} f'(\theta), f'(\theta)\rangle -  \langle I(\theta_0)^{-1} f'(\theta_0), f'(\theta_0)\rangle|
\\
&
\leq \|I(\theta_0)^{-1}\|(2 \|f'(\theta_0)\| \|f'(\theta)-f'(\theta_0)\|+ \|f'(\theta)-f'(\theta_0)\|)
\\
&
\leq  \|I(\theta_0)^{-1}\|(2 \|f'(\theta_0)\|\omega_{f'}(\theta_0,\delta)+ \omega_{f'}^2(\theta_0,\delta))
\\
&
= \|f'(\theta_0)\|^2 \|I(\theta_0)^{-1}\| \Bigl(2 \frac{\omega_{f'}(\theta_0,\delta)}{\|f'(\theta_0)\|}+ \frac{\omega_{f'}^2(\theta_0,\delta)}{\|f'(\theta_0)\|^2}\Bigr).
\end{align*}
Using these two bounds and the inequality $\sigma_f^2(\theta_0)\geq \|I(\theta_0)\|^{-1}\|f'(\theta_0)\|^2,$ we get, for $\|\theta-\theta_0\|\leq \delta,$ that 
\begin{align*}
&
\Bigl|\frac{\sigma_f^2(\theta)}{\sigma_f^2(\theta_0)}-1\Bigr| 
\\
&
\leq 
2\frac{\|f'(\theta_0)\|^2}{\sigma_f^2(\theta_0)} \|I(\theta_0)^{-1}\|^2 \omega_{I}(\theta_0,\delta)\Bigl(1  + 2\frac{\omega_{f'}(\theta_0,\delta)}{\|f'(\theta_0)\|}+ 
\frac{\omega_{f'}^2(\theta_0,\delta)}{\|f'(\theta_0)\|^2}\Bigr)
\\
&
+ \frac{\|f'(\theta_0)\|^2}{\sigma_f^2(\theta_0)} \|I(\theta_0)^{-1}\| \Bigl(2 \frac{\omega_{f'}(\theta_0,\delta)}{\|f'(\theta_0)\|}+ \frac{\omega_{f'}^2(\theta_0,\delta)}{\|f'(\theta_0)\|^2}\Bigr)
\\
&
\leq 
2\|I(\theta_0)\|\|I(\theta_0)^{-1}\|^2 \omega_{I}(\theta_0,\delta)\Bigl(1  + 2\frac{\omega_{f'}(\theta_0,\delta)}{\|f'(\theta_0)\|}+ 
\frac{\omega_{f'}^2(\theta_0,\delta)}{\|f'(\theta_0)\|^2}\Bigr)
\\
&
+ \|I(\theta_0)\| \|I(\theta_0)^{-1}\| \Bigl(2 \frac{\omega_{f'}(\theta_0,\delta)}{\|f'(\theta_0)\|}+ \frac{\omega_{f'}^2(\theta_0,\delta)}{\|f'(\theta_0)\|^2}\Bigr).
\end{align*}
Since $\|I(\theta_0)^{-1}\|\omega_I(\theta_0,\delta)\leq 1/2$ and $\frac{\omega_{f'}(\theta_0,\delta)}{\|f'(\theta_0)\|}\leq 1,$
we can simplify the bound to get
\begin{align*}
&
\Bigl|\frac{\sigma_f^2(\theta)}{\sigma_f^2(\theta_0)}-1\Bigr| 
\leq 
2\|I(\theta_0)\|\|I(\theta_0)^{-1}\|^2 \omega_I(\theta_0,\delta)
+6\|I(\theta_0)\| \|I(\theta_0)^{-1}\| \frac{\omega_{f'}(\theta_0,\delta)}{\|f'(\theta_0)\|}.
\end{align*}

\qed
\end{proof}

To complete the proof, note that, by \eqref{van_trees_7}  for all estimators $T_n(X_1,\dots, X_n),$ 
\begin{align*}
&
\sup_{\|\theta-\theta_0\|\leq \delta}
\frac{\sigma_f^2(\theta)}{\sigma_f^2(\theta_0)}
\sup_{\|\theta-\theta_0\|\leq \delta} \frac{n{\mathbb E}_{\theta}(T_n(X_1,\dots, X_n)-f(\theta))^2}{\sigma_f^2(\theta)}
\geq \sup_{\|\theta-\theta_0\|\leq \delta} \frac{n{\mathbb E}_{\theta}(T_n(X_1,\dots, X_n)-f(\theta))^2}{\sigma_f^2(\theta_0)}
\\
&
\geq 
1- D' \|I(\theta_0)\|\|I(\theta_0)^{-1}\| \Bigl( \frac{\omega_{f'}(\theta_0, \delta)}{\|f'(\theta_0)\|}+ \|I(\theta_0)^{-1}\| \omega_{I}(\theta_0, \delta)
+\frac{\|I(\theta_0)^{-1}\|}{n\delta^2}\Bigr).
\end{align*}
Since, by Lemma \ref{ratio_bd}, 
\begin{align*}
\sup_{\|\theta-\theta_0\|\leq \delta}
\frac{\sigma_f^2(\theta)}{\sigma_f^2(\theta_0)}\leq 1+2\|I(\theta_0)\|\|I(\theta_0)^{-1}\|^2 \omega_I(\theta_0,\delta) + 6 \|I(\theta_0)\|\|I(\theta_0)^{-1}\| \frac{\omega_{f'}(\theta_0,\delta)}{\|f'(\theta_0)\|},
\end{align*}
it follows from the last two bounds that 
\begin{align*}
&
\inf_{T_n}\sup_{\|\theta-\theta_0\|\leq \delta} \frac{n{\mathbb E}_{\theta}(T_n(X_1,\dots, X_n)-f(\theta))^2}{\sigma_f^2(\theta)}
\\
&
\geq 
\frac{1- D' \|I(\theta_0)\|\|I(\theta_0)^{-1}\| \Bigl( \frac{\omega_{f'}(\theta_0, \delta)}{\|f'(\theta_0)\|}+ \|I(\theta_0)^{-1}\| \omega_{I}(\theta_0, \delta)
+\frac{\|I(\theta_0)^{-1}\|}{n\delta^2}\Bigr)}
{1+2\|I(\theta_0)\|\|I(\theta_0)^{-1}\|\Bigl(6\frac{\omega_{f'}(\theta_0,\delta)}{\|f'(\theta_0)\|}+\|I(\theta_0)^{-1}\|\omega_I(\theta_0,\delta)\Bigr)},
\end{align*}
which easily implies the claim of the theorem.

\qed
\end{proof}

Next we prove the minimax lower bound of Proposition \ref{max_min_max_many}.

\begin{proof}
It is enough to prove a lower bound of the order $\frac{\rho^2}{n}$ and also a lower bound of the order $(\rho^2\frac{d}{n})^s$ in the case 
when $\rho\sqrt{\frac{d}{n}}\leq 1$ and of the order $\asymp 1$ when $\rho\sqrt{\frac{d}{n}}> 1.$  

The proof of the first bound is rather simple. It is based on the two hypotheses method (see \cite{Tsybakov}, Section 2.3) and it is enough to consider functionals 
of the form $f(\theta):= \langle \theta^{(1)}-\theta_0^{(1)}, u^{(1)}\rangle \phi(\theta^{(1)}-\theta_0^{(1)})$ (for $s\geq 1$), or $f(\theta):= |\langle \theta^{(1)}-\theta_0^{(1)}, u^{(1)}\rangle|^s \phi(\theta^{(1)}-\theta_0^{(1)})$ (for $s<1$). Here $\phi $ is a $C^{\infty}$ function on ${\mathbb R}^{l_1}$ supported in the ball of radius $2$
with center $0$ and equal to $1$ in the ball of radius $1$ with the same center; 
$u^{(1)}\in {\mathbb R}^{l_1}$ is a unit vector. The two hypotheses in question are $\theta= \theta_0$ and 
$\theta= (\theta_0^{(1)}+ \rho n^{-1/2} u^{(1)},\theta_0^{(2)}, \dots, \theta_0^{(d)}),$ so that $f(\theta)=\rho n^{-1/2}$ and $f(\theta_0)=0.$ 
We skip further details of this part of the proof. 

The proof of the second bound is more involved. 
The first step is to construct a set of ``well separated" vectors $\theta_{\omega}\in \Theta$ parametrized by points $\omega$ of binary 
cube $\{-1,1\}^d.$ Let $v^{(j)}\in {\mathbb R}^{l_j}, j=1,\dots, d$ be fixed unit vectors. For $\omega \in \{-1,1\}^d,$
define
\begin{align*}
\theta_{\omega}:=\Bigl(\theta_0^{(1)}+ \rho n^{-1/2} \omega^{(1)} v^{(1)}, \dots, \theta_0^{(d)}+ \rho n^{-1/2} \omega^{(d)}v^{(d)}\Bigr). 
\end{align*}
By the assumptions on $\theta_0$ and $\rho,$ 
$\theta_{\omega} \in \Theta_n(\theta_0,\rho).$
In addition,
\begin{align*}
\|\theta_{\omega}-\theta_{0}\| = \rho \sqrt{\frac{d}{n}}, \omega \in \{-1,1\}^d
\end{align*}
and also that 
\begin{align}
\label{Euclid-Hamming}
\|\theta_{\omega}-\theta_{\omega'}\|^2 = \frac{4\rho^2}{n} h(\omega, \omega'), \omega, \omega'\in \{-1,1\}^d,
\end{align}
where $h(\omega, \omega'):= \sum_{j=1}^d I(\omega^{(j)}\neq (\omega')^{(j)})$ is the Hamming distance.
Using Varshamov-Gilbert lemma, we can choose a subset $B\subset \{-1,1\}^d$ of cardinality at least $2^{d/8}$
such that 
\begin{align*}
h(\omega, \omega')\geq d/8,\ \omega, \omega'\in B, \omega\neq \omega'.
\end{align*}
This implies that 
\begin{align*}
\frac{\rho}{\sqrt{2}}\sqrt{\frac{d}{n}}\leq \|\theta_{\omega}-\theta_{\omega'}\| \leq 2\rho\sqrt{\frac{d}{n}},\ \omega, \omega'\in B, \omega\neq \omega'. 
\end{align*}

The second step is to obtain a minimax lower bound on mean squared error of estimation of $\theta_{\omega}, \omega \in B.$
To this end, note that 
\begin{align*}
K(P_{\theta_{\omega}}\|P_{\theta_0})\leq C^2\|\theta_{\omega}-\theta_0\|^2, \omega \in \{-1,1\}^d.  
\end{align*}
Therefore,
\begin{align*}
K(P_{\theta_{\omega}}^{\otimes n}\|P_{\theta_0}^{\otimes n})=n K(P_{\theta_{\omega}}\|P_{\theta_0})\leq C^2\rho^2 d \leq \alpha d/8\leq \alpha \log{\rm card}(B)
\end{align*}
for a small enough numerical constant $\alpha$ (provided that $\rho\leq \frac{\alpha}{2\sqrt{2}C}$). It follows from Theorem 2.3 in \cite{Tsybakov} that 
\begin{align}
\label{min_max_omega}
\inf_{\hat \theta} \max_{\omega\in B}{\mathbb E}_{\theta_{\omega}}\|\hat \theta-\theta_{\omega}\|^2 \gtrsim \rho^2 \frac{d}{n}.
\end{align}
where the infimum is taken over all estimators $\hat \theta\in \Theta_B:= \{\theta_{\omega}:\omega\in B\},$ based on i.i.d. observations $P_{\theta_{\omega}}, \omega\in B.$

In the third step, we construct ``least favorable" smooth functionals $f_k, k=1,\dots, d$ whose values 
on parameters $\theta_{\omega}$ are hard to estimate and for which the error rate $(\rho^2 \frac{d}{n})^s$ is attained.
The construction is based on an approach suggested by Nemirovski (see \cite{Nemirovski_1990, Nemirovski}). Let $\varphi:{\mathbb R}\mapsto {\mathbb R}$ be a $C^{\infty}$ function 
supported in the interval $[-1,1]$ with $\varphi(0)>0$ and let $\phi(x):= \varphi (\|x\|^2), x\in {\mathbb R}^l,$ $l=l_1+\dots +l_d.$  
Let $\eps:= \rho \sqrt{\frac{d}{n}}\wedge 1.$ 
Define 
\begin{align*}
f_k(\theta):= \sum_{\omega\in B} \omega_k \eps^s \phi\Bigl(\frac{\theta-\theta_{\omega}}{c\eps}\Bigr), \theta\in {\mathbb R}^l, k=1,\dots, d.
\end{align*}
Recall that $\|\theta_{\omega}-\theta_{\omega'}\|\geq \frac{\rho}{\sqrt{2}}\sqrt{\frac{d}{n}} \geq \frac{\eps}{\sqrt{2}}, \omega, \omega'\in B, \omega\neq \omega'.$
If constant $c>0$ is small enough (say, $c<0.1$), then each of the functionals $f_k$ is a sum of ``bumps" $\pm \eps^s \phi\Bigl(\frac{\theta-\theta_{\omega}}{c\eps}\Bigr)$
supported in disjoint balls with centers $\theta_{\omega}, \omega\in B$ and radius $c \eps,$ and, moreover, the distance between any two balls 
is at least $\eps/2.$ This easily implies that $\|f_k\|_{C^s}\lesssim 1.$ Moreover, by choosing $\varphi(0)$ small enough, one can guarantee that 
$\|f_k\|_{C^s}\leq 1.$ Most importantly, 
\begin{align*}
f_k(\theta_{\omega})= \omega_k \eps^s \varphi(0), \omega\in B, k=1,\dots, d.
\end{align*}
This means that the values 
of functionals $f_k, k=1,\dots, d$ on a vector $\theta_{\omega}, \omega\in B$ could be used to recover the binary coding $\omega$ of this vector. 
  
In the last, fourth step of the proof, we show how to use estimators $\hat T_j$ of $f_j(\theta_{\omega}), \omega \in B$ based on i.i.d. observations $X_1,\dots, X_n\sim P_{\theta_{\omega}}$ to construct an estimator $\hat \theta$ of $\theta_{\omega}.$ Comparing the upper bound on the mean squared error 
of this estimator with the minimax lower bound \eqref{min_max_omega} will allow us to complete the proof. Let $\delta>0$ and let $\hat T_j$ be estimators 
of $f_j(\theta_{\omega})$ such that, for all $j=1,\dots, d,$ 
\begin{align}
\label{upper_hat_T_j}
 \max_{\omega \in B} {\mathbb E}_{\theta_{\omega}} (\hat T_j(X_1,\dots, X_n)-f_j(\theta_{\omega}))^2 \leq \delta^2. 
\end{align}  
Let $\tilde \omega_j=\tilde \omega_j (X_1,\dots, X_n):= {\rm sign}(\hat T_j(X_1,\dots, X_n))$ and let $\tilde T_j(X_1,\dots, X_n)= \tilde \omega_j(X_1,\dots, X_n) \eps^s \varphi(0).$
Since $f_j(\theta_{\omega})= \omega_j \eps^s \varphi(0),$ it is easy to check that 
\begin{align*}
\eps^s \varphi (0) |\tilde \omega_j(X_1,\dots, X_n)-\omega_j| = |\tilde T_j(X_1,\dots, X_n)-f_j(\theta_{\omega})| \leq 2 |\hat T_j(X_1,\dots, X_n)- f_j(\theta_{\omega})|.
\end{align*}
This implies that 
\begin{align*}
\eps^{2s} \varphi^2(0) h(\tilde \omega, \omega) \leq \sum_{j=1}^d \Bigl(\hat T_j(X_1,\dots, X_n)- f_j(\theta_{\omega})\Bigr)^2.
\end{align*}
Let now $\hat \omega\in {\rm Argmin}_{\omega\in B} h(\tilde \omega,\omega)$ and define $\hat \theta:= \theta_{\hat \omega}.$
Clearly, we have for all $\omega\in B,$ 
\begin{align*}
h(\hat \omega, \omega)\leq h(\tilde \omega,\omega)+ h(\tilde \omega, \hat \omega)\leq 2h(\tilde \omega,\omega).
\end{align*}
Using \eqref{Euclid-Hamming}, we get
\begin{align*}
\|\hat \theta-\theta_{\omega}\|^2 &=\frac{4\rho^2}{n} h(\hat \omega, \omega) \leq  \frac{8\rho^2}{n} h(\tilde \omega, \omega)\leq \frac{8\rho^2}{n} \frac{1}{\eps^{2s} \varphi^2(0)}
\sum_{j=1}^d \Bigl(\hat T_j(X_1,\dots, X_n)- f_j(\theta_{\omega})\Bigr)^2
\\
&
=\frac{8}{\eps^{2s-2} \varphi^2(0)}
\frac{1}{d}\sum_{j=1}^d \Bigl(\hat T_j(X_1,\dots, X_n)- f_j(\theta_{\omega})\Bigr)^2.
\end{align*}
Recalling bound \eqref{upper_hat_T_j}, we can conclude that 
\begin{align*}
\max_{\omega\in B}{\mathbb E}_{\theta_{\omega}}\|\hat \theta-\theta_{\omega}\|^2 \leq 
\frac{4 \delta^2}{\eps^{2s-2} \varphi^2(0)}.
\end{align*}
Comparing the last bound with the minimax bound \eqref{min_max_omega}, we get that 
$\eps^2 \lesssim \frac{4 \delta^2}{\eps^{2s-2} \varphi^2(0)},$
or $\delta^2\gtrsim \eps^{2s}.$ If $\rho\sqrt{\frac{d}{n}}\leq 1,$ this yields the bound
$
\delta^2 \gtrsim \Bigl(\rho^2\frac{d}{n}\Bigr)^s,
$
otherwise we get $\delta^2 \gtrsim 1.$
This implies that 
\begin{align*}
\sup_{\|f\|_{C^s}\leq 1} \inf_{T_n} \sup_{\theta\in \Theta_n(\theta_0,\rho)} {\mathbb E}_{\theta} (T_n(X_1,\dots, X_n)-f(\theta))^2 \gtrsim  \rho^{2s} \Bigl(\frac{d}{n}\Bigr)^s
\end{align*}
when $\rho\sqrt{\frac{d}{n}}\leq 1,$ and the lower bound is of the order $1$ otherwise.  

Finally, note that, for 
$P_{\theta}=P_{\theta^{(1)}}\times \dots \times P_{\theta^{(d)}},$ the conditions 
$K(P_{\theta^{(j)}}\| P_{\theta_0^{(j)}}) \leq C^2\|\theta^{(j)}-\theta_0^{(j)}\|^2, j=1,\dots, d$ imply that 
\begin{align*}
K(P_{\theta}\|P_{\theta_0})= \sum_{j=1}^d K(P_{\theta^{(j)}}\|P_{\theta_0^{(j)}}) \leq C^2 \sum_{j=1}^d \|\theta^{(j)}-\theta_0^{(j)}\|^2=C^2 \|\theta-\theta_0\|^2,
\end{align*}
and the last claim follows. 

\qed
\end{proof}

Finally, we provide the proof of Proposition \ref{prop_low_exp} of Section \ref{exponential_section}. 

\begin{proof} First note that it will be enough to prove the minimax lower bound with the supremum over the ball $B=B_{\ell_{\infty}}(t_0,\delta)$
instead of set $G$ itself. 
Note that, for all $\theta, \theta'\in {\rm Int}\Theta,$
\begin{align*}
K(P_{\theta}\|P_{\theta'})= \psi(\theta')-\psi(\theta)-\langle\Psi(\theta), \theta'-\theta\rangle
\leq \langle \Psi(\theta')-\Psi(\theta), \theta'-\theta\rangle. 
\end{align*}
Therefore, for the model $\{P_{\Psi^{-1}(t)}: t\in B\},$ we have 
\begin{align*}
K(P_{\Psi^{-1}(t)}\|P_{\Psi^{-1}(t')}) \leq \langle t'-t, \Psi^{-1}(t')-\Psi^{-1}(t)\rangle
\leq \|\Psi^{-1}\|_{{\rm Lip}(B)}\|t'-t\|^2, \ t',t\in B. 
\end{align*}
Observe now that $(\Psi^{-1})'(t)= (\Psi'(\Psi^{-1}(t)))^{-1}= \Sigma_{\Psi^{-1}(t)}^{-1}.$
Hence, 
\begin{align*}
\|\Psi^{-1}\|_{{\rm Lip}(B)}=\sup_{t\in B}\|\Sigma_{\Psi^{-1}(t)}^{-1}\| = \sup_{\theta\in \Psi^{-1}(B)} 
\|\Sigma_{\theta}^{-1}\| \leq  \frac{1}{\sigma_{\rm min}^2(W)},
\end{align*}
and we have 
\begin{align*}
K(P_{\Psi^{-1}(t)}\|P_{\Psi^{-1}(t')}) \leq  \frac{1}{\sigma_{\rm min}^2(W)}\|t'-t\|^2, t',t\in B.
\end{align*}
Thus, we can use the bound of Proposition \ref{sup_min_max_bd} with $C^2= \frac{1}{\sigma_{\rm min}^2(W)}$ and $\rho^2= \gamma^2 \sigma_{\rm min}^2(W)$
to complete the proof. 
\qed
\end{proof}

\end{document}